\documentclass[11pt]{artformath}
\usetikzlibrary{arrows.meta,calc}
\begin{document}
\title{Holonomy Asymptotics along Quartic Differential Rays}
\author{Weihan Ma \thanks{Chern Institute of Mathematics and LPMC, Nankai University, Tianjin 300071, China
\\ \email{1120210009@mail.nankai.edu.cn}}}    
\date{}    
\maketitle 
\vspace{-2.0em}    
\begin{abstract}
Let \(X\) be a closed Riemann surface and let
\(q\in H^0(X,K^4)\) be a nonzero holomorphic quartic differential on \(X\). For \(t >0\), the ray \(tq\) determines a family of Hitchin representations in the \(\operatorname{PSp}(4,\mathbb R)\)-Hitchin component. We study, as
\(t\to+\infty\), the asymptotic behavior of their holonomy along
closed curves. We obtain explicit asymptotic formulas for all singular
values and for the absolute values of all eigenvalues of the
holonomy. Their logarithmic growth rates are given by integrating the
local fourth roots of \(q\) along the saddle connections forming the
geodesic representative of the curve with respect to the singular
flat metric \(\lvert q\rvert^{1/2}\). No restriction is imposed on the
orders of the zeros of \(q\).
\end{abstract}
\tableofcontents
\section{Introduction}
Let \(S\) be a closed oriented surface of genus at least two. The
character variety of representations of \(\pi_1(S)\) into a reductive
Lie group \(G\) is a natural and important object of study. The non-abelian Hodge correspondence, developed through the work of
Hitchin \cite{hitchin1987self}, Donaldson
\cite{donaldson1987twisted}, Corlette \cite{corlette1988flat},
Simpson \cite{simpson1988constructing}, García-Prada, Gothen, and
Mundet i Riera \cite{garciaprada2012hitchinkobayashi}, and many
others, relates this character variety to the moduli space of Higgs
bundles. More
precisely, after fixing a complex structure \(X\) on \(S\), it
identifies reductive representations with polystable Higgs bundles.
We briefly recall how the corresponding representation is recovered
from a Higgs bundle. Given such a Higgs bundle \((E,\phi)\), there exists a Hermitian metric
\(h\) satisfying the Hitchin equation
\[
F_{\nabla_h}+[\phi,\phi^{*_{h}}]=0,
\]
where \(\nabla_h\) is the Chern connection of \(h\). We call such a metric \(h\) is harmonic. The associated
connection
\(
D_h=\nabla_h+\phi+\phi^{*_{h}}
\)
is flat, and its holonomy recovers the corresponding representation. For extensions of the non-abelian Hodge correspondence to noncompact
curves and varieties, see
\cite{simpson1990harmonic,biquard1997higgs,
biquard2004wild,mochizuki2006kobayashi,
mochizuki2009kobayashi,biquard2020parabolic}. 

When \(G\) is a split real form of a complex simple Lie group, Hitchin \cite{hitchin1992lie}
used this correspondence to discover a distinguished connected
component of the character variety, now called the Hitchin component;
its elements are called Hitchin representations. For a fixed
complex structure on \(S\), the Hitchin section parametrizes the
corresponding component of the moduli space of Higgs bundles by holomorphic differentials, and hence also parametrizes the Hitchin component through the non-abelian Hodge correspondence. 

Labourie \cite{labourie2004} introduced dynamical methods to study
Hitchin representations and proved that they are Anosov. In
particular, Hitchin representations are discrete and faithful. They
may therefore be viewed as natural higher-rank generalizations of
Fuchsian representations, and the Hitchin component provides one of
the fundamental examples of a higher Teichmüller space.

For split real groups of rank two, Labourie proved the uniqueness of
the equivariant minimal surface associated with a Hitchin
representation. Equivalently, every such representation determines a
unique complex structure \(X\) on \(S\) for which the quadratic
differential in the Hitchin parametrization vanishes. Consequently, the Hitchin component is naturally identified with the
total space of a vector bundle over Teichmüller space, whose fiber
over \(X\) is the space of remaining holomorphic
differentials
\cite[Theorem~1.2.1]{labourie2017cyclic}. 

In the case of \(\operatorname{SL}(3,\mathbb R)\), this parametrization was also obtained independently by Loftin \cite{loftin2001affine} and Labourie \cite{labourie2007flat}. The resulting
Labourie--Loftin parametrization identifies the Hitchin component \(\operatorname{Hit}_{\operatorname{SL}(3,\mathbb R)}(S)\) with
the bundle of holomorphic cubic differentials over the Teichmüller
space. In \cite{loftin2026limits}, Loftin, Tamburelli, and Wolf studied representations associated with rays of cubic differentials. They showed that the holonomy admits an explicit asymptotic description in
terms of the cubic differential. Reid \cite{reid2026limits}
subsequently studied the same holonomy asymptotics using methods from
convex projective and Finsler geometry, and extended the corresponding
formulas to more general sequences of representations asymptotic to
cubic-differential rays.

Motivated by the work above, we study the analogous problem for the
\(\operatorname{PSp}(4,\mathbb R)\)-Hitchin component. In this case, every
\(\operatorname{PSp}(4,\mathbb R)\)-Hitchin representation
\(
\rho:\pi_1(S)\longrightarrow\operatorname{PSp}(4,\mathbb R)
\)
determines a unique pair \((X,q)\), where \(X\) is the distinguished
complex structure on \(S\) and
\(q\) is a holomorphic quartic differential on \(X\); see
\cite[Theorem~1.2.1]{labourie2017cyclic}. After choosing a square root \(K^{1/2}\), the pair \((X,q)\)
determines the cyclic \(\operatorname{Sp}(4,\mathbb R)\)-Higgs bundle
in the Hitchin section
\(
\bigl(\mathbb K_{\operatorname{Sp}(4,\mathbb R)},\phi(q)\bigr);
\)
see Definition~\ref{def:cyclic-Higgs-bundles}. This Higgs bundle admits a unique compatible
harmonic metric \(h_q\), and the associated connection
\(
D_q=\nabla_{h_q}+\phi(q)+\phi(q)^{*_{h_q}}
\)
is flat. Its holonomy defines a representation
\(
\widetilde\rho_q:
\pi_1(S)\longrightarrow\operatorname{Sp}(4,\mathbb R)
\)
whose projection to \(\operatorname{PSp}(4,\mathbb R)\) is precisely
the original Hitchin representation \(\rho\).

We now fix the complex structure \(X\) on \(S\) and a nonzero holomorphic quartic differential
\(q\in H^0(X,K^4)\). For \(t>0\), let \(D_t\) denote the flat
connection obtained from \(
\bigl(\mathbb K_{\operatorname{Sp}(4,\mathbb R)},\phi(tq)\bigr)
\), the cyclic \(\operatorname{Sp}(4,\mathbb R)\)-Higgs bundle in the Hitchin section associated with
\((X,tq)\). The aim of this paper is to study the asymptotic behavior of its holonomy \(\operatorname{Hol}_t\) as \(t\to+\infty\). The main result is as follows.

Let \(c_\gamma\) be the geodesic representative of a nontrivial closed curve
\(\gamma\) with respect to the singular flat metric
\(\lvert q\rvert^{1/2}\), and write
\[
c_\gamma=c_\ell*\cdots*c_1
\]
as a concatenation of saddle connections. We choose the base point of \(c_\gamma\) to be any point in the interior of \(c_{\ell}\), hence outside the zero set of \(q\). Throughout,
\(\operatorname{Hol}_t(c_\gamma)\) denotes the holonomy  along the curve \(c_\gamma\) based at this point. Let
\[
\phi_k=\sqrt{-1}^{1-k}q^{1/4},
\qquad
k=1,2,3,4,
\]
denote the local fourth roots of \(q\). For each saddle connection
\(c_i\), let
\[
v_i=(v_{i,1},v_{i,2},v_{i,3},v_{i,4}),
\qquad
v_{i,1}\geq v_{i,2}\geq v_{i,3}\geq v_{i,4},
\]
be the non-increasing rearrangement of
\[
-2\operatorname{Re}\int_{c_i}\phi_1,\qquad
-2\operatorname{Re}\int_{c_i}\phi_2,\qquad
-2\operatorname{Re}\int_{c_i}\phi_3,\qquad
-2\operatorname{Re}\int_{c_i}\phi_4.
\]

Our main theorem expresses the asymptotic spectral data of the holonomy entirely in terms of the vectors \(v_i\), which may be viewed as a quartic-differential analogue of the result of Loftin, Tamburelli, and Wolf \cite[Theorem~A]{loftin2026limits}.

\begin{theorem}[Theorem~\ref{thm:exterior-power-holonomy-asymptotics-norm} and Theorem~\ref{thm:exterior-power-spectral-radius-asymptotics}]
\label{thm:introduction-main-result}
With the notation above, for every nontrivial
\(\gamma\in\pi_1(S)\) and every \(k=1,2,3,4\), we have
\[
\lim_{t\to+\infty}
\frac{1}{t^{1/4}}
\log
\left\|
\bigwedge\nolimits^k
\operatorname{Hol}_t(c_\gamma)
\right\|
=
\sum_{i=1}^{\ell}\sum_{j=1}^{k}v_{i,j},
\]
and
\[
\lim_{t\to+\infty}
\frac{1}{t^{1/4}}
\log
\Lambda\left(
\bigwedge\nolimits^k
\operatorname{Hol}_t(c_\gamma)
\right)
=
\sum_{i=1}^{\ell}\sum_{j=1}^{k}v_{i,j},
\]
where \(\lVert\cdot\rVert\) is any fixed matrix norm and
\(\Lambda\) denotes the spectral radius.
\end{theorem}
As a direct consequence, we obtain the asymptotic behavior of all
the singular values and the absolute values of the eigenvalues of the holonomy.
\begin{corollary}[Corollary~\ref{cor:singular-value-asymptotics} and Corollary~\ref{cor:eigenvalue-modulus-asymptotics}]
\label{cor:introduction-spectral-asymptotics}
For \(M\in\operatorname{Sp}(4,\mathbb R)\), let
\[
\sigma_1(M)\geq \sigma_2(M)\geq
\sigma_3(M)\geq \sigma_4(M)>0
\]
denote its singular values, and let
\[
|\lambda_1(M)|\geq |\lambda_2(M)|\geq
|\lambda_3(M)|\geq |\lambda_4(M)|>0
\]
denote the absolute values of its eigenvalues, counted with multiplicity. Then, for every \(k=1,2,3,4\), we have
\[
\lim_{t\to+\infty}
\frac{1}{t^{1/4}}
\log
\sigma_k\left(
\operatorname{Hol}_t(c_\gamma)
\right)
=
\sum_{i=1}^{\ell}v_{i,k},
\]
and
\[
\lim_{t\to+\infty}
\frac{1}{t^{1/4}}
\log
\left|
\lambda_k\left(
\operatorname{Hol}_t(c_\gamma)
\right)
\right|
=
\sum_{i=1}^{\ell}v_{i,k}.
\]
\end{corollary}
The asymptotic formulas for the absolute values of the eigenvalues
also have a direct application to the study of compactifications of
the \(\operatorname{PSp}(4,\mathbb R)\)-Hitchin component. Recall that the closed positive Weyl chamber of \(\operatorname{PSp}(4,\mathbb R)\) is
\(
\overline{\mathfrak a^+}
=
\left\{
(a,b,-b,-a)\in\mathbb R^4
\mathrel{}\middle|\mathrel{}
a\geq b\geq 0
\right\}.
\)
In \cite{parreau2012compactification}, Parreau introduced a compactification of the Hitchin component based on the marked Weyl-chamber length spectrum. More precisely, each \(\operatorname{PSp}(4,\mathbb R)\)-Hitchin representation \(\rho\)
is associated with its marked Weyl-chamber length spectrum
\[
L_\rho:\pi_1(S)\longrightarrow\overline{\mathfrak a^+},
\qquad
L_\rho(\gamma)=\lambda\bigl(\rho(\gamma)\bigr),
\]
where
\(
\lambda:\operatorname{PSp}(4,\mathbb R)
\longrightarrow\overline{\mathfrak a^+}
\)
is the Jordan projection. More explicitly, let
\(g\in\operatorname{PSp}(4,\mathbb R)\), and let
\(M\in\operatorname{Sp}(4,\mathbb R)\) be any lift of \(g\). If the
absolute values of the eigenvalues of \(M\) are ordered as
\[
|\lambda_1(M)|\geq|\lambda_2(M)|
\geq|\lambda_3(M)|\geq|\lambda_4(M)|>0,
\]
then
\[
\lambda(g)
=
\operatorname{diag}\left(
\log|\lambda_1(M)|,
\log|\lambda_2(M)|,
\log|\lambda_3(M)|,
\log|\lambda_4(M)|
\right).
\]
This definition is independent of the choice of lift.

Let
\(\widehat{\operatorname{Hit}}_{\operatorname{PSp}(4,\mathbb R)}(S)\)
denote the one-point compactification of
\(\operatorname{Hit}_{\operatorname{PSp}(4,\mathbb R)}(S)\).
Following Parreau \cite{parreau2012compactification}, the Weyl-chamber length compactification
\(
\overline{\operatorname{Hit}}_{\operatorname{PSp}(4,\mathbb R)}(S)
^{\mathrm{WL}}
\)
is defined as the closure of the image of the map
\[
\begin{aligned}
j:
\operatorname{Hit}_{\operatorname{PSp}(4,\mathbb R)}(S)
&\longrightarrow
\widehat{\operatorname{Hit}}_{\operatorname{PSp}(4,\mathbb R)}(S)
\times
\mathbb P\left(
\overline{\mathfrak a^+}^{\,\pi_1(S)}
\right),\\
\rho&\longmapsto\bigl(\rho,[L_\rho]\bigr).
\end{aligned}
\]
Our Corollary~\ref{cor:introduction-spectral-asymptotics} gives the following explicit description of the limit point of the quartic-differential ray in \(\overline{\operatorname{Hit}}_{\operatorname{PSp}(4,\mathbb R)}(S)
^{\mathrm{WL}}\).
\begin{corollary}
\label{cor:introduction-weyl-chamber-limit}
Let \(q\) be a nonzero quartic differential on \(X\), and let \(\rho_t\) be the family of \(\operatorname{PSp}(4,\mathbb R)\)-Hitchin representations corresponding to the ray \(tq\). For every nontrivial \(\gamma\in\pi_1(S)\), let 
\[ 
c_\gamma=c_\ell*\cdots*c_1 
\] 
be the decomposition of its geodesic representative with respect to the singular flat metric \(\lvert q\rvert^{1/2}\) into saddle connections.
Define
\[
L_\infty(\gamma)
=
\sum_{i=1}^{\ell}v_i
=
\left(
\sum_{i=1}^{\ell}v_{i,1},
\sum_{i=1}^{\ell}v_{i,2},
\sum_{i=1}^{\ell}v_{i,3},
\sum_{i=1}^{\ell}v_{i,4}
\right).
\]
Then
\[
\lim_{t\to+\infty}
\frac{1}{t^{1/4}}L_{\rho_t}(\gamma)
=
L_\infty(\gamma)
\]
for every nontrivial \(\gamma\in\pi_1(S)\). Consequently,
\[
[L_{\rho_t}]
\longrightarrow
[L_\infty] \quad \text{in} \quad \mathbb P\left(\overline{\mathfrak a^+}^{\,\pi_1(S)}\right),
\]
and the ray \(\rho_t\) converges to the boundary point
\(\bigl(\infty,[L_\infty]\bigr)\) in \(\overline{\operatorname{Hit}}_{\operatorname{PSp}(4,\mathbb R)}(S)
^{\mathrm{WL}}\).
\end{corollary}
It is worth noting that the limiting spectrum is uniformly comparable with the length spectrum of the singular flat metric:
\(
\bigl\|L_\infty(\gamma)\bigr\|_\infty
\asymp
\ell_{|q|^{1/2}}(c_\gamma),
\)
therefore has positive systole, i.e., \(\inf_{\gamma\neq 1}
\bigl\|L_\infty(\gamma)\bigr\|_\infty>0\). By
Burger--Iozzi--Parreau--Pozzetti \cite{burger2021currents}, the point \([L_\infty]\) belongs to the open subset of the Weyl-chamber length boundary on which the mapping class group acts properly discontinuously.

Burger, Iozzi, Parreau, and Pozzetti 
\cite{burger2021real-spectrum,burger2023real-spectrum} introduced a finer compactification of character varieties, called the real-spectrum compactification. There is a natural continuous, \(\operatorname{Out}(\pi_1(S))\)-equivariant, surjective map from the real-spectrum compactification to the Weyl-chamber length compactification \cite[Theorem~6]{burger2021real-spectrum}. Thus, the corollary above
determines the Weyl-chamber image of any real-spectrum limit of the
quartic-differential ray. 

An alternative compactification of the
\(\operatorname{PSp}(4,\mathbb R)\)-Hitchin component was constructed
by Ouyang and Tamburelli \cite{ouyang2023length}. Using the length spectra of the induced metrics on equivariant maximal surfaces, interpreted as geodesic currents, they described the boundary in terms of mixed structures.
Here we use the isomorphism
\(
\operatorname{SO}_0(2,3)\cong\operatorname{PSp}(4,\mathbb R).
\)
For related compactifications associated with other real rank-two Lie
groups, see
\cite{ouyang2021limits,ouyang2023high-energy}.

\subsection*{Outline of the proof and future directions}
We briefly outline the proof. By symplectic duality, it suffices to establish the theorem and its corollary for \(k=1,2\). The proof
naturally divides into the analysis away from the zeros of \(q\) and the analysis near the zeros.

Collier and Li \cite{collier2017asymptotics} studied the asymptotic
behavior away from the zeros of the holomorphic differential for
general cyclic \(\operatorname{SL}(n,\mathbb R)\)-Higgs bundles in
the Hitchin section. For the cyclic Higgs bundles considered here, a
point lies outside \(Z(q)\) if and only if the Higgs field is
regular semisimple at that point. More generally, Mochizuki
\cite{Mochizuki_2016} developed an asymptotic decoupling theory for
generically regular semisimple Higgs bundles, which was subsequently
extended by Sagman and Smillie \cite{sagman2026local} to more general Higgs bundles.
These results provide the asymptotic description needed on
\(X\setminus Z(q)\).

The main difficulty is the asymptotic behavior near the zeros of
\(q\). Since the flat geodesic representative is a concatenation of
saddle connections meeting at these zeros, their contributions to the
holonomy cannot be avoided. At a zero of \(q\), the Higgs field is
no longer semisimple, and hence the results above do not apply.

To analyze the parallel transport near the zeros, we adapt the general
framework developed by Loftin, Tamburelli, and Wolf
\cite{loftin2026limits}. Near a zero of order \(d\), we choose a local
coordinate in which
\(
q=z^ddz^4.
\)
After the natural rescaling associated with the ray \(tq\), the
local problem is reduced to the asymptotic analysis at infinity of the
planar model
\(
\bigl(\mathbb C,z^d \odif{z}^4\bigr).
\) The asymptotic behavior of this model
depends on the sector, giving rise to the Stokes
phenomenon. In particular, when a path crosses several Stokes
directions, one must analyze a product of Stokes matrices. To recover
the global holonomy formula, we show that the product associated with
a zero matches the dominant contributions arising from the two
adjacent saddle connections.

In the \(\operatorname{SL}(3,\mathbb R)\) case, Loftin,
Tamburelli, and Wolf \cite{loftin2026limits} obtained the required nonvanishing property from
the convex geometry of the associated polygon, while convex-projective and Finsler-geometric
structures also play an important role in Reid's approach \cite{reid2026limits}. For the
\(\operatorname{Sp}(4,\mathbb R)\) planar model, however, the
convexity and positivity properties at infinity needed are not established in the work of
Tamburelli and Wolf \cite{tamburelli2024planar}. We therefore take a direct analytic and algebraic approach.
Using the results of Guest, Its and Lin \cite{guestlin2026110730}, we compute the relevant Stokes matrices
explicitly and establish the required properties of their products.

 A further difficulty arises because we seek asymptotic formulas for
all four singular values and all four eigenvalue moduli. The symplectic
duality reduces the problem to \(k=1,2\), but the case \(k=2\)
requires a separate analysis of the induced parallel transport on
\(\bigwedge^2\mathbb R^4\), together with the corresponding products
of Stokes matrices. This issue does not arise separately in the
\(\operatorname{SL}(3,\mathbb R)\) case, since
\(
\bigwedge\nolimits^2\mathbb R^3
\simeq
(\mathbb R^3)^*.
\)
Consequently, Loftin, Tamburelli, and Wolf \cite{loftin2026limits} did not need to carry out
an independent analysis of the second exterior power.

We conclude by mentioning several possible directions for future
work. We conclude by mentioning several possible directions for future
work. First, the results of Guest, Its, and Lin
\cite{guestlin2026110730} apply to the planar models associated with
general cyclic \(\operatorname{SL}(n,\mathbb R)\)-Higgs bundles in
the Hitchin section. Thus, one expects that the methods of this paper
could yield analogous holonomy asymptotics in this more general
setting. The main obstacle is that the explicit analysis of the
Stokes matrices, their products, and their induced actions on
exterior powers becomes increasingly complicated as \(n\) grows.

A related but distinct direction concerns the split real group
\(G_2'\). Evans studied polynomial sextic differentials and the
associated almost-complex curves in
\(\widehat{\mathbb S}^{2,4}\)
\cite{evans2024polynomial}. Since the corresponding Higgs bundles are
sub-cyclic rather than cyclic, the results of Guest, Its, and Lin \cite{loftin2026limits} do not apply directly. It would be interesting to develop an analogous Stokes analysis for these planar models, compute the relevant Stokes matrices, and investigate the resulting holonomy asymptotics.

Second, Tamburelli and Wolf \cite{tamburelli2024planar} identified
polynomial quartic differentials with a connected component of the
moduli space of future-directed negative lightlike polygons in the
Einstein universe, but did not give a complete characterization of
the polygons in this component. The calculations developed here
suggest a possible approach to this problem. One may first establish
the relevant convexity or positivity properties for the polygons
associated with the model differentials
\(z^d\odif{z}^4,\) and then extend these properties to the entire component through a deformation and connectedness argument.

Another natural question concerns the asymptotic behavior of the
associated harmonic maps. In the rank-one case, Wolf related rays of
quadratic differentials to harmonic maps into
\(\mathbb R\)-trees
\cite{wolf1989teichmuller,wolf1995rtrees}. In higher rank, harmonic
maps into buildings and their relation to compactifications have been
studied in \cite{katzarkov2015harmonic,katzarkov2017constructing,
burger2021real-spectrum,parreau2022invariant,martone2024closed,}. Loftin, Tamburelli, and
Wolf \cite{loftin2026limits} proved that, after suitable rescaling
along a cubic-differential ray, the associated harmonic maps converge
to a harmonic map into a Euclidean building. Owing to the length of
the present paper, we do not study the analogous building-valued
limit here.

Finally, the holonomy asymptotics obtained here may provide a first
step toward understanding quartic-differential rays in the
real-spectrum compactification of Burger, Iozzi, Parreau, and Pozzetti
\cite{burger2021real-spectrum,burger2023real-spectrum}. Refining these results to precise
asymptotic expansions of trace functions may lead to the existence
and uniqueness of the corresponding limits in this compactification.

\subsection*{Organization}
The paper is organized as follows. In Section~\ref{sec:preliminaries~and~notations}, we review the theory of \(\operatorname{Sp}(4,\mathbb R)\)-Higgs bundles and fix the conventions used in our parallel transport calculations. We also
introduce the notation associated with quartic differentials and the
corresponding rays in the Hitchin component. In Section~\ref{sec:local-model}, we first recall the results of Tamburelli and Wolf \cite{tamburelli2024planar} on the planar model. We reformulate these results in our own notation. We then use the results of Guest, Its, and Lin \cite{guestlin2026110730} to compute the relevant Stokes matrices explicitly. In Section~\ref{sec:local-comparison-near-zeros}, we show that the asymptotic behavior of parallel
transport near a zero of the quartic differential is governed by the
corresponding planar model. In Section~\ref{sec:asymptotic-holonomy}, we first determine the asymptotic behavior of parallel
transport along two adjacent saddle connections and study the
products of Stokes matrices arising at their common endpoint. We then
assemble these local contributions to obtain the holonomy asymptotics
along the entire closed curve. Finally, we carry out the analogous
analysis on \(\bigwedge^2\operatorname{Hol_t}\) and complete the proofs of the main results.
 \subsection*{Acknowledgment}
  I am deeply grateful to my supervisor, Prof. Qiongling Li, for suggesting the problem and for her stimulating discussions. This work was partially supported by the National Key R\&D Program of China (Grant No. 2022YFA1006600), the Fundamental Research Funds for the Central Universities (Grant No. 63243067), and the Nankai Zhide Foundation.

\section{Preliminaries and notations}\label{sec:preliminaries~and~notations}
This preliminary section collects the main objects and conventions used throughout the paper. We begin by recalling \(\mathrm{Sp}(4,\mathbb{R})\)-Higgs bundles, the Hitchin equation, and the compatibility condition that harmonic metrics must satisfy in this real setting. Next, we set up the formalism of parallel transport, holonomy, and the multiplication rules for transport matrices along concatenated paths.  After that we review how a holomorphic quartic differential on a Riemann surface induces a singular flat metric, and we introduce natural coordinates, saddle connections, and the decomposition of flat geodesics into saddle connections. Finally, we define the one-parameter family of holonomies obtained by scaling the quartic differential, which will be the central object of study. The explicit local Stokes matrices will be introduced later, in Section~\ref{sec:local-model}.

Let \(X\) be a Riemann surface and \(K\) its canonical line bundle.

\subsection{\texorpdfstring{\(\mathrm{Sp}(4,\mathbb{R})\)}{SP}-Higgs bundles}
\label{section SP(4,R)-Higgs bundles}

Since every \(\operatorname{PSp}(4,\mathbb R)\)-Hitchin
representation admits a lift to \(\operatorname{Sp}(4,\mathbb R)\),
we work throughout with \(\operatorname{Sp}(4,\mathbb R)\)-Higgs
bundles. We begin by recalling their definition over \(X\). We then
introduce the Hitchin equations and explain the additional
compatibility condition satisfied by the harmonic metric in this
setting.

\begin{definition}
    An \(\mathrm{Sp}(4,\mathbb{R})\)-Higgs bundle over \(X\) is a pair \((E,\phi)\) consisting of
    \[
    E = V \oplus V^*,\qquad
    \phi = \begin{pmatrix} 0 & \beta \\ \gamma & 0 \end{pmatrix},
    \]
    where
    \begin{itemize}
        \item \(V\) is a holomorphic vector bundle of rank \(2\) over \(X\) and \(V^*\) is its dual,
        \item \(\beta \in \mathrm{H}^0(X, S^2V \otimes K)\) and \(\gamma \in \mathrm{H}^0(X, S^2V^* \otimes K)\).
    \end{itemize}
\end{definition}

In the direct sum decomposition \(E = V \oplus V^*\), the bundle carries a natural symplectic form \(\Omega\) and a natural orthogonal form \(Q\).  In block matrix notation, they are given by
\begin{align}\label{symplectic structure and orthogonal structure}
    \Omega = \begin{pmatrix} 0 & \mathbf{1} \\ -\mathbf{1} & 0 \end{pmatrix},
\qquad
    Q = \begin{pmatrix} 0 & \mathbf{1} \\ \mathbf{1} & 0 \end{pmatrix},
\end{align}
where \(\mathbf{1}\) denotes the identity map on the fibres of \(V\) (equivalently, of \(V^{*}\)).  

One may view \(\beta\) and \(\gamma\) as homomorphisms twisted by \(K\),
\[
\beta \colon V^{*} \longrightarrow V \otimes K,\qquad
\gamma \colon V \longrightarrow V^{*} \otimes K,
\]
and their symmetry is precisely the condition that \(\phi\) is skew-adjoint with respect to \(\Omega\), or equivalently, self-adjoint with respect to \(Q\).

Let us now recall the notion of a harmonic metric.  Given a Hermitian metric \(h\) on \(E\), write \(\nabla_h\) for its Chern connection and \(\phi^{*_h}\) for the adjoint of \(\phi\) with respect to \(h\).  The Hitchin equation for \(h\) is
\begin{align}\label{Hitchin equation}
    F_{\nabla_h} + [\phi, \phi^{*_h}] = 0,
\end{align}
where \(F_{\nabla_h}\) is the curvature of \(\nabla_h\).  This condition is equivalent to requiring that the connection
\[
D_{\phi} = \nabla_h + \phi + \phi^{*_h}
\]
be flat.
\begin{definition}
    A Hermitian metric \(h\) on \(E\) is called harmonic if it satisfies the Hitchin equation \eqref{Hitchin equation}.
\end{definition}

The definition above is the standard one for \(\mathrm{SL}(n,\mathbb{C})\)-Higgs bundles.  In the \(\mathrm{Sp}(4,\mathbb{R})\) case, however, the harmonic metric must additionally be compatible with the real symplectic structure.  We now explain what this compatibility means.

Let \(E\) be a complex vector space equipped with an orthogonal structure \(Q\).  Choose a Hermitian metric \(h\) on \(E\) and denote by \(E^*\) the dual space of \(E\).  Both \(h\) and \(Q\) give natural identifications of \(E\) with \(E^*\):
\begin{align*}
    \Phi_h(v)(u) &= h(u,v), \\
    \Phi_Q(v)(u) &= Q(u,v).
\end{align*}
Here \(\Phi_h\) is anti-\(\mathbb{C}\)-linear, whereas \(\Phi_Q\) is \(\mathbb{C}\)-linear.  The dual Hermitian metric \(h^*\) on \(E^*\) is defined by
\[
h^*(u^*,v^*) = h\bigl(\Phi_h^{-1}(v^*), \Phi_h^{-1}(u^*)\bigr),
\]
and one checks that \(\Phi_h^{-1} = \Phi_{h^*}\).  Composing the two isomorphisms yields an anti-\(\mathbb{C}\)-linear automorphism of \(E\):
\[
\kappa = \Phi_h^{-1} \circ \Phi_Q .
\]
\begin{definition}
    The metric \(h\) is called compatible with \(Q\) if \(\kappa^2 = \id\).
\end{definition}
When this holds, \(\kappa\) is an anti-linear involution, and its fixed point set
\[
E^{\kappa} = \{ v \in E : \kappa(v) = v \}
\]
is a real form of \(E\) (that is, \(E \cong E^{\kappa} \otimes_{\mathbb{R}} \mathbb{C}\)).

Returning to our Higgs bundle, \(E = V \oplus V^*\) for a complex vector space \(V\) of rank \(2\), and it comes with the natural symplectic structure \(\Omega\) and orthogonal structure \(Q\) given by
\begin{align*}
    \Omega\bigl((v,\alpha),(w,\eta)\bigr) &= \eta(v) - \alpha(w), \\
    Q\bigl((v,\alpha),(w,\eta)\bigr) &= \eta(v) + \alpha(w),
\end{align*}
where \(v,w \in V\) and \(\alpha,\eta \in V^*\).
\begin{definition}
    A Hermitian metric \(h\) on \(E\) is said to be compatible with the \(\mathrm{Sp}(4,\mathbb{R})\)-structure if
    \begin{itemize}
        \item \(h\) is compatible with \(Q\);
        \item \(V\) and \(V^*\) are mutually orthogonal with respect to \(h\).
    \end{itemize}
\end{definition}
Assume now that \(h\) is compatible with the \(\mathrm{Sp}(4,\mathbb{R})\)-structure.  Then \(h\) takes the block-diagonal form
\[
h = \begin{pmatrix}
    h_V & 0 \\
    0   & h_V^*
\end{pmatrix},
\]
and consequently
\[
E^{\kappa} = V^{\kappa} \oplus (V^{\kappa})^* .
\]
Thus the restriction \(\Omega|_{E^{\kappa}}\) defines a real symplectic structure on \(E^{\kappa}\).

Let \((E,\phi)\) be an \(\mathrm{Sp}(4,\mathbb{R})\)-Higgs bundle; as explained, it carries natural forms \(\Omega\) and \(Q\) given by \eqref{symplectic structure and orthogonal structure}.  Choose a harmonic metric \(h\) on \(E\) that is fibrewise compatible with the \(\mathrm{Sp}(4,\mathbb{R})\)-structure.  The fixed point set of the involution \(\kappa\) associated with \(h\) is a real subbundle \(E_{\mathbb{R}} \subset E\), and \(\Omega\) restricts to a real symplectic form \(\Omega_{\mathbb{R}}\) on \(E_{\mathbb{R}}\).

One verifies that the flat connection
\[
D_{\phi} = \nabla_h + \phi + \phi^{*_h}
\]
preserves both the real subbundle \(E_{\mathbb{R}}\) and the symplectic form \(\Omega_{\mathbb{R}}\).  Hence it gives a fundamental group representation to \(\mathrm{SP}(4,\mathbb{R})\). For a more detailed discussion we refer the reader to \cite{Li2023generically,li2023higgs}.

\subsection{Parallel transport, frames, and multiplication conventions}
\label{subsec:parallel-transport-holonomy}

Let \((E,\phi)\) be an \(\mathrm{Sp}(4,\mathbb{R})\)-Higgs bundle over \(X\) and let \(h\) be a harmonic metric compatible with the \(\mathrm{Sp}(4,\mathbb{R})\)-structure. The associated flat connection
\[
D_{\phi}=\nabla_h+\phi+\phi^{*_h}
\]
preserves the real symplectic bundle \((E_{\mathbb{R}},\Omega_{\mathbb{R}})\).

Consider an oriented piecewise smooth closed curve \(c\) on \(X\). Fixing a symplectic frame at a base point, parallel transport along \(c\) yields a matrix
\[
\operatorname{Hol}_{\phi}(c)\in\mathrm{Sp}(4,\mathbb{R}),
\]
whose conjugacy class depends only on the free homotopy class of \(c\).Different choices of base point or frame merely conjugate this matrix. 

To carry out explicit computations, it is convenient to lift the whole picture to the universal cover, where a global frame becomes available and parallel transport can be written in matrix form. Let
\[
\pi:\widetilde X\longrightarrow X
\]
be the universal covering map, and set
\[
\widetilde E_{\mathbb{R}}:=\pi^*E_{\mathbb{R}},\qquad
\widetilde D_{\phi}:=\pi^*D_{\phi}
\]
for the pulled-back real symplectic bundle and the pulled-back flat connection. Choose a global real symplectic frame
\[
\mathcal F=(e_1,e_2,e_3,e_4)
\]
of \(\widetilde E_{\mathbb{R}}\). With respect to this frame, the connection form of \(\widetilde D_{\phi}\) is a globally defined \(\mathfrak{sp}(4,\mathbb{R})\)-valued one-form \(A_{\phi}\), determined by
\[
\widetilde D_{\phi}\mathcal F=\mathcal F A_{\phi},
\qquad
A_{\phi}\in\Omega^1(\widetilde X,\mathfrak{sp}(4,\mathbb{R})).
\]

For \(p,p'\in\widetilde X\), we denote by
\[
T^{\phi}_{p,p'}:\widetilde E_{\mathbb{R},p}\longrightarrow\widetilde E_{\mathbb{R},p'}
\]
for the parallel transport of \(\widetilde D_{\phi}\) from \(p\) to \(p'\). Fix a point \(p\in\widetilde X\) and define a matrix-valued function
\[
\Psi_{\phi,p}:\widetilde X\longrightarrow \mathrm{Sp}(4,\mathbb{R})
\]
by the condition
\begin{align}\label{eq:def-Psi}
    T^{\phi}_{p,x}(\mathcal F(p))\Psi_{\phi,p}(x)=\mathcal F(x),\qquad x\in\widetilde X.
\end{align}

Equivalently,
\begin{equation}\label{eq:Psi-connection-equation}
    \begin{dcases}
        \Psi_{\phi,p}(p)=I,  \\
        \Psi_{\phi,p}^{-1}\odif\Psi_{\phi,p}=A_{\phi}.
    \end{dcases}
\end{equation}
Thus, if \(c(s):[a,b]\to\widetilde X\) is a piecewise smooth path with \(c(a)=p\), then
\begin{equation*}
    \begin{dcases}
        \Psi_{\phi,p}(c(a))=I, \\
        \odv{}{s}\Psi_{\phi,p}(c(s))=
\Psi_{\phi,p}(c(s))A_{\phi}(\dot c(s)).
    \end{dcases}
\end{equation*}
We call \(\Psi^{\mathrm M}(\tau)\) the fundamental matrix of \(\widetilde D_{\phi}\), with respect to the
 frame \(\mathcal{F}\) and normalized at the base point
\(p\).

In the global frame \(\mathcal F\), the parallel transport from \(p\) to \(p'\) is represented by the matrix
\[
M_{\mathcal F}^{\phi}(p,p')=\Psi_{\phi,p}(p')^{-1}.
\]
More generally, for arbitrary \(x,y\in\widetilde X\) the matrix of
\[
T^{\phi}_{x,y}:\widetilde E_{\mathbb{R},x}\longrightarrow\widetilde E_{\mathbb{R},y}
\]
with respect to the same global frame \(\mathcal F\) is
\[
M_{\mathcal F}^{\phi}(x,y)=\Psi_{\phi,p}(y)^{-1}\Psi_{\phi,p}(x),
\]
where \(p\in\widetilde X\) is an arbitrary auxiliary base point. The expression does not depend on the choice of \(p\).

We now explain how the holonomy of a closed curve on \(X\) is recovered from these lifted data. Let \(\gamma\in\pi_1(X,x_0)\) be represented by a closed curve \(c\) based at \(x_0\). Choose a lift \(\tilde x_0\in\widetilde X\) of \(x_0\). The curve \(c\) lifts uniquely to a path \(\tilde c\) in \(\widetilde X\) starting at \(\tilde x_0\); its endpoint is \(\tilde x_0\cdot\gamma\), where \(\cdot\) denotes the action of \(\pi_1(X)\) on \(\widetilde X\) by deck transformations. Parallel transport along \(\tilde c\) is described by the matrix \(M_{\mathcal F}^{\phi}(\tilde x_0,\tilde x_0\cdot\gamma)\), which relates the bases \(\mathcal F(\tilde x_0)\) and \(\mathcal F(\tilde x_0\cdot\gamma)\) via
\[
T^{\phi}_{\tilde x_0,\tilde x_0\cdot\gamma}(\mathcal F(\tilde x_0)) =
\mathcal F(\tilde x_0\cdot\gamma)M_{\mathcal F}^{\phi}(\tilde x_0,\tilde x_0\cdot\gamma).
\]

To view the holonomy as an automorphism of the single fibre \(E_{\mathbb{R},x_0}\), we must express both the initial vector and its parallel translate in the same frame. The projection \(\pi\) identifies both \(\widetilde E_{\mathbb{R},\tilde x_0}\) and \(\widetilde E_{\mathbb{R},\tilde x_0\cdot\gamma}\) with \(E_{\mathbb{R},x_0}\), but the two global frames \(\mathcal F(\tilde x_0)\) and \(\mathcal F(\tilde x_0\cdot\gamma)\) may differ. We therefore introduce the transition matrix \(C_\gamma\in\mathrm{Sp}(4,\mathbb{R})\) defined by
\[
\mathcal F(\tilde x_0\cdot\gamma) = \mathcal F(\tilde x_0)C_\gamma .
\]
In the common frame \(\mathcal F(\tilde x_0)\) the parallel transport then reads
\[
T^{\phi}_{\tilde x_0,\tilde x_0\cdot\gamma}(\mathcal F(\tilde x_0)) =
\mathcal F (\tilde x_0)C_\gamma M_{\mathcal F}^{\phi}(\tilde x_0,\tilde x_0\cdot\gamma).
\]
Hence the holonomy of \(c\) with respect to the frame \(\mathcal F(\tilde x_0)\) is
\[
\operatorname{Hol}_{\phi}(c) =
C_\gamma M_{\mathcal F}^{\phi}(\tilde x_0,\tilde x_0\cdot\gamma).
\]
Changing the lift \(\tilde x_0\) or the global frame \(\mathcal F\) conjugates this matrix, so the map
\[
\rho_{\phi}: \pi_1(X) \longrightarrow \mathrm{Sp}(4,\mathbb{R}), \qquad
\gamma \longmapsto C_\gamma M_{\mathcal F}^{\phi}(\tilde x_0,\tilde x_0\cdot\gamma)
\]
is a well-defined representation up to conjugation.

For later use we also record how parallel transport matrices behave under a change of the global frame. Let
\[
\mathcal G=\mathcal F C
\]
be another global symplectic frame of \(\widetilde E_{\mathbb{R}}\), where
\[
C:\widetilde X\longrightarrow \mathrm{Sp}(4,\mathbb{R})
\]
is a smooth map. If
\[
\widetilde D_{\phi}\mathcal G=\mathcal G A_{\phi}^{\mathcal G},
\]
then
\[
A_{\phi}^{\mathcal G}=C^{-1}A_{\phi} C+C^{-1}\odif C.
\]
The corresponding solutions satisfy
\[
\Psi^{\mathcal G}_{\phi,p}(x)=C(p)^{-1}\Psi^{\mathcal F}_{\phi,p}(x)C(x),
\]
and consequently the parallel transport matrices transform by
\[
M_{\mathcal G}^{\phi}(x,y)=C(y)^{-1}M_{\mathcal F}^{\phi}(x,y)C(x).
\]

We now state the multiplication convention for concatenated paths.
Let
\[
c=c_1*c_2*\cdots*c_m,
\]
which means that \(c_1\) is traversed first, then \(c_2\), and so on, with \(c_j\) running from \(x_{j-1}\) to \(x_j\) (see Figure~\ref{fig:concatenated-paths}).

\begin{figure}[htbp]
\centering
\begin{tikzpicture}[scale=1.05,>=Stealth]

% points
\coordinate (x0) at (0,0);
\coordinate (x1) at (2,0.9);
\coordinate (x2) at (4,0.1);
\coordinate (dots) at (5.6,0.55);
\coordinate (xm1) at (7.2,0.9);
\coordinate (xm) at (9.2,0);

% curves
\draw[thick,->] (x0) .. controls (0.7,0.1) and (1.2,0.8) .. (x1);
\draw[thick,->] (x1) .. controls (2.7,1.2) and (3.3,0.2) .. (x2);

% dotted middle
\node at (dots) {\(\cdots\)};

\draw[thick,->] (xm1) .. controls (7.8,0.8) and (8.5,0.1) .. (xm);

% points
\fill (x0) circle (1.5pt) node[below] {$x_0$};
\fill (x1) circle (1.5pt) node[above] {$x_1$};
\fill (x2) circle (1.5pt) node[below] {$x_2$};
\fill (xm1) circle (1.5pt) node[above] {$x_{m-1}$};
\fill (xm) circle (1.5pt) node[below] {$x_m$};

% labels for segments
\node at (0.95,0.65) {$c_1$};
\node at (3.0,0.8) {$c_2$};
\node at (8.25,0.65) {$c_m$};
\end{tikzpicture}
\caption{}
\label{fig:concatenated-paths}
\end{figure}

If all transport matrices are written with respect to the same global frame \(\mathcal F\), then
\[
M_{\mathcal F}^{\phi}(c)=M_{\mathcal F}^{\phi}(c_m)\cdots M_{\mathcal F}^{\phi}(c_2)M_{\mathcal F}^{\phi}(c_1).
\]

Later we will need to work with different frames on consecutive segments. We therefore record the transition rule here.

Suppose that along the segment \(c_j\) we use a frame \(\mathcal F_j\), and let \(M_j\) be the matrix of parallel transport along \(c_j\) with respect to \(\mathcal F_j\).  At the intermediate point \(x_j\) define the change-of-frame matrix \(C_j\) by
\[
\mathcal F_{j+1}(x_j)=\mathcal F_j(x_j)C_j .
\]
Then the total transport matrix from \(x_0\) to \(x_m\), expressed with respect to the initial frame \(\mathcal F_1(x_0)\) and the terminal frame \(\mathcal F_m(x_m)\), is
\begin{equation}\label{eq:different-frames-on-consecutive-segments}
    M_mC_{m-1}^{-1}M_{m-1}\cdots C_2^{-1}M_2C_1^{-1}M_1 .
\end{equation}

\subsection{The cyclic \texorpdfstring{\(\operatorname{Sp}(4,\mathbb{R})\)}{SP}-Higgs bundles in the Hitchin section}
\label{subsec:the-cyclic-SP(4,R)-Higgs-bundles-in-the-Hitchin-ection}

We now turn to the cyclic \(\operatorname{Sp}(4,\mathbb R)\)-Higgs bundles in the Hitchin
section. Fix a square root \(K^{\frac{1}{2}}\) of the canonical line bundle \(K\) and let \(q\) be a holomorphic quartic differential, i.e. \(q \in \mathrm{H}^0(X,K^4)\).  
We construct \(V\), \(\beta\) and \(\gamma\) as follows:
\begin{align*}
    V   &= K^{\frac{3}{2}} \oplus K^{-\frac{1}{2}}, &
    \beta &=
    \begin{pmatrix}
        q & 0 \\
        0 & 1
    \end{pmatrix},
    &
    \gamma &=
    \begin{pmatrix}
        0 & 1 \\
        1 & 0
    \end{pmatrix}.
\end{align*}
With respect to the decomposition \(E = V \oplus V^{*}\), we obtain
\begin{align*}
    E = K^{\frac{3}{2}} \oplus K^{-\frac{1}{2}} \oplus K^{-\frac{3}{2}} \oplus K^{\frac{1}{2}}, \qquad
    \phi =
    \begin{pmatrix}
        0 & 0 & q & 0\\
        0 & 0 & 0 & 1\\
        0 & 1 & 0 & 0\\
        1 & 0 & 0 & 0
    \end{pmatrix}.
\end{align*}
To make the cyclic structure more visible, we reorder the summands of \(E\) as
\[
E = K^{\frac{3}{2}} \oplus K^{\frac{1}{2}} \oplus K^{-\frac{1}{2}} \oplus K^{-\frac{3}{2}}.
\]
In this ordered basis the symplectic structure \(\Omega\) and the orthogonal structure \(Q\) take the form
\[
\Omega =
\begin{pmatrix}
    0 & 0 & 0 & 1\\
    0 & 0 & -1 & 0\\
    0 & 1 & 0 & 0\\
    -1 & 0 & 0 & 0
\end{pmatrix},
\qquad
Q =
\begin{pmatrix}
    0 & 0 & 0 & 1\\
    0 & 0 & 1 & 0\\
    0 & 1 & 0 & 0\\
    1 & 0 & 0 & 0
\end{pmatrix},
\]
while the Higgs field becomes
\[
\phi =
\begin{pmatrix}
    0 & 0 & 0 & q\\
    1 & 0 & 0 & 0\\
    0 & 1 & 0 & 0\\
    0 & 0 & 1 & 0
\end{pmatrix}.
\]
This Higgs bundle can be described by the diagram
\[
\begin{tikzcd}[
    column sep = large,
    arrows = {->, >=Stealth}
]
    K^{\frac{3}{2}}
        \arrow[r, "1"]
    & K^{\frac{1}{2}}
        \arrow[r, "1"]
    & K^{-\frac{1}{2}}
        \arrow[r, "1"]
    & K^{-\frac{3}{2}}
        \arrow[lll, bend left=20, "q"']
\end{tikzcd}
\]
The description above is also consistent with the cyclic \(\mathrm{SL}(n,\mathbb{R})\)-Higgs bundles studied in \cite{collier2017asymptotics} and \cite{li2025complete}, which specializes here to the case \(n=4\).

\begin{definition}\label{def:cyclic-Higgs-bundles}
    For a holomorphic quartic differential \(q\), we call
    \[
    ( E,\; \phi ) =
    \left(
    K^{\frac{3}{2}} \oplus K^{\frac{1}{2}} \oplus K^{-\frac{1}{2}} \oplus K^{-\frac{3}{2}},\;
    \begin{pmatrix}
        0 & 0 & 0 & q\\
        1 & 0 & 0 & 0\\
        0 & 1 & 0 & 0\\
        0 & 0 & 1 & 0
    \end{pmatrix}
    \right)
    \]
    the cyclic \(\mathrm{Sp}(4,\mathbb{R})\)-Higgs bundle in the Hitchin section and denote it by \((\mathbb{K}_{\mathrm{Sp}(4,\mathbb{R})}, \phi(q))\).
\end{definition}
\begin{remark}
    By the non-abelian Hodge correspondence, the above Higgs bundle,
    equipped with its harmonic metric, determines a representation
    \[
    \widetilde{\rho}_q:
    \pi_1(S)\longrightarrow \operatorname{Sp}(4,\mathbb R).
    \]
    Its projection to
    \(\operatorname{PSp}(4,\mathbb R)\) is the Hitchin representation
    associated with \(q\).

    The lift \(\widetilde{\rho}_q\) depends on the choice of the square
    root \(K^{1/2}\). Indeed, any other square root is of the form
    \[
    \widehat K^{1/2}=K^{1/2}\otimes L,
    \qquad L^2\simeq\mathcal O,
    \]
    and the corresponding Higgs bundle is obtained by tensoring \(E\)
    with \(L\). On the representation side, this changes the lift by
    a character
    \[
    \chi_L:\pi_1(S)\longrightarrow\{\pm I\},
    \qquad
    \widehat{\rho}_q(\gamma)
    =
    \chi_L(\gamma)\widetilde{\rho}_q(\gamma).
    \]
    Thus different choices of \(K^{1/2}\) give different
    \(\operatorname{Sp}(4,\mathbb R)\)-lifts of the same
    \(\operatorname{PSp}(4,\mathbb R)\)-Hitchin representation.
\end{remark}
When \(X\) is compact, it is a standard fact that \(\bigl(\mathbb{K}_{\mathrm{Sp}(4,\mathbb{R})},\phi(q)\bigr)\) admits a unique harmonic metric \(h\) compatible with the \(\mathrm{Sp}(4,\mathbb{R})\)-structure and taking the diagonal form
\[
h = \operatorname{diag}(h_1,h_2,h_3,h_4)=\operatorname{diag}(h_1, h_2, h_2^{-1}, h_1^{-1}).
\]
We refer to \cite{baraglia2010g2,collier2017asymptotics} for further details.  

Let us write the Hitchin equation in a local coordinate to see its explicit form.  Choose a local coordinate \(z\) on \(X\) and write \(q = q(z) \odif{z}^4\). Such a coordinate induces a local holomorphic section \(\odif{z}^{\frac{1}{2}}\) of \(K^{\frac{1}{2}}\) satisfying
\[
\odif{z}^{\frac{1}{2}} \otimes \odif{z}^{\frac{1}{2}} = \odif{z}
\]
on its domain of definition.

\begin{definition}\label{def:frame-induced-by-the-coordinate}
    The local frame
    \[
    \mathcal{F}_z = \bigl\{ \odif{z}^{\frac{3}{2}}, \odif{z}^{\frac{1}{2}}, \odif{z}^{-\frac{1}{2}},\odif{z}^{-\frac{3}{2}} \bigr\}
    \]
    is called the canonical frame induced by \(z\).
\end{definition}

Under the frame \(\mathcal{F}_z\), the Hitchin equation \eqref{Hitchin equation} takes the form
\begin{equation}
    \begin{dcases}
         \overline{\partial}\partial \log(h_1) + \abs{q}^2 h_1 h_4^{-1} \odif{z}\wedge\odif{\bar{z}} + h_2 h_1^{-1} \odif{\bar{z}} \wedge \odif{z} &=0, \\
         \overline{\partial}\partial \log(h_2) + h_2 h_1^{-1} \odif{z}\wedge\odif{\bar{z}} + h_3 h_2^{-1} \odif{\bar{z}} \wedge \odif{z} &=0, \\
        \overline{\partial}\partial \log(h_3) + h_3 h_2^{-1} \odif{z}\wedge\odif{\bar{z}} + h_4 h_3^{-1} \odif{\bar{z}} \wedge \odif{z} &=0, \\
        \overline{\partial}\partial \log(h_4) + h_4 h_3^{-1} \odif{z}\wedge\odif{\bar{z}} + \abs{q}^2 h_1 h_4^{-1} \odif{\bar{z}} \wedge \odif{z} &= 0 .
    \end{dcases}
\end{equation}
Equivalently,
\begin{equation}
    \begin{dcases}
        \partial_{z}\partial_{\bar{z}}\log(h_1) = \abs{q}^2 h_1 h_4^{-1} - h_2 h_1^{-1}, \\
         \partial_{z}\partial_{\bar{z}}\log(h_2) = h_2 h_1^{-1} - h_3 h_2^{-1}, \\
       \partial_{z}\partial_{\bar{z}}\log(h_3) = h_3 h_2^{-1} - h_4 h_3^{-1}, \\
        \partial_{z}\partial_{\bar{z}}\log(h_4) = h_4 h_3^{-1} - \abs{q}^2 h_1 h_4^{-1}.
    \end{dcases}
\end{equation}

To write the equation globally, we introduce a conformal metric \(g\) on \(X\).  In a local complex coordinate \(z = x + \sqrt{-1} y\) it can be written as
\[
g = g(z)(\odif{x}^2 + \odif{y}^2).
\]
This metric induces a Hermitian metric on \(K^{-1}\), which we still denote by \(g\):
\[
g = g(z) \odif{z} \otimes \odif{\bar{z}} .
\]
From \(g\) we obtain Hermitian metrics \(g_i\) on \(K^{\frac{5-2i}{2}}\) for \(i = 1,\dots,4\) by setting
\[
g_i = g^{\frac{2i-5}{2}} .
\]
Notice that \(g_1 = g_4^{-1}\) and \(g_2 = g_3^{-1}\).  Writing \(h_i = e^{u_i} g_i\), the relations \(h_1 h_4 = 1\) and \(h_2 h_3 = 1\) imply \(u_1 + u_4 = 0\) and \(u_2 + u_3 = 0\).  Then
\[
    \partial_{z}\partial_{\bar{z}}\log(h_i)
    =  \partial_{z}\partial_{\bar{z}} u_i +  \partial_{z}\partial_{\bar{z}}\log(g_i) 
    =  \partial_{z}\partial_{\bar{z}} u_i + \frac{2i-5}{2}  \partial_{z}\partial_{\bar{z}}\log(g)
\]
for \(i = 1,2,3,4\).

The Laplace operator \(\Delta_g\), the Gaussian curvature \(k_g\), and the pointwise \(g\)-norm \(\abs{q}_g^2\) of \(q\) are defined globally by
\[
\Delta_g := \frac{1}{g}\partial_{z}\partial_{\bar{z}},\qquad
k_g := -\frac{2}{g}\partial_{z}\partial_{\bar{z}}\log(g),\qquad
\abs{q}_g^2 := \frac{\abs{q}^2}{g^4}.
\]
In these terms, the Hitchin equation \eqref{Hitchin equation} becomes
\begin{equation}\label{long eq: globoal cyclic SP(4,R)-Hitchin equation}
    \begin{dcases}
        \Delta_g u_1 = \abs{q}_g^2 e^{u_1-u_4} - e^{u_2-u_1} -\frac{3}{4} k_g, \\
        \Delta_g u_2 =  e^{u_2-u_1} - e^{u_3-u_2} -\frac{1}{4} k_g, \\
        \Delta_g u_3 =  e^{u_3-u_2} - e^{u_4-u_3} +\frac{1}{4} k_g, \\
        \Delta_g u_4 = e^{u_4-u_3}  - \abs{q}_g^2 e^{u_1-u_4} +\frac{3}{4} k_g .
    \end{dcases}
\end{equation}
Because \(u_4 = -u_1\) and \(u_3 = -u_2\), the elliptic system \eqref{long eq: globoal cyclic SP(4,R)-Hitchin equation} reduces to
\begin{equation}\label{eq:globoal-cyclic-SP(4,R)-Hitchin-equation}
    \begin{dcases}
        \Delta_g u_1 = \abs{q}_g^2 e^{2u_1} - e^{u_2-u_1} -\frac{3}{4} k_g, \\
        \Delta_g u_2 =  e^{u_2-u_1} - e^{-2u_2} -\frac{1}{4} k_g .
    \end{dcases}
\end{equation}
\begin{remark}
    The Laplace operator used here differs from the usual Laplacian.  For the flat metric \(g = \odif{x}^2 + \odif{y}^2\), we have
    \[
    \Delta_g = \partial_{z}\partial_{\bar{z}} = \frac{1}{4}(\partial_x^2 + \partial_y^2) = \frac{1}{4}\Delta,
    \]
    where \(\Delta\) denotes the ordinary Laplacian on the complex plane.
\end{remark}

 The corresponding flat connection is denoted by \(D_{q}=\nabla_h+\phi(q)+\phi(q)^{*_h}\).  
We now write its local form in the canonical frame \(\mathcal{F}_z\) induced by the local coordinate \(z\).  
With respect to this frame the connection takes the form \(D_q = \odif + A_q\), where the connection form \(A_q\) is
{\allowdisplaybreaks
\begin{align}
    A_q &=
    \begin{pmatrix} 
        \partial_z\log h_1 & 0 & 0 & 0 \\
        0 & \partial_z\log h_2 & 0 & 0 \\
        0 & 0 & \partial_z\log h_3 & 0 \\
        0 & 0 & 0 & \partial_z\log h_4
    \end{pmatrix} \odif{z} \nonumber\\
    &\qquad + 
    \begin{pmatrix} 
        0 & 0 & 0 & q(z) \\
        1 & 0 & 0 & 0 \\
        0 & 1 & 0 & 0 \\
        0 & 0 & 1 & 0
    \end{pmatrix} \odif{z}
    +
    \begin{pmatrix}
        0 & h_2 h_1^{-1} & 0 & 0 \\
        0 & 0 & h_3 h_2^{-1} & 0 \\
        0 & 0 & 0 & h_4 h_3^{-1} \\
        \bar{q}(z)h_1 h_4^{-1} & 0 & 0 & 0
    \end{pmatrix} \odif{\bar{z}} \nonumber\\
    &= 
    \begin{pmatrix}
        \partial_z u_1 & 0 & 0 & 0 \\
        0 & \partial_z u_2 & 0 & 0 \\
        0 & 0 & -\partial_z u_2 & 0 \\
        0 & 0 & 0 & -\partial_z u_1
    \end{pmatrix} \odif{z} 
    + 
    \begin{pmatrix}
        \partial_z\log g & 0 & 0 & 0 \\
        0 & \partial_z\log g & 0 & 0 \\
        0 & 0 & \partial_z\log g & 0 \\
        0 & 0 & 0 & \partial_z\log g
    \end{pmatrix} \odif{z} \nonumber\\
    &\qquad + 
    \begin{pmatrix}
        0 & 0 & 0 & q(z) \\
        1 & 0 & 0 & 0 \\
        0 & 1 & 0 & 0 \\
        0 & 0 & 1 & 0
    \end{pmatrix} \odif{z}
    +\begin{pmatrix}
        0 & e^{u_2-u_1} & 0 & 0 \\
        0 & 0 & e^{-2u_2} & 0 \\
        0 & 0 & 0 & e^{u_2-u_1}\\
        \bar{q}(z)e^{2u_1} & 0 & 0 & 0
    \end{pmatrix} \odif{\bar{z}} \label{eq:connection form}
\end{align}
}

\subsection{Quartic differentials and the induced flat metrics}\label{subsec:quartic-differentials-and-the-induced-flat-metrics}

Let \(q\in H^0(X,K^4)\) be a non-zero holomorphic quartic differential.  Its zero set is denote by
\[
Z(q)=\{p\in X:q(p)=0\}.
\]
When \(X\) is compact, \(Z(q)\) is finite and every zero has finite order.

The quartic differential induces a singular flat metric \(\abs{q}^{1/2}\) on \(X\), given locally by
\[
\abs{q}^{1/2}=|q(z)|^{1/2}|\odif{z}|^2
\]
when \(q=q(z)\odif{z}^4\) in a holomorphic coordinate \(z\).  This expression is independent of the choice of coordinate, so it defines a genuine flat Riemannian metric on \(X\setminus Z(q)\).

A local holomorphic coordinate \(w\) on \(X\setminus Z(q)\) is called a natural coordinate for \(q\) if
\[
q=\odif{w}^4.
\]
In such a coordinate the flat metric is simply \(|q|^{1/2}=|\mathrm{d}w|^2\).

Natural coordinates exist on any simply connected domain.  Indeed, let \(U\subset X\setminus Z(q)\) be simply connected.  Then \(q\) admits a holomorphic fourth root on \(U\): a holomorphic one-form \(\varphi\) with
\[
\varphi^4=q .
\]
Fix a base point \(z_*\in U\) and set
\[
w(z)=\int_{z_*}^{z}\varphi ,\qquad z\in U .
\]
Because \(U\) is simply connected, the integral is path-independent. We have \(\odif{w}=\varphi\), hence \(q=\odif{w}^4\). Thus \(w\) is a natural coordinate. Changing \(z_*\) merely adds a constant to \(w\).

Natural coordinates are unique up to translation and multiplication by a fourth root of unity: if \(w\) and \(w'\) are two natural coordinates on a connected open set, then
\[
w'=(\sqrt{-1})^k w+b
\]
for some \(k\in\mathbb Z/4\mathbb Z\) and \(b\in\mathbb C\).  The four choices of a holomorphic fourth root of \(q\) are, in a local natural coordinate,
\[
\varphi_k=(\sqrt{-1})^{1-k}\odif{w},\qquad k=1,\dots,4 .
\]
Near a zero \(p\in Z(q)\) of order \(d\), we can choose a local coordinate \(z\) centred at \(p\) with
\(
q=z^d\odif{z}^4,
\)
so that
\[
|q|^{1/2}=|z|^{d/2}|\mathrm{d}z|^2 .
\]
Hence \(p\) is a conical singularity of total angle \(2\pi+\frac{d\pi}{2}\).

With the singular flat metric \(|q|^{1/2}\) at hand, we can now choose for every non-trivial free homotopy class \([\gamma]\) a closed geodesic curve along which to compute the holonomy. Each such class contains a closed geodesic representative for this metric. We pick one and denote it by \(c_\gamma\). Since the holonomy of a flat connection depends only on the free homotopy class (different base points or frames merely conjugate the matrix, and these conjugations do not affect the asymptotic quantities we study), it suffices to work with these geodesic representatives.

To compute the holonomy along \(c_\gamma\) we first look at its geometry. Away from \(Z(q)\) the metric is Euclidean in a natural coordinate, so every geodesic segment there becomes a straight line segment. Recall that a saddle connection of \(q\) is a geodesic segment for \(|q|^{1/2}\) whose endpoints lie in \(Z(q)\) and whose interior avoids \(Z(q)\). The closed geodesic curve \(c_{\gamma}\) passes through zeros of \(q\) and decomposes as a finite concatenation of saddle connections
\[
c_\gamma=c_{\ell}*c_{\ell-1}*\cdots*c_1,
\]
where each \(c_i\) joins \(p_{i}\) and \(p_{i-1}\) as drawn in Figure~\ref{geodesic representation}.

\begin{figure}[h]
 \centering
\begin{tikzpicture}[
    >=Stealth,
    thick,
    every node/.style={font=\normalsize}
]

\coordinate (p1)   at (0,0);
\coordinate (pl) at (4,3);
\coordinate (p2)   at (9,0);

\node[below left]  at (p1)   {$p_1$};
\node[above]       at (pl) {$p_2$};
\node[below right] at (p2)   {$p_{\ell}$};

\draw[->] (p1) -- (p2)    node[midway,below] {$c_1$};

\draw[->] (pl) -- (p1) node[midway,above left] {$c_{2}$};

\draw[->]      (p2) -- ($(pl)!0.7!(p2)$) 
  node[near end,above right] {$c_{\ell}$};
\draw[dashed] ($(pl)!0.7!(p2)$)-- ($(pl)!0.3!(p2)$);
\draw[->]      ($(pl)!0.3!(p2)$)-- (pl) 
node[near end,above right] {$c_{3}$};
\end{tikzpicture}
\caption{}
\label{geodesic representation}
\end{figure}

\subsection{The ray family}
\FloatBarrier
We now turn to the one-parameter family of holonomies that is the central object of this paper.  Fix a non-zero holomorphic quartic differential \(q\in H^0(X,K^4)\).  For \(t>0\), set
\[
q_t = t q .
\]
Denote the corresponding cyclic \(\mathrm{Sp}(4,\mathbb{R})\)-Higgs bundle in the Hitchin section by
\[
(\mathbb{K}_{\mathrm{Sp}(4,\mathbb{R})},\phi(q_t)),
\]
and let \(h_t\) be the compatible harmonic metric solving the Hitchin equation \eqref{Hitchin equation}.  The associated flat connection is
\[
D_t = \nabla_{h_t} + \phi(q_t) + \phi(q_t)^{*_{h_t}} .
\]
For an oriented piecewise smooth closed curve \(c\) on \(X\) we write \(\operatorname{Hol}_t(c)\) for the holonomy of \(D_t\) along \(c\).

The analysis away from the zeros is straightforward in the natural coordinates, which has been studied by \cite{collier2017asymptotics,Mochizuki_2016}. Near a zero, however, the Stokes phenomenon occurs. This is a purely local question, which motivates studying the model
\[
q = z^d \odif{z}^4
\]
on the complex plane. We carry out this local analysis in the next section. Later, in Section~\ref{sec:local-comparison-near-zeros}, we will prove that this model gives the correct rescaled picture near any zero of \(q\) on the compact Riemann surface \(X\).  Its Stokes matrices will then supply the transition matrices needed for the global holonomy computation.
\section{Planar theory and Stokes data}
\label{sec:local-model}

In this section we collect the results from 
\cite{tamburelli2024planar} that are needed in our setting.  We give a
self-contained account both for completeness and to establish the
notation and explicit formulas that will be essential for our own
computations later in the paper. All of the general statements
presented here are proved in \cite{tamburelli2024planar}.We reproduce
them in our right-half-plane convention and with our choice of the
diagonalizing matrix \(S\).

We present the general theory: the existence and uniqueness of the
harmonic metric, the constant model \(q = \odif{z}^4\), the half-plane
decomposition, the comparison with the constant model, the
Tamburelli--Wolf estimates, and the Stokes directions together with
the associated sectorial limits and Stokes jump theorem.  After that
we specialize to the model
\[
q = z^d \odif{z}^4,
\]
which describes the local behavior near a zero of a quartic differential on a compact Riemann surface. For this model, we compute the Stokes data explicitly, based on results from \cite{guestlin2026110730}.

\subsection{The harmonic metric}

Let
\[
q = q(z)\odif{z}^4
\]
be a polynomial quartic differential on \(\mathbb{C}\) with \(q(z)\) a polynomial of degree \(d\).  
Fix a square root \(K^{\frac{1}{2}}\) of the canonical line bundle of \(\mathbb{C}\).  
The cyclic \(\mathrm{Sp}(4,\mathbb{R})\)-Higgs bundle in the Hitchin section \((\mathbb{K}_{\mathrm{Sp}(4,\mathbb{R})},\phi(q))\) is
\begin{align*}
   \mathbb{K}_{\mathrm{Sp}(4,\mathbb{R})}
   = K^{\frac{3}{2}} \oplus K^{\frac{1}{2}} \oplus K^{-\frac{1}{2}} \oplus K^{-\frac{3}{2}},
   \qquad
   \phi =
   \begin{pmatrix}
       0 & 0 & 0 & q \\
       1 & 0 & 0 & 0 \\
       0 & 1 & 0 & 0 \\
       0 & 0 & 1 & 0
   \end{pmatrix}.
\end{align*}
In this ordered basis, the symplectic structure \(\Omega\) and the orthogonal structure \(Q\) take the form
\[
\Omega =
\begin{pmatrix}
    0 & 0 & 0 & 1\\
    0 & 0 & -1 & 0\\
    0 & 1 & 0 & 0\\
    -1 & 0 & 0 & 0
\end{pmatrix},
\qquad
Q =
\begin{pmatrix}
    0 & 0 & 0 & 1\\
    0 & 0 & 1 & 0\\
    0 & 1 & 0 & 0\\
    1 & 0 & 0 & 0
\end{pmatrix}.
\]

We equip \(\mathbb{C}\) with the standard flat metric \(g = \odif{x}^2 + \odif{y}^2\), whose Gaussian curvature is \(k_g = 0\).  
A diagonal harmonic metric
\[
h = \operatorname{diag}(e^{u_1}, e^{u_2}, e^{u_3}, e^{u_4}),
\]
compatible with the \(\mathrm{Sp}(4,\mathbb{R})\)-structure must satisfy
\[
u_1 + u_4 = 0,\qquad u_2 + u_3 = 0.
\]
Inserting these conditions into \eqref{eq:globoal-cyclic-SP(4,R)-Hitchin-equation}, the Hitchin equation becomes
\begin{equation}
\label{eq:planar-hitchin-equation}
\begin{dcases}
\Delta u_1 = |q|^2 e^{2u_1} - e^{u_2-u_1},\\
\Delta u_2 = e^{u_2-u_1} - e^{-2u_2},
\end{dcases}
\end{equation}
where \(\Delta = \partial_z \partial_{\bar z}\).

\begin{remark}
    Note a sign difference with the notation of \cite{tamburelli2024planar}:
    their functions \(u_i^{\mathrm{TW}}\) are the negatives of ours,
    \(u_i^{\mathrm{TW}} = -u_i\).
    After this change of variables, their equations coincide with
    \eqref{eq:planar-hitchin-equation}.
\end{remark}

To guarantee uniqueness, we impose the completeness condition on \(h\): the induced metrics
\[
e^{u_i - u_{i-1}} g,\qquad i=1,\dots,4,
\]
over \(\mathbb{C}\) are required to be complete, with the convention \(u_0 = u_4\).

The existence and uniqueness theory for such solutions is now summarized. 
The general cyclic \(\mathrm{SL}(n,\mathbb{R})\) case was treated in \cite{li2025complete}, while the \(\mathrm{SL}(3,\mathbb{R})\) case appeared earlier in \cite{dumas2015polynomial}.  
The following statement collects the relevant facts for our \(\mathrm{Sp}(4,\mathbb{R})\) setting.

\begin{theorem}{\cite[Theorem~2.1 and Corollary~2.6]{tamburelli2024planar}}
\label{thm:general-planar-theory}
Let \(q = q(z)\odif{z}^4\) be a polynomial quartic differential on \(\mathbb{C}\).  
Then the cyclic \(\mathrm{Sp}(4,\mathbb{R})\)-Higgs bundle \((\mathbb{K}_{\mathrm{Sp}(4,\mathbb{R})},\phi(q))\) admits a unique diagonal complete harmonic metric
\[
h = \operatorname{diag}\bigl(e^{u_1}, e^{u_2}, e^{-u_2}, e^{-u_1}\bigr)
\]
satisfying
\begin{align*}
   -A r^{-\alpha} &\le u_1 + \tfrac{3}{4}\log|q| \le 0,\\
   -A r^{-\alpha} &\le u_2 + \tfrac{1}{4}\log|q| \le 0.
\end{align*}
Here \(r\) denotes the \(|q|^{1/2}\)-distance to the zero set of \(q\) and \(\alpha,A\) are positive constants.
\end{theorem}
\begin{proof}
Existence and the stated estimates follow from the super-subsolution method used in \cite[Theorem~2.1]{tamburelli2024planar}.
In the same work, uniqueness is formulated in the language of good filtered Higgs bundles \cite{mochizuki2010wild,Mochizuki2021good} (see also \cite{mochizuki2014harmonic}).
To keep our exposition elementary, we avoid this machinery and instead rely on the uniqueness theorem for complete solutions proved in \cite{li2025complete} for general cyclic \(\mathrm{SL}(n,\mathbb{R})\)-Higgs bundles.
Any solution satisfying the above bounds is easily checked to be complete and hence unique.
\end{proof}

\subsection{The constant model \texorpdfstring{\(q=\odif{z}^4\)}{q}}

We briefly reproduce the constant-coefficient model from
\cite[Section~3.2]{tamburelli2024planar}, translated into our notation and ordering. It is easy to verify that
\[
u_1=u_2=0
\]
is the unique complete solution of \eqref{eq:planar-hitchin-equation},
hence the diagonal harmonic metric is
\[
h=\operatorname{diag}(1,1,1,1).
\]

In the canonical frame \(\mathcal F_z\) induced by the coordinate \(z\) (see Definition~\ref{def:frame-induced-by-the-coordinate}; in the planar case this frame is globally defined), the flat connection \(D_0\) becomes
\[
D_0 = \odif + U_0\odif{z} + V_0\odif{\bar{z}},
\]
with
\[
U_0=
\begin{pmatrix}
0&0&0&1\\
1&0&0&0\\
0&1&0&0\\
0&0&1&0
\end{pmatrix},
\qquad
V_0=
\begin{pmatrix}
0&1&0&0\\
0&0&1&0\\
0&0&0&1\\
1&0&0&0
\end{pmatrix}.
\]

Let \(\Psi_{0}(z)\) be the fundamental matrix of \(D_0\) with base point \(0\) and with respect to the frame \(\mathcal F_z\), as in \eqref{eq:def-Psi} and \eqref{eq:Psi-connection-equation}.  Thus
\begin{equation}
    \begin{dcases}
        \Psi_{0}(0)=I, \\
        \Psi_{0}^{-1}\odif\Psi_{0}=U_0\odif{z}+V_0\odif{\bar{z}}.
    \end{dcases}
\end{equation}
For the ray \(\gamma_\theta(s)=s e^{\sqrt{-1}\theta}\) (\(s\ge 0\)), set \(\psi_0(s)=\Psi_{0}(\gamma_\theta(s))\).  Then
\begin{equation}
\label{eq:constant-model-ode}
\begin{dcases}
    \psi_0(0)=I,\\
    \dfrac{\odif\psi_0}{\odif s}(s)
    =\psi_0(s)\bigl(e^{\sqrt{-1}\theta}U_0+e^{-\sqrt{-1}\theta}V_0\bigr).
\end{dcases}
\end{equation}
Solving this linear ODE yields
\[
\psi_0(s)=
\exp\left(s
\begin{pmatrix}
     0 & e^{-\sqrt{-1}\theta} & 0 & e^{\sqrt{-1}\theta} \\
       e^{\sqrt{-1}\theta} & 0 & e^{-\sqrt{-1}\theta} & 0 \\
       0 & e^{\sqrt{-1}\theta} & 0 & e^{-\sqrt{-1}\theta} \\
       e^{-\sqrt{-1}\theta} & 0 & e^{\sqrt{-1}\theta}& 0
\end{pmatrix}
\right).
\]
Hence, writing \(z=|z|e^{\sqrt{-1}\theta}\), we have
\begin{equation}\label{eq:constant-standard-model}
    \Psi_0(z)=
    \exp\left(|z|
    \begin{pmatrix}
         0 & e^{-\sqrt{-1}\theta} & 0 & e^{\sqrt{-1}\theta} \\
           e^{\sqrt{-1}\theta} & 0 & e^{-\sqrt{-1}\theta} & 0 \\
           0 & e^{\sqrt{-1}\theta} & 0 & e^{-\sqrt{-1}\theta} \\
           e^{-\sqrt{-1}\theta} & 0 & e^{\sqrt{-1}\theta}& 0
    \end{pmatrix}
    \right).
\end{equation}
Introduce the constant unitary matrix
\begin{equation}\label{eq:definition-of-S}
    S=\frac{1}{2}
    \begin{pmatrix}
        1 & e^{-\frac{3\pi\sqrt{-1}}{4}} 
        & e^{\frac{\pi\sqrt{-1}}{2}} 
        & e^{-\frac{\pi\sqrt{-1}}{4}} \\ 
        1 & e^{-\frac{\pi\sqrt{-1}}{4}} 
        & e^{-\frac{\pi\sqrt{-1}}{2}} 
        & e^{-\frac{3\pi\sqrt{-1}}{4}} \\ 
        1 & e^{\frac{\pi\sqrt{-1}}{4}} 
        & e^{\frac{\pi\sqrt{-1}}{2}} 
        & e^{\frac{3\pi\sqrt{-1}}{4}} \\ 
         1 & e^{\frac{3\pi\sqrt{-1}}{4}} 
        & e^{-\frac{\pi\sqrt{-1}}{2}} 
        & e^{\frac{\pi\sqrt{-1}}{4}}
    \end{pmatrix}.
\end{equation}
A direct computation shows that \(S\) diagonalizes the coefficient matrix in \eqref{eq:constant-standard-model}:
\[
S^{-1}
\begin{pmatrix}
     0 & e^{-\sqrt{-1}\theta} & 0 & e^{\sqrt{-1}\theta} \\
       e^{\sqrt{-1}\theta} & 0 & e^{-\sqrt{-1}\theta} & 0 \\
       0 & e^{\sqrt{-1}\theta} & 0 & e^{-\sqrt{-1}\theta} \\
       e^{-\sqrt{-1}\theta} & 0 & e^{\sqrt{-1}\theta}& 0
\end{pmatrix}
S
=
\begin{pmatrix}
    2\cos\theta &   &   &  \\
    & 2\sin\theta &  &    \\
    &  & -2\cos\theta &    \\
    &  &    & -2\sin\theta
\end{pmatrix}.
\]
The ordering of the eigenvalues differs from the one used in \cite{tamburelli2024planar}; our choice is motivated by the cyclic form \eqref{eq:cyclic-form-for-S-and-D} that \(S\) and the diagonal matrix admit, which simplifies the later calculations. Moreover, \(\mathcal{F}_zS\) is also a real frame for the induced real structure.

Set
\begin{equation}\label{eq:diagonal-in-constant-standard-model}
    D(z)=
    \begin{pmatrix}
        2|z|\cos\theta &   &   &  \\
        & 2|z|\sin\theta &  &    \\
        &  & -2|z|\cos\theta &    \\
        &  &    & -2|z|\sin\theta
    \end{pmatrix}.
\end{equation}
Then from \eqref{eq:constant-standard-model} we obtain
\begin{equation}\label{eq:constant-standard-model-diagonal}
     \Psi_0(z)
     = S\exp \left(D(z)\right)S^{-1} \notag\\
     = S
     \begin{pmatrix}
         e^{2|z|\cos\theta} &   &   &  \\
         & e^{2|z|\sin\theta} &  &    \\
         &  & e^{-2|z|\cos\theta} &    \\
         &  &    & e^{-2|z|\sin\theta}
     \end{pmatrix}
     S^{-1}.
   \end{equation}

For later calculations, it is convenient to write \(S\) and \(D(\theta)\) in the following form
\begin{equation}\label{eq:cyclic-form-for-S-and-D}
    (S)_{i,j}= \frac{1}{2}
    e^{\frac{\pi\sqrt{-1}}{4}(2i-5)(j-1)},
    \qquad
    \bigl(D(\theta)\bigr)_{i,i}= 2\cos\left(\theta-\frac{(i-1)\pi}{2}\right),
\end{equation}
where \(1\le i,j\le 4\).

\subsection{Half-plane decompositions and Stokes jumps}\label{subsec:Half-plane-decompositions-and-Stokes-jumps}

The constant model is not merely a toy example: on a half-plane equipped with a suitable natural coordinate, any polynomial quartic differential reduces to this local form. The constant-coefficient connection \(D_0\) therefore serves as the asymptotic model with which the planar connection \(D_q\) is compared.

For a general monic polynomial, a standard half-plane decomposition was established in \cite{tamburelli2024planar}. The statement given there uses \emph{upper} half-planes. Passing to right-half-planes replaces the quarter-turn \(\sqrt{-1}\) by \(-\sqrt{-1}\) in the transition law. We recall the adapted statement.

\begin{definition}
    A \emph{\(q\)-right-half-plane} is a pair \((U,w)\) where
    \(U\subset\mathbb C\setminus Z(q)\) is a simple connected domain and \(w\) is a natural
    coordinate for \(q\) (so \(q = \odif{w}^4\)) such that
    \[
    w : U \longrightarrow \mathbb{H}= \{\operatorname{Re} w > 0\}
    \]
    is biholomorphic.
\end{definition}

\begin{theorem}
    \label{thm:standard-half-plane-decomposition}
    Let \(q = q(z)\odif{z}^4\) be a monic polynomial quartic differential of
    degree \(d \ge 1\). Set \(N=d+4\). Then there exist a compact set
    \(K_0\supset K\) and \(q\)-right-half-planes
    \[
    (U_1,w_1),\dots,(U_N,w_N),
    \]
    with indices taken modulo \(N\), satisfying the following properties.
    \begin{enumerate}
        \item \(\mathbb C\setminus K_0 = \bigcup_{j=1}^{N} U_j\).
        \item The Euclidean ray \(\arg z = \frac{2\pi j}{N}\) is eventually
              contained in \(U_j\).
        \item The adjacent rays \(\arg z = \frac{2\pi(j-1)}{N}\) and
              \(\arg z = \frac{2\pi(j+1)}{N}\) are eventually disjoint from \(U_j\).
        \item On \(U_j\cap U_{j+1}\) the natural coordinates satisfy
              \(
              w_{j+1} = -\sqrt{-1} w_j + c_j
              \)
              for some constant \(c_j\in\mathbb C\). Both \(w_j\) and \(w_{j+1}\)
              map the overlap onto a sector of angle \(\pi/2\).
        \item Every Euclidean ray is eventually contained in one of the half-planes \(U_j\).
    \end{enumerate}
\end{theorem}

By Theorem~\ref{thm:standard-half-plane-decomposition}, each Euclidean ray is eventually contained in one of the half-planes \(U_j\).  Hence, to understand the asymptotic behavior of the fundamental matrix along a given ray, it suffices to work in the corresponding right-half-plane. We now fix one such half-plane and denote it by \((U,w)\).  Thus
\[
q = \odif{w}^4 \qquad\text{on } U,
\qquad
w : U \xrightarrow{\cong} \mathbb{H}= \{\operatorname{Re} w > 0\}.
\]
All subsequent computations in this subsection are performed with respect to the canonical frame \(\mathcal{F}_w\) induced by the natural coordinate \(w\) (see Definition~\ref{def:frame-induced-by-the-coordinate}).

With the geometric framework in place, we now write the connection form
explicitly in the chosen half-plane.  Let \(D_q\) be the flat connection
associated with the cyclic
\(\mathrm{Sp}(4,\mathbb{R})\)-Higgs bundle
\((\mathbb{K}_{\mathrm{Sp}(4,\mathbb{R})},\phi(q))\) over \(\mathbb{C}\).
On \(U\) we use the flat metric \(g = |q|^{1/2} = |\mathrm{d}w|^2\).  
In the canonical frame \(\mathcal F_w\) the connection takes the form
\(D_q = \odif + A_q\) with
\begin{align}\label{eq:connection-form-on-right-half-plane}
    A_q &=
    \begin{pmatrix}
    \partial_w \tilde u_1 & 0 & 0 & 0 \\
    0 & \partial_w \tilde u_2 & 0 & 0 \\
    0 & 0 & -\partial_w \tilde u_2 & 0 \\
    0 & 0 & 0 & -\partial_w \tilde u_1
    \end{pmatrix} \odif{w}
    +
    \begin{pmatrix}
    0 & 0 & 0 & 1\\
    1 & 0 & 0 & 0 \\
    0 & 1 & 0 & 0 \\
    0 & 0 & 1 & 0
    \end{pmatrix} \odif{w} \\
    &\qquad +
    \begin{pmatrix}
    0 & e^{\tilde u_2-\tilde u_1} & 0 & 0 \\
    0 & 0 & e^{-2\tilde u_2} & 0 \\
    0 & 0 & 0 & e^{\tilde u_2-\tilde u_1}\\
    e^{2\tilde u_1} & 0 & 0 & 0
    \end{pmatrix} \odif{\bar{w}} .
\end{align}
Here
\[
    \tilde u_1 = \log h_1(\odif{w}^{\frac32},\odif{w}^{\frac32})
               = u_1 + \tfrac34\log|q|,\qquad
    \tilde u_2 = \log h_2(\odif{w}^{\frac12},\odif{w}^{\frac12})
               = u_2 + \tfrac14\log|q|.
\]
By \eqref{eq:globoal-cyclic-SP(4,R)-Hitchin-equation} they satisfy
\begin{equation}
\label{eq:tilde-u-equation on w-plane}
\begin{dcases}
\Delta_w \tilde u_1 = e^{2\tilde u_1} - e^{\tilde u_2-\tilde u_1},\\
\Delta_w \tilde u_2 = e^{\tilde u_2-\tilde u_1} - e^{-2\tilde u_2},
\end{dcases}
\end{equation}
where \(\Delta_w = \partial_w \partial_{\bar w}\).

For the constant model \(q = \odif{w}^4\), the corresponding connection is \(D_0 = \odif + A_0\) with
\[
A_0 =
\begin{pmatrix}
0&0&0&1\\
1&0&0&0\\
0&1&0&0\\
0&0&1&0
\end{pmatrix}\odif{w}
+
\begin{pmatrix}
0&1&0&0\\
0&0&1&0\\
0&0&0&1\\
1&0&0&0
\end{pmatrix}\odif{\bar{w}} .
\]

Fix a base point \(w_*\in w(U)=\mathbb{H}\) and let \(\Psi_q(w)\) be the fundamental matrix of \(D_q\) in the frame \(\mathcal F_w\), normalized by \(\Psi_q(w_*)=I\).  As in Section~\ref{subsec:parallel-transport-holonomy}, it satisfies \(\Psi_q^{-1}\odif\Psi_q = A_q\).

For the constant model, we let \(\Psi_0(w)\) be the fundamental matrix normalized by \(\Psi_0(0)=I\) (the value at the origin is understood as the limit from the interior of \(\mathbb{H}\)).  Thus \(\Psi_0^{-1}\odif\Psi_0 = A_0\) and, writing \(w = |w|e^{\sqrt{-1}\theta}\), we have
\[
\Psi_0(w) = S\exp\bigl(D(w)\bigr)S^{-1} 
          = S
            \begin{pmatrix}
                   e^{2|w|\cos\theta} & & & \\
                   & e^{2|w|\sin\theta} & & \\
                   & & e^{-2|w|\cos\theta} & \\
                   & & & e^{-2|w|\sin\theta}
            \end{pmatrix}
          S^{-1}.
 \]

We now compare the two connections by defining
\[
G(w) = \Psi_q(w)\Psi_0(w)^{-1}.
\]
Then
\begin{equation}\label{eq:formula-for-G}
    \odif G = G\Theta,\qquad
    \Theta = \Psi_0(A_q-A_0)\Psi_0^{-1}.
\end{equation}
Hence the asymptotic analysis of \(D_q\) relative to the constant model reduces to the study of the one-form \(\Theta\).

Set
\[
U=
\begin{pmatrix}
\partial_{w} \tilde u_1 & 0 & 0 & 0 \\
0 & \partial_{w} \tilde u_2 & 0 & 0 \\
0 & 0 & -\partial_{w} \tilde u_2 & 0 \\
0 & 0 & 0 & -\partial_{w} \tilde u_1
\end{pmatrix},
\qquad
V=
\begin{pmatrix}
0 & e^{\tilde u_2-\tilde u_1}-1 & 0 & 0 \\
0 & 0 & e^{-2\tilde u_2}-1 & 0 \\
0 & 0 & 0 & e^{\tilde u_2-\tilde u_1}-1\\
e^{2\tilde u_1}-1 & 0 & 0 & 0
\end{pmatrix}.
\]
Then 
\[
A_q-A_0 = U\odif w + V\odif{\bar w}.
\]  
Introducing
\(R_1 = S^{-1} U S\) and \(R_2 = S^{-1} V S\), we obtain
\begin{equation}\label{eq:formula-for-Theta}
    \Theta
    = \Psi_0(U\odif w+V\odif{\bar w})\Psi_0^{-1}
    = S\exp(D(w))(R_1\odif w+R_2\odif{\bar w})\exp(-D(w))S^{-1}.
\end{equation}
To obtain explicit expressions for \(R_1\) and \(R_2\), we use the form \eqref{eq:cyclic-form-for-S-and-D} of \(S\) to compute.  For convenience, write
\(\tilde u_3=-\tilde u_2\) and \(\tilde u_4=-\tilde u_1\).  Because \(S\)
is unitary, \((S^{-1})_{i,k}= \overline{S_{k,i}}\).  Then
\[
(R_1)_{i,j}
= \sum_{k=1}^4 (S^{-1})_{i,k}U_{k,k}S_{k,j} 
= \frac14 \sum_{k=1}^4
   e^{\frac{\pi\sqrt{-1}}{4}(5-2k)(i-j)}\partial_w\tilde u_k .
\]
Separating the contributions of \(k=1,2\) and using
\(\tilde u_3=-\tilde u_2,\tilde u_4=-\tilde u_1\) gives
\begin{equation}\label{eq:explict-formula-for-R1}
    (R_1)_{i,j}
    = (-1)^{j-i}\frac{\sqrt{-1}}{2}\Bigl(
        \sin\Bigl(\frac{\pi(j-i)}{4}\Bigr)\partial_w\tilde u_1
       +\sin\Bigl(\frac{3\pi(j-i)}{4}\Bigr)\partial_w\tilde u_2
       \Bigr).
\end{equation}
For the \(\odif{\bar w}\)-part we compute similarly, understanding the
index \(k+1\) modulo \(4\):
\begin{align}\label{eq:explict-formula-for-R2-old}
(R_2)_{i,j}
& = \sum_{k=1}^4 (S^{-1})_{i,k}V_{k,k+1}S_{k+1,j} \nonumber \\
& = \frac14 e^{\frac{\pi\sqrt{-1}}{2}(j-1)}
   \sum_{k=1}^4
   e^{\frac{\pi\sqrt{-1}}{4}(5-2k)(i-j)}
   \bigl(e^{\tilde{u}_{k+1}-\tilde u_k}-1\bigr).
\end{align}

To simplify the above formulas we introduce the auxiliary functions
\[
    \tilde{w}_1 = \tilde{w}_3 = \tilde u_1+\tilde u_2,\qquad
    \tilde{w}_2 = \tilde u_1-\tilde u_2,
\]
where all indices are taken modulo \(4\).  Then, for \(i \neq j\),
\begin{align}
    (R_1)_{i,j}
        &= (-1)^{j-i}\frac{\sqrt{-1}}{2}
          \sin\Bigl(\frac{\pi(j-i)}{4}\Bigr)
          \partial_w \tilde{w}_{j-i},\label{eq:R1-offdiag}\\[2mm]
    (R_2)_{i,j}
        &= -\frac14 e^{\frac{\pi\sqrt{-1}}{4}(j-i)}\sqrt{-1}^{1+i}
          \Delta_w \tilde w_{j-i},\label{eq:R2-offdiag}
\end{align}
with the index \(j-i\) understood modulo \(4\) as well.  For the
diagonal entries, we have
\begin{equation}\label{eq:R1-diag}
    (R_1)_{i,i} = 0, \qquad
(R_2)_{i,i}
   = \frac14 e^{\frac{\pi\sqrt{-1}}{2}(i-1)}
      \bigl(2e^{\tilde u_2-\tilde u_1}+e^{2\tilde u_1}+e^{-2\tilde u_2}-4\bigr).
\end{equation}

We shall need the following estimates for the auxiliary functions.

\begin{proposition}[{\cite[Lemma~3.11]{tamburelli2024planar}}]
\label{prop:estimates-of-tilde-w}
Let \(w=r e^{\sqrt{-1}\theta}\).  Then, as \(r\to+\infty\),
\[
\tilde w_k = O\biggl( \frac{e^{-2|1-(\sqrt{-1})^{k}|r}}{\sqrt{r}} \biggr), \qquad k=1,2.
\]
Moreover, their first derivatives and Laplacians satisfy
\[
\partial_w \tilde w_{k}
   = O\biggl( \frac{e^{-2|1-(\sqrt{-1})^{k}|r}}{\sqrt{r}} \biggr),\qquad
\Delta_w \tilde w_{k}
   = O\biggl( \frac{e^{-2|1-(\sqrt{-1})^{k}|r}}{\sqrt{r}} \biggr), \qquad k=1,2
\]
as \(r \to +\infty\).
\end{proposition}
\begin{proof}
The first statement is a direct reformulation of
\cite[Lemma~3.11]{tamburelli2024planar}. Our auxiliary functions differ from those used by Tamburelli and Wolf only by constant multiplicative factors. The estimates for the derivatives and Laplacians follow from the arguments given in the proof of Lemma~3.11 and Proposition~3.3 of \cite{tamburelli2024planar}..
\end{proof}

From the above estimates and the explicit formulas \eqref{eq:R1-offdiag}, \eqref{eq:R2-offdiag} and \eqref{eq:R1-diag}, we immediately obtain the following bounds.
\begin{proposition}
\label{prop:TW-error-estimate}
Let \(w=r e^{\sqrt{-1}\theta}\). Then, as \(r\to+\infty\), the entries
of \(R_1\) and \(R_2\) satisfy
\[
(R_\nu)_{i,j} =
O\biggl(
\frac{\exp\bigl(-2\bigl|1-(\sqrt{-1})^{i-j}\bigr|r\bigr)}{\sqrt{r}}
\biggr),\qquad i\neq j,
\]
and
\[
(R_\nu)_{i,i} =
o\biggl(
\frac{e^{-2\sqrt{2}r}}{\sqrt{r}}
\biggr),\qquad \nu=1,2 .
\]
\end{proposition}
\begin{proof}
This follows directly from Proposition~\ref{prop:estimates-of-tilde-w}
together with the expressions \eqref{eq:R1-offdiag},
\eqref{eq:R2-offdiag} and \eqref{eq:R1-diag}.
\end{proof}

\subsubsection{Stokes directions and sectorial limits}

\begin{definition}
\label{def:stokes-directions}
The \emph{Stokes directions} inside the right half-plane $\mathbb{H}$ are
\[
\theta_{1} = -\frac{\pi}{4},\qquad \theta_{2} = 0,\qquad \theta_{3} = \frac{\pi}{4}.
\]
These three rays divide $\mathbb{H}$ into four open sectors.
We denote by $J_0,J_1,J_2,J_3$ both the angular intervals
\[
J_{0} = \Bigl(-\frac{\pi}{2},-\frac{\pi}{4}\Bigr),\qquad
J_{1}  = \Bigl(-\frac{\pi}{4},0\Bigr),\qquad
J_{2}  = \Bigl(0,\frac{\pi}{4}\Bigr),\qquad
J_{3} = \Bigl(\frac{\pi}{4},\frac{\pi}{2}\Bigr).
\]
and the corresponding sectors of $\mathbb{H}$:
\[
J_k = \{ w = r e^{\sqrt{-1}\theta} : r>0, -\frac{\pi}{2}+\frac{k\pi}{4}<\theta<-\frac{\pi}{2}+\frac{(k+1)\pi}{4} \}.
\]
\end{definition}

\begin{figure}[htbp]
\centering
\begin{tikzpicture}[scale=1.8]
 
  \draw[dashed, gray] (0,-1.5) -- (0,1.5) node[above] {\(\operatorname{Im} w\)};

  \draw[thick, red] (0,0) -- (1.6,1.6) node[right] {\(\theta_3 = \frac{\pi}{4}\)};
  \draw[thick, red] (0,0) -- (1.8,0)   node[below] {\(\theta_2 = 0\)};
  \draw[thick, red] (0,0) -- (1.6,-1.6) node[right] {\(\theta_1 = -\frac{\pi}{4}\)};
  
  \node at (1.1,0.6)  {\(J_{2}\)};
  \node at (1.1,-0.3) {\(J_{1}\)};
  \node at (0.5,0.9)  {\(J_{3}\)};
  \node at (0.5,-0.9) {\(J_{0}\)};
  
  \fill (0,0) circle (0.5pt) node[below left] {\(0\)};
\end{tikzpicture}
\caption{Stokes directions and stable sectors in the right half-plane \(\mathbb{H}\).}
\label{fig:stokes-directions}
\end{figure}
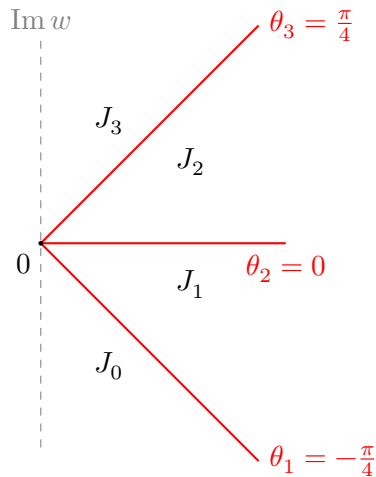
\begin{theorem}
\label{thm:sectorial-limits}
Let \(\Psi_q(w)\) be the fundamental matrix of the connection \(D_q\) with respect the frame \(\mathcal{F}_w\), and let \(\Psi_0(w)\) be the corresponding
fundamental matrix for the constant model \(q = \odif{w}^4\).  For each Stokes sector
\[
J_{*} \in \{ J_{0},\; J_{1},\; J_{2},\; J_{3} \},
\]
there exists a constant matrix \(L_{*} \in \mathrm{Sp}(4,\mathbb{R})\)
such that
\[
\lim_{r\to+\infty}
\Psi_q(r e^{\sqrt{-1}\theta})
\Psi_0(r e^{\sqrt{-1}\theta})^{-1}
= L_{*}
\]
as \(r\to+\infty\) with \(\theta \in J_*\).
\end{theorem}

\begin{proof}[Sketch of proof]
The argument follows exactly \cite[Lemma~3.5]{tamburelli2024planar}, only translated into our convention. We refer to \cite[Lemma~3.5]{tamburelli2024planar} for the full details.
\end{proof}

\begin{remark}
The limit matrices \(L_{*}\) depend on the choice of the base
point used to normalize the fundamental matrices \(\Psi_q\). The unipotent matrices \(U\) that describe the Stokes jumps below, however, do not depend on this choice.
\end{remark}

\subsubsection{The Stokes jump theorem}

\begin{theorem}
\label{thm:stokes-jumps}
With the notation of Theorem~\ref{thm:sectorial-limits}, the sectorial
limits
\[
L_{0},\qquad L_{1},\qquad L_{2},\qquad L_{3}
\]
satisfy the following Stokes jump relations. There exist unipotent
matrices
\[
 U_{1},\qquad
 U_{2},\qquad
 U_{3}
\]
such that
\[
L_{k-1}^{-1}L_k = S U_k S^{-1},\qquad k=1,2,3.
\]
\end{theorem}

\begin{proof}
We give the proof in full detail, both for completeness and because the explicit form of the unipotent matrices \(U_{1},U_2,U_{3}\) obtained along the way will be essential for all later computations.  The argument follows the general framework of Tamburelli--Wolf \cite[Lemma~3.6]{tamburelli2024planar}, expressed in our right-half-plane coordinate convention.

Fix an index \(k\in\{1,2,3\}\) and let \(\theta_k=-\frac{\pi}{2}+\frac{k\pi}{4}\) be the Stokes direction that separates \(J_{k-1}\) from \(J_k\).
For fixed small \(\varepsilon>0\) and \(t>0\), consider the two rays
\[
\gamma_{k-1}(t)=t e^{\sqrt{-1}(\theta_k-\varepsilon)},
\qquad
\gamma_{k}(t)=t e^{\sqrt{-1}(\theta_k+\varepsilon)}.
\]
By Theorem~\ref{thm:sectorial-limits}, as \(t\to+\infty\),
\[
G(\gamma_{k-1}(t))\longrightarrow L_{k-1},
\qquad
G(\gamma_{k}(t))\longrightarrow L_{k}.
\]
Set
\[
G_{k-1}(t):=G(\gamma_{k-1}(t)),\qquad
G_{k}(t):=G(\gamma_{k}(t)).
\]
Join \(\gamma_{k-1}(t)\) and \(\gamma_{k}(t)\) by the circular arc
\[
c_{t,\varepsilon}(s)= t e^{\sqrt{-1}s},\qquad s\in[\theta_k-\varepsilon,\theta_k+\varepsilon].
\]

To capture the jump across \(\theta_k\), we study the relative transport along this arc.  Define
\[
g_t(s)=G\bigl(c_{t,\varepsilon}(\theta_k-\varepsilon)\bigr)^{-1}
G\bigl(c_{t,\varepsilon}(s)\bigr),\qquad s\in[\theta_k-\varepsilon,\theta_k+\varepsilon].
\]
Then \(g_t(\theta_k-\varepsilon)=I\).  Differentiating with respect to \(s\) and using \(\odif G = G\Theta\) gives
\[
g_t(s)^{-1}g_t'(s)
= \Theta\bigl(c_{t,\varepsilon}(s)\bigr)
c_{t,\varepsilon}'(s).
\]
Evaluating at the right endpoint yields
\[
g_t(\theta_k+\varepsilon) = G_{k-1}(t)^{-1}G_{k}(t).
\]

As \(s\) crosses the Stokes direction \(\theta_k\), the one-form \(\Theta\) is no longer exponentially small.  We now compute it explicitly.  Recall that 
\[
\Theta = S\exp(D(w))(R_1\odif w+R_2\odif{\bar w})\exp(-D(w))S^{-1}.
\]
Hence on the arc \(w = t e^{\sqrt{-1}s}\), we have
\[
\Theta\bigl(c_{t,\varepsilon}(s)\bigr)c_{t,\varepsilon}'(s)
= Se^{D(te^{\sqrt{-1}s})}
\bigl(R_1\sqrt{-1}te^{\sqrt{-1}s} - R_2\sqrt{-1}te^{-\sqrt{-1}s}\bigr)
e^{-D(te^{\sqrt{-1}s})}S^{-1}\odif s .
\]
Set
\begin{equation}\label{eq:formula-for-R}
    R(te^{\sqrt{-1}s})=R_1\sqrt{-1}te^{\sqrt{-1}s}-R_2\sqrt{-1}te^{-\sqrt{-1}s}.
\end{equation}
By \eqref{eq:cyclic-form-for-S-and-D} we have
\[
\bigl(D(te^{\sqrt{-1}s})\bigr)_{i,i}=2t\cos\Bigl(s-\frac{(i-1)\pi}{2}\Bigr).
\]
Hence
\begin{align*}
    \bigl(e^{D(te^{\sqrt{-1}s})}R(te^{\sqrt{-1}s})e^{-D(te^{\sqrt{-1}s})}\bigr)_{i,j}
    = \exp\Bigl(2t\Bigl[\cos\Bigl(s-\frac{(i-1)\pi}{2}\Bigr)-\cos\Bigl(s-\frac{(j-1)\pi}{2}\Bigr)\Bigr]\Bigr)
      R(te^{\sqrt{-1}s})_{i,j}.
\end{align*}
Set
   \begin{equation}\label{eq:definition-for-lambda}
     \lambda_{i,j}(s)
     = 2\Bigl[\cos\Bigl(s-\frac{(i-1)\pi}{2}\Bigr)-\cos\Bigl(s-\frac{(j-1)\pi}{2}\Bigr)\Bigr] 
     = 4\sin\Bigl(\frac{(i-j)\pi}{4}\Bigr)
        \sin\Bigl(s-\frac{(i+j-2)\pi}{4}\Bigr).
   \end{equation}
By Proposition~\ref{prop:TW-error-estimate} and \eqref{eq:formula-for-R}, for \(i\neq j\)
\begin{equation}\label{eq:estiamte-for-R}
  R(te^{\sqrt{-1}s})_{i,j}=
  O\left(
  \exp\Bigl(-2\bigl|1-(\sqrt{-1})^{i-j}\bigr|t\Bigr)\sqrt{t}
  \right),
\end{equation}
while the diagonal entries satisfy
\[
R(te^{\sqrt{-1}s})_{i,i}=
o\bigl(e^{-2\sqrt{2}t}\sqrt{t}\bigr).
\]
Since conjugation by a diagonal matrix leaves diagonal entries unchanged,
\[
\bigl(e^{D(te^{\sqrt{-1}s})}R(te^{\sqrt{-1}s})e^{-D(te^{\sqrt{-1}s})}\bigr)_{i,i}
=R(te^{\sqrt{-1}s})_{i,i}
=o\bigl(e^{-2\sqrt{2}t}\sqrt{t}\bigr).
\]

We now examine the off-diagonal entries.  Note that
\[
4\Bigl|\sin\Bigl(\frac{(i-j)\pi}{4}\Bigr)\Bigr|
=2\bigl|1-(\sqrt{-1})^{i-j}\bigr|,
\]
and for \(i\neq j\)
\[
4\Bigl|\sin\Bigl(\frac{(i-j)\pi}{4}\Bigr)\Bigr|
=
\begin{dcases}
\;\;4\sin\Bigl(\frac{(i-j)\pi}{4}\Bigr), & 1\geq i>j \geq 4,\\
-4\sin\Bigl(\frac{(i-j)\pi}{4}\Bigr), & 1 \leq i<j\leq 4 .
\end{dcases}
\]
Combining with \eqref{eq:definition-for-lambda} and \eqref{eq:estiamte-for-R}, the off-diagonal entry
\[
\bigl(e^{D(te^{\sqrt{-1}s})}R(te^{\sqrt{-1}s})e^{-D(te^{\sqrt{-1}s})}\bigr)_{i,j}=e^{t\lambda_{i,j}(s)}R(te^{\sqrt{-1}s})_{i,j}
\]
is not exponentially small precisely when the corresponding one of the following two equalities holds for some integer \(m\):
  \begin{equation}\label{eq:non-decay-condition}
      \begin{dcases}
      s-\frac{(i+j-2)\pi}{4}= \frac{\pi}{2}+2\pi m, & \text{if } i>j,\\
      s-\frac{(i+j-2)\pi}{4}=-\frac{\pi}{2}+2\pi m, & \text{if } i<j .
      \end{dcases}
  \end{equation}
In that case the entry is of order \(O(\sqrt{t})\), otherwise it decays exponentially as \(t\to\infty\). In summarize, this behavior can only occur when \(s=-\frac{\pi}{2}+\frac{m\pi}{4}\) for some integer \(m\), namely, equals one of the Stokes directions.  The positions that are not exponentially small at each Stokes direction are listed in Table~\ref{tab:non-decay-positions}.

\begin{table}[htbp]
\centering
\begin{tabular}{@{\hspace{2em}}r@{=}l|c}
\toprule
\multicolumn{2}{c|}{Stokes direction \(\theta_i\)} & Non-decaying entries \((i,j)\) \\
\(\theta_1\) & \(-\dfrac{\pi}{4}\) & \((4,3),\;(1,2)\) \\[4pt]
\(\theta_2\) & \(0\)               & \((1,3)\) \\[4pt]
\(\theta_3\) & \(\dfrac{\pi}{4}\)  & \((1,4),\;(2,3)\) \\
\bottomrule
\end{tabular}
\caption{Positions \((i,j)\) at which the error entry is not exponentially small at each Stokes direction.}
\label{tab:non-decay-positions}
\end{table}

We now given the precise formula for the unipotent matrices \(U_k\) in terms of the error one-form \(\Theta\).  
Along the arc \(c_{t,\varepsilon}(s)=t e^{\sqrt{-1}s}\) across a Stokes
direction, the error one-form decomposes as
\[
\Theta(c_{t,\varepsilon}(s)) = \Theta_0(c_{t,\varepsilon}(s)) + \sum_{(i,j)\in\mathcal{I}_k} \mu_t^{(i,j)}(s)S E_{i,j} S^{-1},
\]
where \(\Theta_0(c_{t,\varepsilon}(s)) = O(e^{-\alpha t})\) for some \(\alpha>0\),
\(\mathcal{I}_k\) is the set of non-decaying index pairs at the Stokes
direction \(\theta_k\) listed in Table~\ref{tab:non-decay-positions},
\(E_{i,j}\) is the elementary matrix with a single non-zero entry at
position \((i,j)\), and
\[
\mu_t^{(i,j)}(s)=\bigl(e^{D(te^{\sqrt{-1}s})}R(te^{\sqrt{-1}s})e^{-D(te^{\sqrt{-1}s})}\bigr)_{i,j}
               =e^{t\lambda_{i,j}(s)}R(te^{\sqrt{-1}s})_{i,j}.
\]

Each \(\mu_t^{(i,j)}(s)\) admits a Gaussian upper bound,
rescaled so that its integral over \(\mathbb{R}\) is independent of \(t\).
Hence it is uniformly absolutely integrable over
\(s\in[\theta_k-\varepsilon,\theta_k+\varepsilon]\) (see
\cite[Lemma~3.6]{tamburelli2024planar}).  By \cite[Lemma~B.2]{dumas2015polynomial},
we obtain
\[
\left|g_t(\theta_k+\varepsilon) - S\exp\Bigl(\sum_{(i,j)\in\mathcal{I}_k} \int_{\theta_k-\varepsilon}^{\theta_k+\varepsilon}\mu_t^{(i,j)}(s)E_{i,j}\Bigr) S^{-1}\right|
\longrightarrow 0 \qquad\text{as}\qquad t\to+\infty .
\]

Taking the limits \(t\to+\infty\) and then \(\varepsilon\to0^+\) yields
\[
L_{k-1}^{-1}L_k = S U_k S^{-1},
\]
with
\[
U_k = \lim_{t\to+\infty}
      \exp\Bigl(\sum_{(i,j)\in\mathcal{I}_k} \int_{\theta_k-\varepsilon}^{\theta_k+\varepsilon}\mu_t^{(i,j)}(s)\odif{s}E_{i,j}\Bigr)
    = I + \sum_{(i,j)\in\mathcal{I}_k}
      \lim_{t\to+\infty}
      \int_{\theta_k-\varepsilon}^{\theta_k+\varepsilon}\mu_t^{(i,j)}(s)\odif{s}E_{i,j} .
\]
The last equality uses that the summands are unipotent and that distinct elementary matrices in the sum commute. This
completes the proof.
\end{proof}

For later explicit computation, we summarize the structure of the matrices
\(U_k\) that has already been established in the proof of
Theorem~\ref{thm:stokes-jumps}.

\begin{proposition}
\label{prop:stokes-factor-structure}
The unipotent matrices \(U_1,U_2,U_3\) in Theorem~\ref{thm:stokes-jumps}
satisfy the following properties.
\begin{enumerate}
    \item Their non-zero off-diagonal entries are exactly those listed in
    Table~\ref{tab:non-decay-positions}.  Concretely,
    \[
    U_1 = I + c_{4,3} E_{4,3} + c_{1,2} E_{1,2},\qquad
    U_2 = I + c_{1,3} E_{1,3},\qquad
    U_3 = I + c_{1,4} E_{1,4} + c_{2,3} E_{2,3}.
    \]
    \item Each constant is given by the limit
    \[
    c_{i,j}=\lim_{t\to+\infty}\int_{\theta_k-\varepsilon}^{\theta_k+\varepsilon}\mu_t^{(i,j)}(s)\odif s
    =\lim_{t\to+\infty}\int_{\theta_k-\varepsilon}^{\theta_k+\varepsilon}
    e^{t\lambda_{i,j}(s)}R(te^{\sqrt{-1}s})_{i,j}\odif s
    \]
    for the corresponding index pair \((i,j)\) and Stokes direction
    \(\theta_k\) appearing in the definition of the respective constant. Here
    \[
    \lambda_{i,j}(s)= 4\sin\Bigl(\frac{(i-j)\pi}{4}\Bigr)\sin\Bigl(s-\frac{(i+j-2)\pi}{4}\Bigr).
    \]
\end{enumerate}
\end{proposition}

\subsection{Explicit computation of the Stokes constants for \texorpdfstring{\(q=z^d\odif{z}^4\)}{q}}\label{subsec:explicit-computation-for-z^d}

For the rest of the paper we work exclusively with the monomial model
\[
q = z^d\odif{z}^4,
\]
and all subsequent computations are performed on this differential.

Set \(N = d+4\). For \(j\in\mathbb Z/N\mathbb Z\), define the sector
\[
U_j = \bigl\{ z = r e^{\sqrt{-1}\theta} : r>0,
          \tfrac{2\pi(j-1)}{N} < \theta < \tfrac{2\pi(j+1)}{N} \bigr\},
\]
centered along the ray \(\ell_j = e^{2\pi\sqrt{-1}j/N}\mathbb R_{>0}\).
Using the principal branch of the logarithm, let
\[
w_j(z) = \frac{4}{N} \bigl( z e^{-\frac{2\pi j\sqrt{-1}}{N}} \bigr)^{N/4}.
\]
A direct computation shows
\[
\odif{w}_j^4 = z^d\odif{z}^4 = q,
\]
and \(w_j : U_j \to \mathbb{H}\) is biholomorphic.  Thus each
\((U_j,w_j)\) is a \(q\)-right-half-plane satisfying the properties of
Theorem~\ref{thm:standard-half-plane-decomposition} with all \(c_j = 0\) and the compact set \(K_0=\{0\}\).
In particular, on the overlap
\[
U_j \cap U_{j+1}
   = \bigl\{ z : \tfrac{2\pi j}{N} < \theta < \tfrac{2\pi(j+1)}{N} \bigr\},
\]
the transition law simplifies to
\[
w_{j+1} = -\sqrt{-1} w_j .
\]

\subsubsection{Radial symmetry and reduction to a single half-plane}

With the decomposition of \(\mathbb{C}\) into the sectors \(U_j\) in
place, we now study the harmonic metric on each of
them.  For \(q = z^d\odif{z}^4\), the elliptic
system~\eqref{eq:planar-hitchin-equation} takes the form
\begin{equation}
\label{eq:planar-hitchin-equation-for-z^d}
\begin{dcases}
\Delta u_1 = |z^d|^2 e^{2u_1} - e^{u_2-u_1},\\
\Delta u_2 = e^{u_2-u_1} - e^{-2u_2}.
\end{dcases}
\end{equation}
Here \(\Delta=\partial_z\partial_{\bar{z}}\). Since \(\bigl|(e^{\sqrt{-1}\theta}z)^d\bigr|=|z^d|\) for any angle
\(\theta\), the right-hand side depends only on \(|z|\). By the
uniqueness of the complete solution
(Theorem~\ref{thm:general-planar-theory}), the functions \(u_1,u_2\)
are radial:
\[
u_1(z)=u_1(|z|),\qquad u_2(z)=u_2(|z|).
\]
Consequently, the functions
\begin{equation}\label{eq:definition-for-tilde-u-z^d}
    \tilde u_1 = u_1 + \tfrac34\log|q| = u_1 + \tfrac34\log|z^d|,
    \qquad
    \tilde u_2 = u_2 + \tfrac14\log|q| = u_2 + \tfrac14\log|z^d|
\end{equation}
introduced in \eqref{eq:connection-form-on-right-half-plane} are also
radial in \(z\).

Each sector \(U_j\) is obtained from \(U_0\) by the rotation
\(z \mapsto e^{2\pi\sqrt{-1}j/N}z\), while \(\tilde u_1,\tilde u_2\)
depend only on \(|z|\) and are therefore invariant under this rotation.
Hence \(\tilde u_1,\tilde u_2\) are identical on every \(U_j\), and the
Stokes data on each right-half-plane \((U_j,w_j)\) coincide.  It
therefore suffices to compute them on a single half-plane.

We fix the sector \(U_0\) and write \(U = U_0\), \(w = w_0\) for
brevity.  The origin \(z=0\) is the unique zero of \(q\), of order
\(d\).  The
natural coordinate \(w\) sends this singularity to \(w=0\), while the
interior of \(U\) is mapped to the right half-plane
\(\mathbb{H} = \{\operatorname{Re} w > 0\}\).

Let us write the relation between the two coordinate systems.  With
\(z = r_z e^{\sqrt{-1}\theta_z}\) and \(w = r_w e^{\sqrt{-1}\theta_w}\),
the formula \(w = \frac{4}{N} z^{N/4}\) (\(N = d+4\)) yields
\[
r_w = \frac{4}{N} r_z^{N/4},\qquad \theta_w = \frac{N}{4}\theta_z .
\]
Since \(\tilde u_1,\tilde u_2\) depend only on \(r_z\) and \(r_z\) is a
function of \(r_w\) alone, they depend only on \(r_w\) on \(U\):
\[
\tilde u_1 = \tilde u_1(r_w),\qquad \tilde u_2 = \tilde u_2(r_w).
\]

Using the expressions above,
\[
\tilde u_1(r_w) = u_1(r_z) + \frac{3d}{N}\log r_w + \text{constant},
\qquad
\tilde u_2(r_w) = u_2(r_z) + \frac{d}{N}\log r_w + \text{constant},
\]
where \(r_z = (\frac{N}{4}r_w)^{4/N}\).  By
Theorem~\ref{thm:general-planar-theory}, \(u_1,u_2\) are bounded in a
neighborhood of the origin, so the logarithmic terms dominate the
singular behavior of \(\tilde u_1,\tilde u_2\) as \(r_w \to 0\):
\[
\tilde u_1(r_w) \sim \frac{3d}{N}\log r_w,\qquad
\tilde u_2(r_w) \sim \frac{d}{N}\log r_w .
\]
Their asymptotic behavior as \(r_w = |w|\to\infty\)  is what will determine the Stokes constants.

Thus, on the half-plane \(U\), \(\tilde u_1,\tilde u_2\) are smooth radial
solutions of
\begin{equation}
\label{eq:tilde-u-equation}
\begin{dcases}
\Delta_w \tilde u_1 = e^{2\tilde u_1} - e^{\tilde u_2-\tilde u_1},\\
\Delta_w \tilde u_2 = e^{\tilde u_2-\tilde u_1} - e^{-2\tilde u_2},
\end{dcases}
\end{equation}
with logarithmic singularities at the origin as above. Here \(\Delta_w=\partial_w\partial_{\bar{w}}\). 

The radial solutions of this elliptic system on \(\mathbb{C}^*\) were studied in
\cite{guestlin2026110730}.  Following that work, we regard
\(\tilde u_1,\tilde u_2\) as the smooth radial solution on
\(\mathbb{C}^*\) of \eqref{eq:tilde-u-equation}, with the above
singular behavior at the origin.
\subsubsection{Computation of the Stokes constants}

In order to evaluate the limits appearing in
Lemma~\ref{prop:stokes-factor-structure}, we need the asymptotic
expansion of \(\tilde u_1,\tilde u_2\) as \(r_w\to\infty\).  The
necessary result is provided by \cite{guestlin2026110730}.  We now
state the asymptotics that will be used in our computation.

\begin{theorem}[{\cite[Theorem~B.4 and Corollary~8.14]{guestlin2026110730}}]
\label{thm:asymptotics-radial}
Consider the \(\mathrm{tt}^*\)-Toda equations of type \(A_3\):
\begin{equation}
\label{eq:w-equation}
2\Delta w_i = e^{2(w_i-w_{i-1})} - e^{2(w_{i+1}-w_{i})},
\qquad w_i : \mathbb{C}^* \to \mathbb{R},
\end{equation}
where the index \(i\) runs over \(\mathbb{Z}/4\mathbb{Z}\) and
\(\Delta = \partial_z \partial_{\bar{z}}\).  The solutions are required
to satisfy the symmetry condition \(w_i + w_{3-i} = 0\).  For every
point \(m = (m_0,m_1,m_2,m_3)\) in the polytope
\[
\mathcal{A} = \{ m\in \mathbb{R}^4 :
m_i - m_{i-1} > -1,\; m_i + m_{3-i} = 0 \},
\]
there exist unique smooth radial solutions \(w_i\) such that
\[
w_i(x) \sim -m_i \log x \qquad \text{as } x \to 0,
\]
where \(x = |z|\).  Moreover, for \(j = 1,2\) the following asymptotic
expansion holds as \(x \to +\infty\):
\[
\begin{aligned}
w_0(x) \sin\Bigl(\frac{j\pi}{4}\Bigr)
 + w_1(x) \sin\Bigl(\frac{3j\pi}{4}\Bigr)
 = -\frac{1}{2} s_j (\pi L_j x)^{-\frac{1}{2}}
   e^{-2L_j x}
   + O\Bigl(x^{-\frac{3}{2}} e^{-2L_j x}\Bigr),
\end{aligned}
\]
where \(L_j = 2\sin(j\pi/4)\) and \(s_j\) denotes the \(j\)-th
elementary symmetric function of the four numbers
\((\sqrt{-1})^{m_0+3/2}\), \((\sqrt{-1})^{m_1+1/2}\),
\((\sqrt{-1})^{m_2-1/2}\), \((\sqrt{-1})^{m_3-3/2}\).
\end{theorem}

\begin{remark}
The results of \cite{guestlin2026110730} are formulated for the
\(\mathrm{tt}^*\)-Toda equations of type \(A_{n}\) for general \(n\).
For our purposes it suffices to state the specialization to
\(A_{3}\), which corresponds to the \(\mathrm{Sp}(4,\mathbb{R})\)
setting studied in this paper.
\end{remark}

With Theorem~\ref{thm:asymptotics-radial} in hand, we can now obtain the
asymptotic behavior of our functions \(\tilde u_1,\tilde u_2\) at
infinity.  To do so we simply need to identify the correct parameter
\(m\) that matches the singular behavior of \(\tilde u_1,\tilde u_2\)
at the origin, as determined in the previous subsection.  The result
is the following.

\begin{lemma}[Asymptotics of \(\tilde u_1,\tilde u_2\)]
\label{lem:asymptotics-tilde-u}
On the half-plane \((U,w)\), let \(r = |w|\).  For \(j = 1,2\), as
\(r \to +\infty\),
\[
\begin{aligned}
\tilde u_1(r) \sin\Bigl(\frac{j\pi}{4}\Bigr)
 + \tilde u_2(r) \sin\Bigl(\frac{3j\pi}{4}\Bigr)
 = - s_j (\pi L_j r)^{-\frac{1}{2}}
   e^{-2L_j r}
   + O\Bigl(r^{-\frac{3}{2}} e^{-2L_j r}\Bigr),
\end{aligned}
\]
where \(L_j = 2\sin(j\pi/4)\) and
\[
s_j = \prod_{i=1}^{j}\frac{\sin\bigl(\frac{(5-i)\pi}{N}\bigr)}
                         {\sin\bigl(\frac{i\pi}{N}\bigr)} .
\]
\end{lemma}

\begin{proof}
The functions \(\tilde u_1,\tilde u_2\) satisfy the same elliptic system
\eqref{eq:tilde-u-equation} as \(2w_0,2w_1\) in
\eqref{eq:w-equation}, and the symmetry condition
\(w_i+w_{3-i}=0\) matches the relations
\(\tilde u_3=-\tilde u_2,\tilde u_4=-\tilde u_1\) that hold on our
half-plane.  We may therefore identify
\[
\tilde u_1 = 2w_0,\qquad \tilde u_2 = 2w_1.
\]

Recall from the previous subsection that
\[
\tilde u_1(r_w) \sim \frac{3d}{N}\log r_w,\qquad
\tilde u_2(r_w) \sim \frac{d}{N}\log r_w
\]
as \(r_w \to 0\) in the right-half-plane, where \(N = d+4\).
Hence as \(x \to 0\),
\[
w_0(x) \sim \frac{3d}{2N}\log x,\qquad
w_1(x) \sim \frac{d}{2N}\log x,
\]
which yields
\[
m_0 = -\frac{3d}{2N},\qquad m_1 = -\frac{d}{2N}.
\]
By the symmetry condition \(m_i + m_{3-i} = 0\), we obtain
\[
m_i = \frac{(2i-3)d}{2N},\qquad i = 0,1,2,3.
\]
One readily verifies that these values satisfy the inequalities
\(m_i - m_{i-1} > -1\) required in Theorem~\ref{thm:asymptotics-radial}.

Thus we only need to compute the symmetric functions \(s_j\) for
\(j = 1,2\).  Note that \(s_j\) is the \(j\)-th elementary symmetric
function of \((\sqrt{-1})^{\alpha_i}\) for \(i = 0,1,2,3\), where
\[
\alpha_i = \frac{(2i-3)d}{2N} + \frac{3}{2} - i = \frac{6-4i}{N}.
\]

To evaluate these symmetric functions we use the following identity.

\begin{lemma}\label{lem:q-formula-for-symmetric-polynomial}
Let \(e_j\) be the \(j\)-th elementary symmetric function of \(n\)
variables \(x_1,\dots,x_n\), and let \(q\) be a parameter.  Then
\[
e_j(1,q,\dots,q^{n-1})
   = q^{j(j-1)/2} \prod_{i=1}^{j} \frac{1-q^{n+1-i}}{1-q^{i}} .
\]
\end{lemma}

\begin{proof}
Let \(y\) be a formal variable.  The generating function for the
elementary symmetric functions is
\[
\sum_{j=0}^{\infty} e_j y^j = \prod_{k=1}^{n} (1 + x_k y).
\]
Hence
\[
\sum_{j=0}^{\infty} e_j(1,q,\dots,q^{n-1}) y^j
   = \prod_{k=0}^{n-1} (1 + q^{k} y).
\]
By the \(q\)-binomial theorem (see \cite[Chapter~5]{kac2002quantum}),
\[
\prod_{k=0}^{n-1} (1 + q^{k} y)
   = \sum_{j=0}^{n-1} q^{j(j-1)/2} \binom{n}{j}_{q} y^j,
\]
where
\[
\binom{n}{j}_{q}
   = \frac{\prod_{i=1}^{j} (1-q^{n+1-i})}
          {\prod_{i=1}^{j} (1-q^{i})} .
\]
Comparing coefficients yields the claimed formula.
\end{proof}

Set \(\xi = (\sqrt{-1})^{-4/N} = e^{-2\pi\sqrt{-1}/N}\).  Then
\[
1-\xi^{i}
   = \xi^{i/2}\bigl(\xi^{-i/2} - \xi^{i/2}\bigr)
   = \xi^{i/2}\cdot 2\sqrt{-1}\sin\Bigl(\frac{i\pi}{N}\Bigr).
\]
Applying Lemma~\ref{lem:q-formula-for-symmetric-polynomial} with
\(n = 4\) and \(q = \xi\), we obtain
\begin{align*}
e_j(1,\xi,\xi^2,\xi^3)
  &= \xi^{j(j-1)/2}
     \prod_{i=1}^{j} \frac{1-\xi^{5-i}}{1-\xi^{i}} \\[2mm]
  &= \xi^{j(j-1)/2}
     \prod_{i=1}^{j}
     \frac{\xi^{(5-i)/2}\sin\bigl(\frac{(5-i)\pi}{N}\bigr)}
          {\xi^{i/2}\sin\bigl(\frac{i\pi}{N}\bigr)} \\[2mm]
  &= \xi^{j(j-1)/2}
     \prod_{i=1}^{j}
     \xi^{5/2-i}
     \frac{\sin\bigl(\frac{(5-i)\pi}{N}\bigr)}
          {\sin\bigl(\frac{i\pi}{N}\bigr)} \\[2mm]
  &= \xi^{j(j-1)/2}
     \xi^{5j/2 - j(j+1)/2}
     \prod_{i=1}^{j}
     \frac{\sin\bigl(\frac{(5-i)\pi}{N}\bigr)}
          {\sin\bigl(\frac{i\pi}{N}\bigr)} \\[2mm]
  &= \xi^{3j/2}
     \prod_{i=1}^{j}
     \frac{\sin\bigl(\frac{(5-i)\pi}{N}\bigr)}
          {\sin\bigl(\frac{i\pi}{N}\bigr)} .
\end{align*}
Consequently,
\begin{align*}
s_j
  &= (\sqrt{-1})^{6j/N}
     e_j(1,\xi,\xi^2,\xi^3) \\
  &= (\sqrt{-1})^{6j/N}
     \xi^{3j/2}
     \prod_{i=1}^{j}
     \frac{\sin\bigl(\frac{(5-i)\pi}{N}\bigr)}
          {\sin\bigl(\frac{i\pi}{N}\bigr)} \\
  &= \prod_{i=1}^{j}
     \frac{\sin\bigl(\frac{(5-i)\pi}{N}\bigr)}
          {\sin\bigl(\frac{i\pi}{N}\bigr)} .
\end{align*}
This completes the proof of Lemma~\ref{lem:asymptotics-tilde-u}.
\end{proof}

For later use we introduce the notation \(\theta = \pi/N\) with
\(N = d+4\).  Then
\[
s_1 = \frac{\sin(4\theta)}{\sin\theta},\qquad
s_2 = \frac{\sin(4\theta)\sin(3\theta)}{\sin\theta\sin(2\theta)} .
\]

We now turn to the derivative of the asymptotic expansion.

\begin{lemma}
\label{lem:derivative-asymptotics}
As \(r \to +\infty\),
\[
\tilde u_1'(r) \sin\Bigl(\frac{j\pi}{4}\Bigr)
 + \tilde u_2'(r) \sin\Bigl(\frac{3j\pi}{4}\Bigr)
 \sim 2L_j s_j (\pi L_j r)^{-\frac{1}{2}} e^{-2L_j r} .
\]
\end{lemma}

\begin{proof}
Set
\begin{align*}
v_1(r) &= \tilde u_1(r) \sin\Bigl(\frac{\pi}{4}\Bigr)
        + \tilde u_2(r) \sin\Bigl(\frac{3\pi}{4}\Bigr)
        = \frac{\sqrt{2}}{2}\bigl(\tilde u_1(r) + \tilde u_2(r)\bigr), \\[2mm]
v_2(r) &= \tilde u_1(r) \sin\Bigl(\frac{2\pi}{4}\Bigr)
        + \tilde u_2(r) \sin\Bigl(\frac{6\pi}{4}\Bigr)
        = \tilde u_1(r) - \tilde u_2(r),
\end{align*}
and let \(A_j(r) = (\pi L_j r)^{-\frac{1}{2}} e^{-2L_j r}\).
By Lemma~\ref{lem:asymptotics-tilde-u} we have
\(v_j(r) \sim -s_j A_j(r) \) as \(r \to +\infty\).  

Linearizing the system \eqref{eq:tilde-u-equation} around
\((\tilde u_1,\tilde u_2) = (0,0)\) gives
\[
\Delta v_j = L_j^2 v_j + N_j,
\]
where the nonlinear remainder satisfies
\(N_j = O(v_1^2 + v_2^2)\).  Because
\(v_1 = O(A_1)\) and \(v_2 = O(A_2)\), and
\(A_1^2, A_2^2 = o(A_j)\) for both\(j = 1,2\), we obtain
\(N_j = o(A_j)\).  Therefore
\[
\Delta v_j = L_j^2 v_j + o(A_j).
\]

Standard interior elliptic estimates give the rough bound
\(v_j'(r) = O(A_j(r))\).  Since the \(v_j\) are radial and our
convention is \(\Delta = \partial_w \partial_{\bar w}\), we have
\[
\Delta v_j = \frac14\Bigl(v_j'' + \frac{1}{r} v_j'\Bigr).
\]
Thus
\[
v_j'' + \frac{1}{r} v_j' = 4L_j^2 v_j + o(A_j).
\]
Using \(v_j \sim -s_j A_j\) and
\(\frac{1}{r} v_j' = O(A_j/r) = o(A_j)\), we deduce
\[
v_j'' \sim -4L_j^2 s_j A_j .
\]

On the other hand, differentiating \(A_j\) gives
\[
A_j'(r) = \Bigl(-2L_j - \frac{1}{2r}\Bigr) A_j(r),
\qquad
A_j''(r) \sim 4L_j^2 A_j(r).
\]
Hence
\[
\frac{v_j''(r)}{A_j''(r)} \longrightarrow -s_j .
\]

Since both \(v_j'(r)\) and \(A_j'(r)\) tend to zero as \(r\to\infty\),
l'Hôpital's rule yields
\[
\frac{v_j'(r)}{A_j'(r)} \longrightarrow -s_j .
\]
Finally, because \(A_j'(r) \sim -2L_j A_j(r)\), we conclude that
\[
v_j'(r) \sim 2L_j s_j A_j(r)
       = 2L_j s_j (\pi L_j r)^{-\frac{1}{2}} e^{-2L_j r},
\]
which is exactly the stated formula.
\end{proof}

We are now ready to compute the asymptotics of the entries of
\(R_1\) and \(R_2\) along a ray, which will be used to determine the Stokes constants.

\begin{lemma}
\label{lem:asymptotics-R1-R2}
Let \(t \to +\infty\).  The convergence in the following asymptotic
formulas is uniform in \(s\) for \(s \in (-\pi/2,\pi/2)\).  For
\(j > i\),
\begin{align*}
(R_1)_{i,j}(t e^{\sqrt{-1}s})
  &\sim (-1)^{j-i}\frac{\sqrt{-1}}{4} e^{-\sqrt{-1}s}
     2L_{j-i} s_{j-i} (\pi L_{j-i} t)^{-\frac{1}{2}}
     e^{-2L_{j-i} t}, \\
(R_2)_{i,j}(t e^{\sqrt{-1}s})
  &\sim (-1)^{j-i}\frac{\sqrt{-1}^{j} - \sqrt{-1}^{i}}{2}
     s_{j-i} (\pi L_{j-i} t)^{-\frac{1}{2}}
     e^{-2L_{j-i} t}.
\end{align*}
For \(i > j\),
\begin{align*}
(R_1)_{i,j}(t e^{\sqrt{-1}s})
  &\sim (-1)^{j-i+1}\frac{\sqrt{-1}}{4} e^{-\sqrt{-1}s}
     2L_{i-j} s_{i-j} (\pi L_{i-j} t)^{-\frac{1}{2}}
     e^{-2L_{i-j} t}, \\
(R_2)_{i,j}(t e^{\sqrt{-1}s})
  &\sim (-1)^{j-i}\frac{\sqrt{-1}^{i} - \sqrt{-1}^{j}}{2}
     s_{i-j} (\pi L_{i-j} t)^{-\frac{1}{2}}
     e^{-2L_{i-j} t}.
\end{align*}
Here we define \(s_3=-s_1\).
\end{lemma}
\begin{proof}
Recall the explicit formulas for \(R_1\) and \(R_2\) from
\eqref{eq:explict-formula-for-R1} and
\eqref{eq:explict-formula-for-R2-old}.
Since \(\tilde u_1,\tilde u_2\) are radial, we have
\[
\partial_w \tilde u_i(r e^{\sqrt{-1}\theta})
   = \frac{1}{2} e^{-\sqrt{-1}\theta} \tilde u_i'(r),
\]
and the asymptotic estimates that follow depend only on \(r = |w|\),
hence are uniform in \(s\) on \((-\pi/2,\pi/2)\).

We begin with \(R_1\).  As \(t \to +\infty\), when \(j > i\),
\begin{align*}
(R_1)_{i,j}(t e^{\sqrt{-1}s})
  &= (-1)^{j-i}\frac{\sqrt{-1}}{4} e^{-\sqrt{-1}s}
     \Bigl(
        \sin\Bigl(\frac{\pi(j-i)}{4}\Bigr)\tilde u_1'(t)
       +\sin\Bigl(\frac{3\pi(j-i)}{4}\Bigr)\tilde u_2'(t)
     \Bigr) \\[2mm]
  &\sim (-1)^{j-i}\frac{\sqrt{-1}}{4} e^{-\sqrt{-1}s}
     2L_{j-i} s_{j-i} (\pi L_{j-i} t)^{-\frac{1}{2}}
     e^{-2L_{j-i} t}.
\end{align*}
When \(i > j\), using
\(\sin(\frac{\pi(j-i)}{4}) = -\sin(\frac{\pi(i-j)}{4})\) and
\(\sin(\frac{3\pi(j-i)}{4}) = -\sin(\frac{3\pi(i-j)}{4})\), we obtain
\begin{align*}
(R_1)_{i,j}(t e^{\sqrt{-1}s})
  &= (-1)^{j-i}\frac{\sqrt{-1}}{4} e^{-\sqrt{-1}s}
     \Bigl(
        \sin\Bigl(\frac{\pi(j-i)}{4}\Bigr)\tilde u_1'(t)
       +\sin\Bigl(\frac{3\pi(j-i)}{4}\Bigr)\tilde u_2'(t)
     \Bigr) \\
  &= (-1)^{j-i+1}\frac{\sqrt{-1}}{4} e^{-\sqrt{-1}s}
     \Bigl(
        \sin\Bigl(\frac{\pi(i-j)}{4}\Bigr)\tilde u_1'(t)
       +\sin\Bigl(\frac{3\pi(i-j)}{4}\Bigr)\tilde u_2'(t)
     \Bigr) \\
  &\sim (-1)^{j-i+1}\frac{\sqrt{-1}}{4} e^{-\sqrt{-1}s}
     2L_{i-j} s_{i-j} (\pi L_{i-j} t)^{-\frac{1}{2}}
     e^{-2L_{i-j} t}.
\end{align*}

We now turn to \(R_2\).  From \eqref{eq:explict-formula-for-R2-old},
as \(t \to +\infty\) we expand the exponentials to first order and
rearrange the sum:
\allowdisplaybreaks[1]
\begin{align*}
(R_2)_{i,j}(t e^{\sqrt{-1}s})
  &= \frac14 e^{\frac{\pi\sqrt{-1}}{2}(j-1)}
     \sum_{k=1}^4
     e^{\frac{\pi\sqrt{-1}}{4}(5-2k)(i-j)}
     \bigl(e^{\tilde u_{k+1}-\tilde u_k} - 1\bigr) \\
  &\sim \frac14 e^{\frac{\pi\sqrt{-1}}{2}(j-1)}
     \sum_{k=1}^4
     e^{\frac{\pi\sqrt{-1}}{4}(5-2k)(i-j)}
     (\tilde u_{k+1} - \tilde u_k) \\
  &= \frac14 e^{\frac{\pi\sqrt{-1}}{2}(j-1)}
     \sum_{k=1}^4
     \Bigl(e^{\frac{\pi\sqrt{-1}}{4}(5-2(k-1))(i-j)}
          - e^{\frac{\pi\sqrt{-1}}{4}(5-2k)(i-j)}\Bigr)\tilde u_k \\
  &= \frac14 e^{\frac{\pi\sqrt{-1}}{2}(j-1)}
     \Bigl(e^{\frac{\pi\sqrt{-1}}{4}2(i-j)} - 1\Bigr)
     \sum_{k=1}^4
     e^{\frac{\pi\sqrt{-1}}{4}(5-2k)(i-j)}\tilde u_k .
\end{align*}
The prefactor simplifies to
\[
\frac14 e^{\frac{\pi\sqrt{-1}}{2}(j-1)}
   \Bigl(e^{\frac{\pi\sqrt{-1}}{4}2(i-j)} - 1\Bigr)
 = \frac{\sqrt{-1}^{i-1} - \sqrt{-1}^{j-1}}{4}.
\]
Hence
\[
(R_2)_{i,j}(t e^{\sqrt{-1}s})
   \sim \frac{\sqrt{-1}^{i-1} - \sqrt{-1}^{j-1}}{4}
     2\sqrt{-1}(-1)^{j-i}
     \Bigl(
        \sin\Bigl(\frac{\pi(j-i)}{4}\Bigr)\tilde u_1(t)
       +\sin\Bigl(\frac{3\pi(j-i)}{4}\Bigr)\tilde u_2(t)
     \Bigr).
\]

Inserting the asymptotics of \(\tilde u_1,\tilde u_2\) from
Lemma~\ref{lem:asymptotics-tilde-u}, we obtain the final formulas.
For \(j > i\),
\[
(R_2)_{i,j}(t e^{\sqrt{-1}s})
   \sim (-1)^{j-i}\frac{\sqrt{-1}^{j} - \sqrt{-1}^{i}}{2}
     s_{j-i} (\pi L_{j-i} t)^{-\frac{1}{2}} e^{-2L_{j-i} t},
\]
and for \(i > j\),
\[
(R_2)_{i,j}(t e^{\sqrt{-1}s})
   \sim (-1)^{j-i}\frac{\sqrt{-1}^{i} - \sqrt{-1}^{j}}{2}
     s_{i-j} (\pi L_{i-j} t)^{-\frac{1}{2}} e^{-2L_{i-j} t}.
\]
This completes the proof of Theorem~\ref{thm:stokes-jumps} and
Lemma~\ref{lem:asymptotics-R1-R2}.
\end{proof}

We are now evaluate the limits defining the Stokes
constants and obtain the explicit form of the unipotent matrices
\(U_1,U_2,U_3\).

\begin{proposition}
\label{prop:explicit-U}
With the notation of Proposition~\ref{prop:stokes-factor-structure}, the
unipotent matrices are given by
\[
U_1 =
\begin{pmatrix}
1 & s_1 & 0 & 0 \\
0 &  1  & 0 & 0 \\
0 & 0   & 1 & 0 \\
0 & 0 & -s_1 & 1
\end{pmatrix},\qquad
U_2 =
\begin{pmatrix}
1 & 0 & -s_2 & 0 \\
0 &  1  & 0 & 0 \\
0 & 0   & 1 & 0 \\
0 & 0 &  0 & 1
\end{pmatrix},\qquad
U_3 =
\begin{pmatrix}
1 & 0 & 0 & s_1 \\
0 &  1  & s_1 & 0 \\
0 & 0   & 1 & 0 \\
0 & 0 &  0 & 1
\end{pmatrix},
\]
where
\[
s_1 = \frac{\sin(4\theta)}{\sin\theta},\qquad
s_2 = \frac{\sin(4\theta)\sin(3\theta)}{\sin\theta\sin(2\theta)},
\qquad \theta = \frac{\pi}{d+4}.
\]
\end{proposition}

\begin{proof}
For brevity we present the computation only for the case \(j > i\). The remaining entries are obtained by an entirely analogous argument.
Let \(c_{i,j}\) denote the Stokes constant associated with the index
pair \((i,j)\).  By Proposition~\ref{prop:stokes-factor-structure} and the
definition of \(\mu_t^{(i,j)}(s)\),
\[
c_{i,j} = \lim_{t\to+\infty}
\int_{\theta_k-\varepsilon}^{\theta_k+\varepsilon}
e^{t\lambda_{i,j}(s)} R(t e^{\sqrt{-1}s})_{i,j} \odif s,
\]
where \(\theta_k = -\frac{\pi}{2} + \frac{k\pi}{4}\) for \(k = 1,2,3\)
and
\[
\lambda_{i,j}(s)
   = 4\sin\Bigl(\frac{(i-j)\pi}{4}\Bigr)
      \sin\Bigl(s-\frac{(i+j-2)\pi}{4}\Bigr).
\]

Recall the condition~\eqref{eq:non-decay-condition} in the proof of Theorem~\ref{thm:stokes-jumps} which characterizes the positions of non-exponentially
small off-diagonal entries.  For the index pairs that actually
contribute (see Table~\ref{tab:non-decay-positions}) we have
\[
\frac{(i+j-2)\pi}{4} = \frac{k\pi}{4} + 2m\pi
\]
for some integer \(m\), and consequently
\[
\lambda_{i,j}(s)
   = 4\sin\Bigl(\frac{(j-i)\pi}{4}\Bigr)
      \sin\Bigl(\frac{\pi}{2} + \theta_k - s\Bigr)
   = 2L_{j-i} \cos(s - \theta_k).
\]

From \eqref{eq:formula-for-R} we have
\[
R(t e^{\sqrt{-1}s})
   = R_1 \sqrt{-1}t e^{\sqrt{-1}s}
     - R_2 \sqrt{-1}t e^{-\sqrt{-1}s}.
\]
We evaluate the contribution of each term separately.

\noindent\textbf{Contribution of \(R_1\).}
Using Lemma~\ref{lem:asymptotics-R1-R2} and the Lebesgue dominated
convergence theorem,
\begin{align*}
&\lim_{t\to+\infty}
   \int_{\theta_k-\varepsilon}^{\theta_k+\varepsilon}
   e^{t\lambda_{i,j}(s)} (R_1)_{i,j} \sqrt{-1}t e^{\sqrt{-1}s}
   \odif s \\
&\qquad =
   \lim_{t\to+\infty}
   \int_{\theta_k-\varepsilon}^{\theta_k+\varepsilon}
   e^{t(2L_{j-i}\cos(s-\theta_k))}
   (-1)^{j-i}\frac{\sqrt{-1}}{4} e^{-\sqrt{-1}s}
   2L_{j-i} s_{j-i} (\pi L_{j-i} t)^{-\frac{1}{2}}
   e^{-2L_{j-i} t}
   \sqrt{-1}t e^{\sqrt{-1}s} \odif s \\
&\qquad =
   \frac{(-1)^{j-i+1}}{2}
   s_{j-i} \Bigl(\frac{L_{j-i}}{\pi}\Bigr)^{\!1/2}
   \lim_{t\to+\infty}
   \int_{\theta_k-\varepsilon}^{\theta_k+\varepsilon}
   e^{2L_{j-i}(\cos(s-\theta_k)-1)t}
   t^{1/2} \odif s .
\end{align*}
The change of variables \(x = t^{1/2}(s-\theta_k)\) transforms the
integral into
\[
\int_{-t^{1/2}\varepsilon}^{t^{1/2}\varepsilon}
   e^{2L_{j-i}(\cos(x/\sqrt{t})-1)t} \odif x .
\]
As \(t\to+\infty\), the integrand converges pointwise to
\(e^{-L_{j-i}x^2}\) and is dominated by an integrable Gaussian.
Applying dominated convergence again,
\[
\lim_{t\to+\infty}
\int_{-t^{1/2}\varepsilon}^{t^{1/2}\varepsilon}
   e^{2L_{j-i}(\cos(x/\sqrt{t})-1)t} \odif x
= \int_{\mathbb{R}} e^{-L_{j-i}x^2} \odif x
= \sqrt{\frac{\pi}{L_{j-i}}} .
\]
Hence the contribution of \(R_1\) equals
\[
\frac{(-1)^{j-i+1}}{2} s_{j-i}.
\]

\noindent\textbf{Contribution of \(R_2\).}
Proceeding in the same way,
\begin{align*}
&\lim_{t\to+\infty}
   \int_{\theta_k-\varepsilon}^{\theta_k+\varepsilon}
   e^{t\lambda_{i,j}(s)} (-R_2)_{i,j} \sqrt{-1}t e^{-\sqrt{-1}s}
   \odif s \\
&\qquad =
   \lim_{t\to+\infty}
   \int_{\theta_k-\varepsilon}^{\theta_k+\varepsilon}
   e^{t(2L_{j-i}\cos(s-\theta_k))}
   (-1)^{j-i+1}\frac{\sqrt{-1}^{j} - \sqrt{-1}^{i}}{2}
   s_{j-i} (\pi L_{j-i} t)^{-\frac{1}{2}}
   e^{-2L_{j-i} t}
   \sqrt{-1}t e^{-\sqrt{-1}s} \odif s \\
&\qquad =
   (-1)^{j-i+1}\frac{\sqrt{-1}^{j+1} - \sqrt{-1}^{i+1}}{2}
   s_{j-i} (\pi L_{j-i})^{-\frac{1}{2}}
   \lim_{t\to+\infty}
   \int_{\theta_k-\varepsilon}^{\theta_k+\varepsilon}
   e^{2L_{j-i}(\cos(s-\theta_k)-1)t}
   t^{1/2} e^{-\sqrt{-1}s} \odif s .
\end{align*}
After the same change of variables \(x = t^{1/2}(s-\theta_k)\),
the integral becomes
\[
\int_{-t^{1/2}\varepsilon}^{t^{1/2}\varepsilon}
   e^{2L_{j-i}(\cos(x/\sqrt{t})-1)t}
   e^{-\sqrt{-1}(x/\sqrt{t}+\theta_k)} \odif x .
\]
Taking the limit \(t\to+\infty\) and applying dominated convergence
yields
\[
e^{-\sqrt{-1}\theta_k} \int_{\mathbb{R}} e^{-L_{j-i}x^2} \odif x
= e^{-\sqrt{-1}\theta_k} \sqrt{\frac{\pi}{L_{j-i}}} .
\]
It remains to simplify the coefficient. Observe that
\begin{align*}
\sqrt{-1}^{j+1} - \sqrt{-1}^{i+1}
&= e^{\frac{\pi\sqrt{-1}}{2}(j+1)}
   - e^{\frac{\pi\sqrt{-1}}{2}(i+1)} \\
&= e^{\frac{\pi\sqrt{-1}}{2}\cdot\frac{i+j+2}{2}}
   \Bigl(e^{\frac{\pi\sqrt{-1}}{2}(j-i)}
        - e^{\frac{\pi\sqrt{-1}}{2}(i-j)}\Bigr) \\
&= e^{\frac{(i+j-2)\pi\sqrt{-1}}{4}}
   e^{\pi\sqrt{-1}}
   \cdot 2\sqrt{-1}
   \sin\Bigl(\frac{(j-i)\pi}{4}\Bigr) \\
&= e^{\theta_k\sqrt{-1}} L_{j-i}.
\end{align*}
Substituting this identity into the previous expression,
the contribution of \(R_2\) equals
\[
\frac{(-1)^{j-i+1}}{2} s_{j-i}.
\]

\noindent\textbf{Conclusion.}
Adding the two contributions we obtain
\[
c_{i,j} = (-1)^{j-i+1} s_{j-i}.
\]

Inspecting the sign of \((-1)^{j-i+1}\) for the individual index pairs
listed in Table~\ref{tab:non-decay-positions} and recalling the
notation of Lemma~\ref{prop:stokes-factor-structure} gives the matrices
\(U_1,U_2,U_3\) exactly as stated.
\end{proof}

\subsubsection{Circular arcs across Stokes directions} \label{subsubsec:circular-arcs-across-stokes-directions} 
We now describe how the Stokes matrices computed in this section are assembled along a circular arc that
crosses several Stokes directions. This will be used later to evaluate the local contribution around a zero point to the asymptotic holonomy.

Let
\[
q = z^d\odif{z}^4 , \qquad N = d+4 ,
\]
and let \(D_q\) be the flat connection associated with the cyclic
\(\mathrm{Sp}(4,\mathbb{R})\)-Higgs bundle
\((\mathbb{K}_{\mathrm{Sp}(4,\mathbb{R})},\phi(q))\) over \(\mathbb{C}\).
In this subsection we study the asymptotic behavior of the holonomy of
\(D_q\) along a circular arc that crosses several Stokes directions.

Recall the sectorial natural coordinates
\[
w_j(z) = \frac{4}{N} \bigl( z e^{-\frac{2\pi j\sqrt{-1}}{N}} \bigr)^{N/4},
\qquad j \in \mathbb Z/N\mathbb Z .
\]
Each \(w_j\) identifies the corresponding sector
\[
U_j = \bigl\{ z = r e^{\sqrt{-1}\theta} :
          \tfrac{2\pi(j-1)}{N} < \theta < \tfrac{2\pi(j+1)}{N} \bigr\}
\]
in the punctured \(z\)-plane with the right half-plane.  On \(U_j\), we
work with the canonical frame \(\mathcal{F}_{w_j}\) induced by the
coordinate \(w_j\) (Definition~\ref{def:frame-induced-by-the-coordinate}).
The corresponding fundamental matrices for \(q = z^d\odif{z}^4\) and for
the constant model are denoted by \(\Psi^{w_j}_q\) and \(\Psi^{w_j}_0\),
respectively. 

In the \(w_j\)-right half-plane, the four stable sectors are
\begin{alignat*}{2}
J_0^{(j)} &= \Bigl\{ -\frac{\pi}{2} < \arg w_j < -\frac{\pi}{4} \Bigr\},
\qquad &
J_1^{(j)} &= \Bigl\{ -\frac{\pi}{4} < \arg w_j < 0 \Bigr\}, \\
J_2^{(j)} &= \Bigl\{ 0 < \arg w_j < \frac{\pi}{4} \Bigr\},
\qquad &
J_3^{(j)} &= \Bigl\{ \frac{\pi}{4} < \arg w_j < \frac{\pi}{2} \Bigr\},
\end{alignat*}
see Definition~\ref{def:stokes-directions} and Figure~\ref{fig:stokes-directions}.  Adjacent sectorial coordinates
satisfy
\[
w_{j+1} = -\sqrt{-1}w_j .
\]
Hence the overlap of the two right-half-plane charts \(w_j\) and
\(w_{j+1}\) corresponds, in the \(w_j\)-coordinate, to the quadrant
\[
0 < \arg w_j < \frac{\pi}{2},
\]
which contains the two stable sectors \(J_2^{(j)}\) and \(J_3^{(j)}\).
Under the coordinate change \(w_{j+1} = -\sqrt{-1}w_j\), these are identified as
\[
J_2^{(j)} = J_0^{(j+1)}, \qquad
J_3^{(j)} = J_1^{(j+1)} .
\]
Consequently, every point not lying on a Stokes direction belongs
uniquely to a stable sector of the form \(J_0^{(j)}\) or \(J_1^{(j)}\)
for some \(j\).

When a circular arc crosses from \(J_1^{(j)}\) to \(J_0^{(j+1)}\), we
pass from the frame \(\mathcal{F}_{w_j}\) to the frame
\(\mathcal{F}_{w_{j+1}}\). The change-of-frame matrix involves a choice of lift to the
half-canonical bundle. On the overlap \(U_j\cap U_{j+1}\), we fix the
lift
\begin{equation}
\label{eq:half-differential-transition}
(\odif w_{j+1})^{1/2}
=
e^{-\frac{\pi\sqrt{-1}}{4}}
(\odif w_j)^{1/2}.
\end{equation}

Consequently, the two induced frames satisfy
\[
\mathcal F_{w_{j+1}}
=
\mathcal F_{w_j}R,
\]
where
\begin{equation}
\label{eq:definition-for-frame-change-R}
R
=
\operatorname{diag}\left(
e^{-\frac{3\pi\sqrt{-1}}{4}},
e^{-\frac{\pi\sqrt{-1}}{4}},
e^{\frac{\pi\sqrt{-1}}{4}},
e^{\frac{3\pi\sqrt{-1}}{4}}
\right).
\end{equation} 
Let \(S\) be the matrix that diagonalizes the constant model (see
\eqref{eq:definition-of-S}).  In the \(S\)-diagonal convention, the
change from \(\mathcal{F}_{w_j}S\) to \(\mathcal{F}_{w_{j+1}}S\) is
represented by
\[
Q = S^{-1}RS .
\]
A direct computation yields
\begin{equation}\label{eq:definition-for-frame-change-Q}
Q =
\begin{pmatrix}
0 & 0 & 0 & -1\\
1 & 0 & 0 & 0\\
0 & 1 & 0 & 0\\
0 & 0 & 1 & 0
\end{pmatrix},
\qquad\text{so that}\qquad
R = S Q S^{-1}.
\end{equation}
We will use this matrix whenever we pass from the sectorial frame
\(\mathcal{F}_{w_j}\) to the adjacent sectorial frame
\(\mathcal{F}_{w_{j+1}}\).

We now study the holonomy asymptotics along a large circular
arc.  Let
\[
\beta_r(\theta) = r e^{\sqrt{-1}\theta}, \qquad
\theta_0 \leq \theta \leq \theta_1,
\]
and assume that \(\theta_0\) and \(\theta_1\) are not Stokes directions.
\begin{proposition}
\label{prop:one-stokes-direction}
Let
\[
\beta_r(\theta) = r e^{\sqrt{-1}\theta}, \qquad
\theta_0 \leq \theta \leq \theta_1,
\]
be an oriented circular arc.  Assume that the initial point lies in
\(J_1^{(j)}\) and the terminal point lies in \(J_0^{(j+1)}\).  Then,
with respect to the initial frame \(\mathcal{F}_{w_j}\) and the terminal
frame \(\mathcal{F}_{w_{j+1}}\),
\[
\operatorname{Hol}(\beta_r)
   = \bigl( \Psi_0^{w_{j+1}}(\beta_r(\theta_1)) \bigr)^{-1}
     \bigl( S Q^{-1} U_2^{-1} S^{-1} + o(I) \bigr)
     \Psi_0^{w_j}(\beta_r(\theta_0))
\]
as \(r \to +\infty\).
\end{proposition}

\begin{proof}
Choose \(\theta'\in J_0^{(j+1)} = J_2^{(j)}\) with
\(\theta_0 < \theta' < \theta_1\).  Using the composition rule for
different frames \eqref{eq:different-frames-on-consecutive-segments},
the holonomy along \(\beta_r\) is
\[
\operatorname{Hol}(\beta_r)
   = \bigl( \Psi_q^{w_{j+1}}(r e^{\sqrt{-1}\theta_1}) \bigr)^{-1}
     \Psi_q^{w_{j+1}}(r e^{\sqrt{-1}\theta'})
     R^{-1}
     \bigl( \Psi_q^{w_j}(r e^{\sqrt{-1}\theta'}) \bigr)^{-1}
     \Psi_q^{w_j}(r e^{\sqrt{-1}\theta_0}) .
\]

By Theorem~\ref{thm:sectorial-limits} and
Theorem~\ref{thm:stokes-jumps}, as \(r \to +\infty\), we have
\begin{align*}
\bigl( \Psi_q^{w_j}(r e^{\sqrt{-1}\theta'}) \bigr)^{-1}
   \Psi_q^{w_j}(r e^{\sqrt{-1}\theta_0})
   &= \Psi_0^{w_j}(r e^{\sqrt{-1}\theta'})^{-1}
      \bigl( S U_2^{-1} S^{-1} + o(I) \bigr)
      \Psi_0^{w_j}(r e^{\sqrt{-1}\theta_0}), \\[1mm]
\bigl( \Psi_q^{w_{j+1}}(r e^{\sqrt{-1}\theta_1}) \bigr)^{-1}
   \Psi_q^{w_{j+1}}(r e^{\sqrt{-1}\theta'})
   &= \Psi_0^{w_{j+1}}(r e^{\sqrt{-1}\theta_1})^{-1}
      \bigl( I + o(I) \bigr)
      \Psi_0^{w_{j+1}}(r e^{\sqrt{-1}\theta'}) .
\end{align*}
Therefore 
\begin{align*}
\operatorname{Hol}(\beta_r)
   &= \Psi_0^{w_{j+1}}(r e^{\sqrt{-1}\theta_1})^{-1}
      \bigl( I + o(I) \bigr)
      \Psi_0^{w_{j+1}}(r e^{\sqrt{-1}\theta'})
      R^{-1} \\
   &\qquad \cdot
      \Psi_0^{w_j}(r e^{\sqrt{-1}\theta'})^{-1}
      \bigl( S U_2^{-1} S^{-1} + o(I) \bigr)
      \Psi_0^{w_j}(r e^{\sqrt{-1}\theta_0}) .
\end{align*}
as \(r \to +\infty\).
On the overlap of the two
adjacent natural coordinates, the matrices \(\Psi_0^{w_j}\) and
\(\Psi_0^{w_{j+1}}\) satisfy the same linear differential equation with
the same initial condition, but are expressed respectively in the frames
\(\mathcal{F}_{w_j}\) and \(\mathcal{F}_{w_{j+1}}\).  Since these frames
are related by \(\mathcal{F}_{w_{j+1}} = \mathcal{F}_{w_j} R\), uniqueness
of solutions gives
\[
\Psi_0^{w_{j+1}}(w_{j+1}(p)) = R^{-1} \Psi_0^{w_j}(w_j(p)) R ,
\qquad p \in U_{j+1} \cap U_j .
\]
We obtain
\begin{align*}
\operatorname{Hol}(\beta_r)
   &= \bigl( \Psi_0^{w_{j+1}}(r e^{\sqrt{-1}\theta_1}) \bigr)^{-1}
      \bigl( I + o(1) \bigr)
      R^{-1}
      \bigl( S U_2^{-1} S^{-1} + o(I) \bigr)
      \Psi_0^{w_j}(r e^{\sqrt{-1}\theta_0}) \\
   &= \bigl( \Psi_0^{w_{j+1}}(r e^{\sqrt{-1}\theta_1}) \bigr)^{-1}
      \bigl( S Q^{-1} U_2^{-1} S^{-1} + o(I) \bigr)
      \Psi_0^{w_j}(r e^{\sqrt{-1}\theta_0}) ,
\end{align*}
as claimed.
\end{proof}

We now consider a counterclockwise circular arc crossing \(k\)
successive sectorial charts. Iterating
Proposition~\ref{prop:one-stokes-direction} along these \(k\)
transitions yields the following general statement.

\begin{proposition}
\label{prop:large-circular-holonomy}
Let
\[
\beta_r(\theta) = r e^{\sqrt{-1}\theta}, \qquad
\theta_0 \leq \theta \leq \theta_1,
\]
be an oriented circular arc.  Suppose that the initial point belongs to
\(J_a^{(j)}\) and the terminal point belongs to \(J_b^{(j+k)}\), with
\(a,b \in \{0,1\}\).  Then, with respect to the initial frame
\(\mathcal{F}_{w_j}\) and the terminal frame \(\mathcal{F}_{w_{j+k}}\),
\[
\operatorname{Hol}(\beta_r)
   = \bigl( \Psi_0^{w_{j+k}}(\beta_r(\theta_1)) \bigr)^{-1}
     \bigl( S T_{ab}^{(k)} S^{-1} + o(I) \bigr)
     \Psi_0^{w_j}(\beta_r(\theta_0)),
\]
where
\[
T_{00}^{(k)} = T^{k},\qquad
T_{01}^{(k)} = U_1^{-1} T^{k},\qquad
T_{10}^{(k)} = T^{k} U_1,\qquad
T_{11}^{(k)} = U_1^{-1} T^{k} U_1,
\]
with \(T = Q^{-1} U_2^{-1} U_1^{-1}\).  Equivalently,
\[
T_{ab}^{(k)} = A_b^{-1} T^{k} A_a ,\qquad
A_0 = I,\qquad
A_1 = U_1 .
\]
\end{proposition}
\begin{proof}
The formula is obtained by repeatedly applying Theorem~\ref{thm:sectorial-limits},Theorem~\ref{thm:stokes-jumps} and
Proposition~\ref{prop:one-stokes-direction}.  The only point is the location of the initial and terminal points. 
\end{proof}

\section{Local comparison near zeros}
\label{sec:local-comparison-near-zeros}

In this section we compare the parallel transport of the flat connection
\(D_t\) on the compact Riemann surface with that of the corresponding
polynomial model near a zero of \(q\).  The comparison is carried out
directly using matrix-valued functions built
from parallel transports.

Let \(X\) be a compact Riemann surface and 
\(q \in H^0(X, K^4)\) a holomorphic quartic differential.
We consider the ray family
\[
q_t = t q , \qquad t > 0 .
\]
Denote by
\((\mathbb{K}_{\mathrm{Sp}(4,\mathbb{R})}, \phi_t, h_t, D_t)\)
the corresponding cyclic \(\mathrm{Sp}(4,\mathbb{R})\)-Higgs bundle in the
Hitchin section, its harmonic metric, and the associated flat connection,
where
\[
\phi_t = \phi(q_t) , \qquad
D_t = \nabla_{h_t} + \phi_t + \phi_t^{*_{h_t}} .
\]
(See Section~\ref{subsec:the-cyclic-SP(4,R)-Higgs-bundles-in-the-Hitchin-ection}.)
Near a zero \(p\) of \(q\) of order \(d\), we pick a local holomorphic coordinate \(z\) with
\[
q = z^d \odif{z}^4 ,
\qquad\text{and therefore}\qquad
q_t = t z^d \odif{z}^4 .
\]
Choose \(\varepsilon_0>0\) sufficiently small and set
\(B_0 = \{ z : |z| < \varepsilon_0 \}\).
We also introduce the polynomial model on \(\mathbb{C}\) given by
\[
q_t^{\mathrm M} = t z^d \odif{z}^4 .
\]
Let
\(
(\mathbb{K}^{\mathrm M}_{\mathrm{Sp}(4,\mathbb R)},
 \phi_t^{\mathrm M}, h_t^{\mathrm M}, D_t^{\mathrm M})
\)
be the corresponding cyclic \(\mathrm{Sp}(4,\mathbb R)\)-Higgs bundle,
harmonic metric, and flat connection on the complex plane. Let \(N=d+4\) and 
\[
\tau=t^{1/N}z.
\]
Both on the surface and on the plane, we will work on the canonical
holomorphic frame \(\mathcal{F}_{\tau}\) induced by the coordinate \(\tau\), i.e., 
\begin{equation}\label{eq:definition-frmae-tau}
    \mathcal{F}_{\tau}=\bigl\{ t^{\frac{3}{2N}}\odif{z}^{\frac{3}{2}}, t^{\frac{1}{2N}}\odif{z}^{\frac{1}{2}}, t^{-\frac{1}{2N}}\odif{z}^{-\frac{1}{2}},t^{-\frac{3}{2N}}\odif{z}^{-\frac{3}{2}} \bigr\}
\end{equation}
(see
Definition~\ref{def:frame-induced-by-the-coordinate}). In summarize, we will work on the coordinate \(z\) but with the holomorphic frame \(\mathcal{F}_{\tau}\). All fundamental matrices of flat connections are normalized at the zero \(p\) (corresponding to
\(z=0\)).  With respect to this frame, let
\(
\Psi_t(z)
\)
be the parallel transport matrix for \(D_t\), characterized by
\[
D_t\Psi_t=0, \qquad \Psi_t(0)=I.
\]
Similarly, let \(\Psi_t^{\mathrm M}(z)\) be the parallel transport
matrix for the model flat connection \(D_t^{\mathrm M}\), expressed in the
same frame and characterized by
\[
D_t^{\mathrm M}\Psi_t^{\mathrm M}=0, \qquad
\Psi_t^{\mathrm M}(0)=I.
\]
The comparison matrix on the punctured disk \(B\) is then
\[
G_t(z) = \Psi_t(z) \bigl(\Psi_t^{\mathrm M}(z)\bigr)^{-1},
\qquad z\in B .
\]
Our objective is to understand the asymptotic behavior of
\(G_t(z)\) as \(t\to+\infty\).

\subsection{Comparison of harmonic metrics}

Since the connection one-form of \(D_t\) is expressed through the harmonic metric \(h_t\) and its first derivatives, the first step is to compare \(h_t\) with the model harmonic metric \(h_t^{\mathrm M}\).

Recall that we work in the holomorphic frame \(\mathcal{F}_{\tau}\) (see \eqref{eq:definition-frmae-tau}), but with coordinate \(z\). Here \(\tau=t^{1/N}z\) for \(N=d+4\).
Fix the local background flat metric \(g = |\mathrm{d}\tau|^2\) around the zero. With respect to the ordered splitting
\[
\mathbb{K}_{\mathrm{Sp}(4,\mathbb{R})}
   = K^{3/2} \oplus K^{1/2} \oplus K^{-1/2} \oplus K^{-3/2},
\]
the harmonic metric \(h_t\) is diagonal and takes the form
\[
h_{t} = \operatorname{diag}\bigl( h_{1,\tau}, h_{2,\tau}, h_{2,\tau}^{-1}, h_{1,\tau}^{-1} \bigr).
\]
Define scalar functions \(u_{1,\tau}\) and \(u_{2,\tau}\) locally by
\[
h_{1,\tau} = e^{u_{1,\tau}}g^{-3/2} , \qquad
h_{2,\tau} = e^{u_{2,\tau}}g^{-1/2} .
\]
On the dual factors one then has
\[
h_{2,\tau}^{-1} = e^{-u_{2,\tau}} g^{1/2}, \qquad
h_{1,\tau}^{-1} = e^{-u_{1,\tau}} g^{3/2}.
\]
From the global form of the Hitchin equation \eqref{eq:globoal-cyclic-SP(4,R)-Hitchin-equation}, we obtain the local elliptic system
\begin{equation}\label{eq:local-Hitchin-equation-with-tau}
\begin{dcases}
\Delta_{\tau} u_{1,\tau} = |\tau^d|^{2} e^{2u_{1,\tau}} - e^{u_{2,\tau}-u_{1,\tau}} ,\\
\Delta_{\tau} u_{2,\tau} = e^{ u_{2,\tau}- u_{1,\tau}} - e^{-2 u_{2,\tau}} .
\end{dcases}
\end{equation}
Here \(\Delta_{\tau} = \partial_{\tau} \partial_{\bar{\tau}}\).
Transforming back to the \(z\)-variable, we introduce the functions
\begin{align*}
    u_{1,t}(z)&=u_{1,\tau}(\tau)=u_{1,\tau}(t^{1/N}z),\\
    u_{2,t}(z)&=u_{2,\tau}(\tau)=u_{2,\tau}(t^{1/N}z).
\end{align*}
with \(N=d+4\).
Then
\begin{equation}
\begin{dcases}
\Delta_{z} u_{1,t} = t^{2/N}\left(|t z^d|^{2} e^{2u_{1,t}} - e^{u_{2,t}-u_{1,t}} \right),\\
\Delta_{z} u_{2,t} =t^{2/N}\left( e^{ u_{2,t}- u_{1,t}} - e^{-2 u_{2,t}}\right) .
\end{dcases}
\end{equation}
Here \(\Delta_{z} = \partial_{z} \partial_{\bar{z}}\). We have obtained the local form of the Hitchin equation with respect to the frame \(\mathcal{F}_{\tau}\), but written with coordinate \(z\).

Recall that in Section~\ref{sec:local-model} we studied the equation \eqref{eq:local-Hitchin-equation-with-tau} over \(\mathbb{C}\).  To distinguish the notation, we let \(u_1^{\mathrm M}(\tau), u_2^{\mathrm M}(\tau)\) denote the unique radial complete solution of
\begin{equation}
\label{eq:planar-hitchin-equation-for-tau}
\begin{dcases}
\Delta_{\tau} u_1^{\mathrm M} = |\tau^d|^2 e^{2u_1^{\mathrm M}}
                               - e^{u_2^{\mathrm M}-u_1^{\mathrm M}},\\
\Delta_{\tau} u_2^{\mathrm M} = e^{u_2^{\mathrm M}-u_1^{\mathrm M}}
                               - e^{-2u_2^{\mathrm M}},
\end{dcases}
\end{equation}
on \(\mathbb{C}\) with \(\Delta_{\tau} = \partial_{\tau} \partial_{\bar{\tau}}\).

Using the coordinate transformation \(\tau = t^{\frac{1}{d+4}} z\), a direct computation shows that the rescaled functions
\begin{align*}
    u_{1,t}^{\mathrm M}(z)
        &= u_1^{\mathrm M}(\tau) 
         = u_1^{\mathrm M}\bigl( t^{\frac{1}{d+4}} z \bigr),\\
    u_{2,t}^{\mathrm M}(z)
        &= u_2^{\mathrm M}(\tau)
         = u_2^{\mathrm M}\bigl( t^{\frac{1}{d+4}} z \bigr),
\end{align*}
satisfy the same local elliptic system as \(u_{1,t}, u_{2,t}\):
\begin{equation}\label{eq:hitchin-equation-for-tz^d}
\begin{dcases}
\Delta_z u_{1,t}^{\mathrm M}
    = t^{2/N}\left(|t z^d|^{2} e^{2u_{1,t}^{\mathrm M}}
      - e^{u_{2,t}^{\mathrm M}-u_{1,t}^{\mathrm M}}\right) ,\\
\Delta_z u_{2,t}^{\mathrm M}
    =t^{2/N}\left( e^{ u_{2,t}^{\mathrm M}- u_{1,t}^{\mathrm M}}
      - e^{-2 u_{2,t}^{\mathrm M}}\right) ,
\end{dcases}
\end{equation}
with \(\Delta_z = \partial_z \partial_{\bar z}\).
Indeed, \(u_{1,t}^{\mathrm M},u_{2,t}^{\mathrm M}\) are the scalar components of the harmonic metric with respect to the frame \(\mathcal{F}_{\tau}\), but written with coordinate \(z\).

We shall estimate the differences
\(v_{k,t}=u_{k,t} - u_{k,t}^{\mathrm{M}}\) for \(k = 1,2\).
\begin{proposition}
\label{prop:boundary-comparison-v}
For any \(0<\delta<1\), there exists a constant \(C>0\) such that, on the circle
\(\partial B_0 = \{ z : |z| = \varepsilon_0 \}\),
\[
|v_{k,t}(z)|
   \leq C t^{-\frac{1}{8}}
     e^{-2\sqrt{2} t^{\frac{1}{4}} r_0\delta},
\qquad k = 1,2,
\]
for all sufficiently large \(t\), where \(r_0 = \frac{4}{d+4} \varepsilon_0^{\frac{d+4}{4}}\).
\end{proposition}

\begin{proof}
The idea of the proof is to work in the natural coordinate \(w\), where
exponentially decaying estimates are available for both \(u_{k,t}\)
and \(u_{k,t}^{\mathrm{M}}\).  When returning to the \(z\)-coordinate,
each function acquires a logarithmic divergence term.  These divergent
terms cancel in the difference \(v_{k,t}\), leaving an exponentially
decaying error.

We first analyze \(u_{k,t}\).  These functions were studied in
\cite{collier2017asymptotics}, but with different normalizations.
To relate the two formulations, we work in the natural coordinate \(w = \frac{4}{d+4} z^{\frac{d+4}{4}}\), which is defined on a sector
contained in the disk (see Section~\ref{subsec:explicit-computation-for-z^d})
and satisfies \(q = \odif{w}^4\).  In \cite{collier2017asymptotics} the
authors use the rescaled frame \(\sigma_1,\dots,\sigma_4\) given by
\[
\sigma_k = t^{\frac{5-2k}{8}} \odif{w}^{\frac{5-2k}{2}},
\]
and introduce the functions
\begin{equation}\label{eq:definition-of-tilde-u-in-proof}
    \tilde u_{1,t}(w) = -u_{1,t}(w) - \tfrac{3}{4}\log|\tau^d|,\qquad
    \tilde u_{2,t}(w) = -u_{2,t}(w) - \tfrac{1}{4}\log|\tau^d|,
\end{equation}
for \(\tau=t^{1/N}z\). Together with the symmetric extensions
\(\tilde u_{3,t} = -\tilde u_{2,t}\), \(\tilde u_{4,t} = -\tilde u_{1,t}\),
they define
\begin{align*}
\tilde w_{1,t}
   &= \frac{1}{2} \sum_{i \in \mathbb{Z}_4} \sqrt{-1}^{i}
      (\tilde u_{i,t} - \tilde u_{i+1,t})
    = -(\tilde u_{1,t} + \tilde u_{2,t}),\\[2mm]
\tilde w_{2,t}
   &= \frac{1}{2} \sum_{i \in \mathbb{Z}_4} (-1)^{i}
      (\tilde u_{i,t} - \tilde u_{i+1,t})
    = -2(\tilde u_{1,t} - \tilde u_{2,t}).
\end{align*}
According to \cite[Theorem~6.1]{collier2017asymptotics}, for any
\(0<\delta<1\), there exists a constant \(C_1>0\) such that
\[
|\tilde w_{k,t}(w)|
   \leq C_1 t^{-\frac{3}{8}} e^{-2|1-(\sqrt{-1})^{k}| t^{\frac{1}{4}} |w|\delta}
\qquad\text{as } t \to +\infty .
\]
Restrict to the circle \(|z| = \varepsilon_0\) and write
\(|w| = r_0 = \frac{4}{d+4} \varepsilon_0^{\frac{d+4}{4}}\).  Since
\(\tilde u_{k,t}\) are linear combinations of \(\tilde w_{1,t}\) and
\(\tilde w_{2,t}\), we obtain a constant \(C_2>0\) such that
\[
\bigl| \tilde u_{k,t}(z) \bigr|
   \leq C_2 t^{-\frac{3}{8}} e^{-2\sqrt{2} t^{\frac{1}{4}} r_0\delta}
\]
for all \(z\in\partial B_0\) and all sufficiently large \(t\).  By
\eqref{eq:definition-of-tilde-u-in-proof}, there exists a constant
\(C_3>0\) such that
\begin{equation}
    \begin{aligned}\label{eq:estimate-for-u-in-proof}
        \bigl| u_{1,t}(z) + \tfrac{3}{4}\log|t^{d/N} \varepsilon_0^d| \bigr|
       & \leq C_3 t^{-\frac{3}{8}} e^{-2\sqrt{2} t^{\frac{1}{4}} r_0\delta}, \\
     \bigl| u_{2,t}(z) + \tfrac{1}{4}\log|t^{d/N} \varepsilon_0^d| \bigr|
       & \leq C_3 t^{-\frac{3}{8}} e^{-2\sqrt{2} t^{\frac{1}{4}} r_0\delta}
    \end{aligned}
\end{equation}
for all \(z\in\partial B_0\) and all sufficiently large \(t\). 

We now turn to \(u_{k,t}^{\mathrm{M}}\).  Let \(\tau = t^{\frac{1}{d+4}} z\).
From the definition of the rescaled model functions, we have
\begin{equation}
    \begin{aligned}\label{eq:relation-of-model-functions}
    u_{1,t}^{\mathrm M}(z) &= u_1^{\mathrm M}(\tau) ,\\
    u_{2,t}^{\mathrm M}(z) &= u_2^{\mathrm M}(\tau),
    \end{aligned}
\end{equation}
where \(u_1^{\mathrm M},u_2^{\mathrm M}\) is the unique complete solution of
\eqref{eq:planar-hitchin-equation-for-tau}. From
\eqref{eq:definition-for-tilde-u-z^d}, we can write
\begin{align*}
u_1^{\mathrm M}(\tau) + u_2^{\mathrm M}(\tau) &= -\log|\tau^d| + \tilde w_1(\tau),\\
u_1^{\mathrm M}(\tau) - u_2^{\mathrm M}(\tau) &= -\frac12\log|\tau^d| + \tilde w_2(\tau),
\end{align*}
where \(\tilde w_1,\tilde w_2\) are the combinations appearing in
Proposition~\ref{prop:estimates-of-tilde-w}.  Applying the estimates of
Proposition~\ref{prop:estimates-of-tilde-w} (see also
Lemma~\ref{lem:asymptotics-tilde-u}) to the error terms \(\tilde w_1,\tilde w_2\)
yields a constant \(C_4>0\) such that
\begin{align*}
\bigl| u_1^{\mathrm M}(\tau) + u_2^{\mathrm M}(\tau) + \log|\tau^d| \bigr|
   &\leq C_4 e^{-2\sqrt{2}r} r^{-\frac{1}{2}},\\
\bigl| u_1^{\mathrm M}(\tau) - u_2^{\mathrm M}(\tau) + \tfrac12\log|\tau^d| \bigr|
   &\leq C_4 e^{-4r} r^{-\frac{1}{2}},
\end{align*}
for all sufficiently large \(|\tau|\), where
\(r = \frac{4}{d+4} |\tau|^{\frac{d+4}{4}}\).  Hence there exists a
constant \(C_5>0\) such that
\begin{align*}
\bigl| u_1^{\mathrm M}(\tau) + \tfrac34 \log|\tau^d| \bigr|
   &\leq C_5 e^{-2\sqrt{2}r} r^{-\frac{1}{2}},\\
\bigl| u_2^{\mathrm M}(\tau) + \tfrac14\log|\tau^d| \bigr|
   &\leq C_5 e^{-2\sqrt{2}r} r^{-\frac{1}{2}}
\end{align*}
for all sufficiently large \(|\tau|\).

Translating back to the \(z\)-variable and combining with
\eqref{eq:relation-of-model-functions}, we obtain a constant
\(C_6>0\) such that, on \(\partial B_0\),
\begin{align*}
\bigl| u_{1,t}^{\mathrm{M}}(z) + \tfrac34\log|t^{d/N} \varepsilon_0^d| \bigr|
   &\leq C_6 t^{-\frac{1}{8}} e^{-2\sqrt{2} t^{\frac{1}{4}} r_0},\\
\bigl| u_{2,t}^{\mathrm{M}}(z) + \tfrac14\log|t^{d/N} \varepsilon_0^d| \bigr|
   &\leq C_6 t^{-\frac{1}{8}} e^{-2\sqrt{2} t^{\frac{1}{4}} r_0},
\end{align*}
for all sufficiently large \(t\).

Note that
\begin{align*}
\bigl| u_{1,t}(z) - u_{1,t}^{\mathrm{M}}(z) \bigr|
   &\leq \bigl| u_{1,t}(z) + \tfrac34\log|t^{d/N} \varepsilon_0^d| \bigr|
      + \bigl| u_{1,t}^{\mathrm{M}}(z) + \tfrac34\log|t^{d/N} \varepsilon_0^d| \bigr|,\\
\bigl| u_{2,t}(z) - u_{2,t}^{\mathrm{M}}(z) \bigr|
   &\leq \bigl| u_{2,t}(z) + \tfrac14\log|t^{d/N} \varepsilon_0^d| \bigr|
      + \bigl| u_{2,t}^{\mathrm{M}}(z) + \tfrac14\log|t^{d/N} \varepsilon_0^d| \bigr|.
\end{align*}
 The term with the slower exponential decay dominates
the sum. Hence there exists a constant \(C>0\) such that, for all
\(z\in\partial B_0\) and all sufficiently large \(t\),
\[
|v_{k,t}(z)|
   \leq C t^{-\frac{1}{8}}
     e^{-2\sqrt{2} t^{\frac{1}{4}} r_0\delta},
\qquad k = 1,2.
\]
This completes the proof.
\end{proof}
We now extend the boundary estimates of Proposition~\ref{prop:boundary-comparison-v}
to the interior of the disk.

\begin{proposition}
\label{prop:interior-comparison-v}
For any \(0<\delta<1\), there exists a constant \(C>0\) such that
\[
|v_{k,t}(z)| = |u_{k,t}(z) - u_{k,t}^{\mathrm{M}}(z)|
   \leq C t^{-\frac{1}{8}}
            e^{-2\sqrt{2} t^{\frac{1}{4}} r_0\delta},
\qquad k = 1,2,
\]
for all \(z\in B_0 = \{ z : |z| \leq \varepsilon_0 \}\) and all sufficiently large
\(t\), where \(r_0 = \frac{4}{d+4} \varepsilon_0^{\frac{d+4}{4}}\).
\end{proposition}

\begin{proof}
We first derive the elliptic system satisfied by the differences
\(v_{1,t},v_{2,t}\).  Recall that
\begin{equation*}
\begin{dcases}
t^{-\frac{2}{N}}\Delta_z u_{1,t} = |t z^d|^{2} e^{2u_{1,t}} - e^{u_{2,t}-u_{1,t}} ,\\
t^{-\frac{2}{N}}\Delta_z u_{2,t} = e^{ u_{2,t}- u_{1,t}} - e^{-2 u_{2,t}} ,
\end{dcases}
\end{equation*}
and that the model functions \(u_{1,t}^{\mathrm M},u_{2,t}^{\mathrm M}\) satisfy
the same system.  Subtracting the two systems gives
\begin{equation}\label{eq:elliptic-systems-for-v}
\begin{dcases}
t^{-\frac{2}{N}}\Delta_z v_{1,t} = |t z^d|^{2} \bigl(e^{2u_{1,t}} - e^{2u^{\mathrm M}_{1,t}}\bigr)
                  - \bigl(e^{ u_{2,t}- u_{1,t}} - e^{ u^{\mathrm M}_{2,t}- u^{\mathrm M}_{1,t}}\bigr),\\
t^{-\frac{2}{N}}\Delta_z v_{2,t} = \bigl(e^{ u_{2,t}- u_{1,t}} - e^{ u^{\mathrm M}_{2,t}- u^{\mathrm M}_{1,t}}\bigr)
                  - \bigl(e^{-2 u_{2,t}} - e^{-2u^{\mathrm M}_{2,t}}\bigr).
\end{dcases}
\end{equation}
Using the identity \(e^a-e^b = (a-b)\int_{\,0}^1 e^{s a+(1-s)b}\odif s\), we rewrite
the right-hand side as
\begin{equation}\label{eq:elliptic-systems-for-v-f}
\begin{dcases}
t^{-\frac{2}{N}}\Delta_z v_{1,t} = 2v_{1,t}f_1 + (v_{1,t}-v_{2,t})f_3,\\
t^{-\frac{2}{N}}\Delta_z v_{2,t} = (v_{2,t}-v_{1,t})f_3 + 2v_{2,t}f_2,
\end{dcases}
\end{equation}
with the strictly positive coefficients
\begin{align*}
f_1 &= \int_0^1 \exp\Bigl(2s(u_{1,t}+\log|t z^d|)
           + 2(1-s)(u^{\mathrm M}_{1,t}+\log|t z^d|)\Bigr)\odif s,\\
f_2 &= \int_0^1 \exp\Bigl(-2s u_{2,t} - 2(1-s)u^{\mathrm M}_{2,t}\Bigr)\odif s,\\
f_3 &= \int_0^1 \exp\Bigl(s(u_{2,t}-u_{1,t})
           + (1-s)(u^{\mathrm M}_{2,t}-u^{\mathrm M}_{1,t})\Bigr)\odif s .
\end{align*}

Note that \(f_3>0\) making the system cooperative. The maximum principle (see \cite[Theorem~1]{Sirakov2009}) implies
\[
\sup_{B_0} \max_{k=1,2} |v_{k,t}| \leq \sup_{\partial B_0} \max_{k=1,2} |v_{k,t}|.
\]
The stated estimate therefore follows
directly from Proposition~\ref{prop:boundary-comparison-v}.
\end{proof}

We now estimate the first derivatives and Laplacians of \(v_{k,t}\).

\begin{proposition}\label{prop:interior-comparison-v-derivatives-Laplacians}
For any \(0<\delta<1\), there exist constants \(C>0\) and \(M>0\) such that
\[
|\Delta_z v_{k,t}(z)| \leq C t^M e^{-2\sqrt{2}t^{1/4} r_0\delta},
\qquad k=1,2,
\]
for all \(z \in B_0\) and all sufficiently large \(t\).  Moreover, for any \(\varepsilon < \varepsilon_0\) and the
smaller disk \(B_{\varepsilon} = \{ z : |z| \leq \varepsilon \}\), there exist a positive constant \(C(\varepsilon) >0\) such that 
\[
|\partial_z v_{k,t}(z)| + |\partial_{\bar z} v_{k,t}(z)|
   \leq C(\varepsilon) t^M e^{-2\sqrt{2}t^{1/4} r_0\delta},
\qquad k=1,2,
\]
hold for all \(z\in B_{\varepsilon}\) and sufficiently large \(t\).
\end{proposition}
\begin{proof}
We start from an equivalent form of the elliptic system for \(v_{k,t}\):
\begin{equation}\label{eq:elliptic-systems-for-v-changed}
\begin{dcases}
t^{-\frac{2}{N}}\Delta_z v_{1,t} = |t z^d|^{2} e^{2u^{\mathrm M}_{1,t}} (e^{2v_{1,t}}-1)
                  - e^{ u^{\mathrm M}_{2,t}- u^{\mathrm M}_{1,t}}
                    (e^{ v_{2,t}- v_{1,t}}-1),\\
t^{-\frac{2}{N}}\Delta_z v_{2,t} = e^{ u^{\mathrm M}_{2,t}- u^{\mathrm M}_{1,t}}
                    (e^{ v_{2,t}- v_{1,t}}-1)
                  - e^{-2u^{\mathrm M}_{2,t}} (e^{-2 v_{2,t}}-1).
\end{dcases}
\end{equation}
Let \(\tau = t^{\frac{1}{N}} z\).  By definition of the rescaled model functions,
\begin{equation*}
u_{1,t}^{\mathrm M}(z) = u_1^{\mathrm M}(\tau), \qquad
u_{2,t}^{\mathrm M}(z) = u_2^{\mathrm M}(\tau) ,
\end{equation*}
where \(u_1^{\mathrm M},u_2^{\mathrm M}\) is the unique complete solution of
\eqref{eq:planar-hitchin-equation-for-tau} on \(\mathbb{C}\).
Theorem~\ref{thm:general-planar-theory} provides constants
\(C_1,C_2>0\) such that
\[
|u_1^{\mathrm M}(\tau)|
   \leq \max\!\Bigl\{C_1,\; \tfrac{3}{4}\log|\tau^d| + C_2\Bigr\},\qquad
|u_2^{\mathrm M}(\tau)|
   \leq \max\!\Bigl\{C_1,\; \tfrac{1}{4}\log|\tau^d| + C_2\Bigr\},
\]
for all \(\tau\in\mathbb{C}\).  Consequently, there exist constants
\(C_3>0\) and \(M>0\) such that,
\[
|t z^d|^2 e^{2u^{\mathrm M}_{1,t}} \leq C_3 t^N,\qquad
e^{u^{\mathrm M}_{2,t}-u^{\mathrm M}_{1,t}} \leq C_3 t^N,\qquad
e^{-2u^{\mathrm M}_{2,t}} \leq C_3 t^N .
\]
for all \(z\in B\) and sufficiently large \(t\).
Set \(E_t = t^{-1/8} e^{-2\sqrt{2}t^{1/4} r_0\delta}\).  By
Proposition~\ref{prop:interior-comparison-v}, there exists a constant
\(C_4>0\) such that
\[
|v_{1,t}(z)| \leq C_4 E_t,\qquad |v_{2,t}(z)| \leq C_4 E_t,
\qquad z\in B_0 .
\]
for all sufficiently large \(t\).
Since \(E_t\to 0\) as \(t\to+\infty\), for all sufficiently large
\(t\) the argument of each exponential in
\eqref{eq:elliptic-systems-for-v-changed} lies in a fixed compact set
on which \(|e^x-1| \leq 2|x|\) holds.  Hence there exists a constant \(C_5 >0\) such that
\[
|\Delta_z v_{k,t}(z)| \leq C_5 t^M E_t,\qquad k=1,2,
\]
 for all \(z \in B_0\) and sufficiently large \(t\).  This proves the
Laplacian estimates.

The bounds for the first derivatives follow from standard interior
elliptic estimates for the Poisson equation on nested disks.
\end{proof}

\subsection{Comparison of parallel transports}

We continue to work in the local coordinate \(z\) with \(q=z^d\odif{z}^4\) 
and use the canonical holomorphic frame \(\mathcal{F}_{\tau}\) induced by \(\tau=t^{1/N}z\) with \(N=d+4\) 
(see Definition~\ref{def:frame-induced-by-the-coordinate} and equation~\eqref{eq:definition-frmae-tau}).
Note that \(G_t(z) = \Psi_t(z) \bigl(\Psi_t^{\mathrm M}(z)\bigr)^{-1}\) obeys the differential equation
\[
\odif G_t = G_t  \Theta_t,\qquad
\Theta_t = \Psi_t^{\mathrm M} \bigl( A_t - A_t^{\mathrm M} \bigr)\bigl( \Psi_t^{\mathrm M} \bigr)^{-1},
\]
where \(A_t\) and \(A_t^{\mathrm M}\) are the connection matrices of
\(D_t\) and \(D_t^{\mathrm M}\) with respect to the frame \(\mathcal{F}_{\tau}\) respectively. 
Moreover, \(G_t(0)=I\). Note that the inverse of \(G_t\) satisfies the following differential equation:
\[
\odif G_t^{-1} = -\Theta_tG_t^{-1}.
\]

The goal of this subsection is to analyze the asymptotic behavior
of \(G_t\) as \(t \to +\infty\), using the estimates obtained in the previous subsection.

As above, let \(B_0 = \bigl\{ z : |z| \leq \varepsilon_0 \bigr\}\) 
with a small constant \(\varepsilon_0>0\).
The parameter \(r_0 = \frac{4}{d+4}\varepsilon_0^{\frac{d+4}{4}}\) 
is fixed once \(\varepsilon_0\) is chosen. 
Where necessary, we work on a smaller disk 
\(B_{\varepsilon}=\{z : |z| \leq \varepsilon\} \subset B_0\) 
for some \(\varepsilon < \varepsilon_0\).

Recall that inside the right half-plane \(\mathbb{H}\), the constant model
\(q = \odif{w}^4\) has the three Stokes directions
\[
\theta_1 = -\frac{\pi}{4},\qquad \theta_2 = 0,\qquad \theta_3 = \frac{\pi}{4},
\]
introduced in Definition~\ref{def:stokes-directions}.  We now translate
this notion to the \(z\)-coordinate using the sectorial natural
coordinates.  Set \(N=d+4\).  In
Section~\ref{subsec:explicit-computation-for-z^d} we constructed, for
each \(j\in\mathbb Z/N\mathbb Z\), the natural coordinate
\[
w_j(z) = \frac{4}{N} \bigl( z e^{-\frac{2\pi j\sqrt{-1}}{N}} \bigr)^{N/4},
\]
which maps the sector
\[
U_j = \bigl\{ z = r e^{\sqrt{-1}\theta} :
          \tfrac{2\pi(j-1)}{N} < \theta < \tfrac{2\pi(j+1)}{N} \bigr\}
\]
biholomorphically onto the right half-plane \(\mathbb{H}\) and satisfies
\(z^d\odif{z}^4 = \odif{w_j}^4\).

On the \(j\)-th sector, the rays in the \(z\)-plane that map under \(w_j\) to the Stokes directions in the right half-plane are precisely 
\[
\arg z = \frac{(2j-1)\pi}{N},\qquad
\arg z = \frac{2j\pi}{N},\qquad
\arg z = \frac{(2j+1)\pi}{N}.
\]
Taking the union over all sectors, the Stokes directions in the
punctured \(z\)-disk are the rays 
\[
\arg z = \frac{k\pi}{N},\qquad k\in\mathbb Z .
\]  
This description is compatible on overlaps because adjacent natural coordinates satisfy
\(w_{j+1} = -\sqrt{-1}w_j\), so a Stokes ray in one right-half-plane
coordinate is sent to a Stokes ray in the adjacent one.

\begin{definition}\label{def:Stokes-directions-in-z-coordinate}
The Stokes directions for \(q=z^d\odif{z}^4\) in the \(z\)-coordinate are
\[
\theta_k = \frac{k\pi}{N}, \qquad k= 0,1,2,\dots, 2N-1,
\]
where \(N=d+4\).
\end{definition}

The Stokes rays divide \(B_0\setminus\{0\}\) into finitely many open
sectors, each lying between two adjacent Stokes directions. We fix a
closed sector \(W_0\subset B_0\) with vertex at the origin such that
\(W_0\setminus\{0\}\) is compactly contained in one of these open
sectors. When necessary, we consider the corresponding closed sector
\(W_\varepsilon\subset B_\varepsilon\) with the same angular bounds.
\begin{proposition}
\label{prop:conjugated-error-estimate}
There exists \(\varepsilon < \varepsilon_0\) such that the corresponding sector \(W_\varepsilon\) has the following property. 
Let \(\|\cdot\|\) be any sub-multiplicative matrix norm. Then there exists \(t_0>0\) such that, for all \(t>t_0\) and
all \(z \in W_{\varepsilon}\), we have
\[
\bigl\| \Theta_t(z) \bigr\| \leq e^{-\alpha t^{1/4}} .
\]
for some \(\alpha>0\).
\end{proposition}
\begin{proof}
We first relate \(\Psi_t^{\mathrm M}\) to the parallel transport of the
planar model with \(q_{\tau}=\tau^d\odif{\tau}^4\) studied in
Section~\ref{subsec:Half-plane-decompositions-and-Stokes-jumps}.
Let \(\Psi^{\mathrm M}(\tau)\) denote the fundamental matrix, with respect to the
canonical frame \(\mathcal{F}_{\tau}\) and normalized at \(0\), of the flat connection \(D_{q_{\tau}}\) associated with this polynomial differential. Under the coordinate change \(\tau = t^{\frac{1}{d+4}} z\), we have
\(t z^d\odif{z}^4 = \tau^d\odif{\tau}^4\). Consequently the connection
\(D_t^{\mathrm M}\) written in the \(\tau\)-coordinate and in the frame
\(\mathcal{F}_{\tau}\) coincides with \(D_{q_{\tau}}\).  Hence its parallel
transport matrix with respect to the frame
\(\mathcal{F}_{\tau}\) is exactly \(\Psi^{\mathrm M}(\tau)\).

Transforming back to the
\(z\)-variable, we obtain
\[
\Psi_t^{\mathrm M}(z)=\Psi^{\mathrm M}(\tau)=\Psi^{\mathrm M}(t^{\frac{1}{d+4}} z).
\]
On the sector \(W_0\), we have a globally well-defined natural coordinate \(w\) such that \(\tau^d\odif{\tau}^4 = \odif{w}^4\). We now work in the canonical frame \(\mathcal F_w\) induced by the coordinate \(w\). Let
 \(
 \Psi^{\mathrm M}_w(w) 
 \)
denote the fundamental matrix of the polynomial model connection \(D_{q_\tau}\), written in the frame \(\mathcal F_w\).

Recalling \eqref{eq:constant-standard-model-diagonal}, the fundamental
matrix of the constant model is 
\[
\Psi_0(w) = S\exp\bigl(D(w)\bigr)S^{-1}
          = S
            \begin{pmatrix}
              e^{2|w|\cos\theta} & & & \\
              & e^{2|w|\sin\theta} & & \\
              & & e^{-2|w|\cos\theta} & \\
              & & & e^{-2|w|\sin\theta}
            \end{pmatrix}
            S^{-1}.
\]
By Theorem~\ref{thm:sectorial-limits}, there exists a constant \(C_1 >0\) such that 
\[
\|\Psi^{\mathrm M}_w(w)\| \leq C_1 \|\Psi_0(w)\|
\]
when \(|w|\) is large enough. Consequently, translating back to the \(z\)-coordinate and \(\tau\)-frame \(\mathcal{F}_{\tau}\), we obtain 
\[
\|\Psi_t^{\mathrm M}(z)\| \leq C_1(t) \|\Psi_0(w(t^{\frac{1}{d+4}}z))\|
\]
for all \(z \in W_0\), where \(C_1(t)>0\) grows at most polynomially in \(t\) as \(t \to +\infty\). A similar estimate also holds for the inverse of \(\Psi^{\mathrm M}_w(w)\). Hence we have
\[
 \|\Psi_t^{\mathrm M}(z)\|\|\Psi_t^{\mathrm M}(z)^{-1}\| \leq C_1(t)^2\|\Psi_0(w(t^{\frac{1}{d+4}}z))\|\|\Psi_0(w(t^{\frac{1}{d+4}}z))^{-1}\| \leq C_2(t) e^{4t^{1/4}\frac{4}{d+4}|z|^{\frac{d+4}{4}}}.
\]
for \(z \in W_0\), where \(C_2(t)\) grows at most polynomially in \(t\) as \(t \to +\infty\).

We now estimate the conjugated difference
\[
\begin{aligned}
    \Theta_t
    &= \Psi_t^{\mathrm M} \bigl( A_t - A_t^{\mathrm M} \bigr) \bigl( \Psi_t^{\mathrm M} \bigr)^{-1} \\
    &= \Psi_t^{\mathrm M}
       \left( U_t
        \odif{z}
         + V_t
         \odif{\bar{z}}
       \right)
       \bigl( \Psi_t^{\mathrm M} \bigr)^{-1},
\end{aligned}
\]
where
\begin{align*}
    U_t &= \begin{pmatrix}
            \partial_z v_{1,t} & 0 & 0 & 0 \\
            0 & \partial_z v_{2,t} & 0 & 0 \\
            0 & 0 & -\partial_z v_{2,t} & 0 \\
            0 & 0 & 0 & -\partial_z v_{1,t}
         \end{pmatrix} ,\\
   V_t &= t^{\frac{1}{d+4}} \begin{pmatrix}
            0 & e^{u^{\mathrm M}_{2,t} - u^{\mathrm M}_{1,t}} \bigl( e^{v_{2,t} - v_{1,t}} - 1 \bigr) & 0 & 0 \\
            0 & 0 & e^{-2u^{\mathrm M}_{2,t}} \bigl( e^{-2v_{2,t}} - 1 \bigr) & 0 \\
            0 & 0 & 0 & e^{u^{\mathrm M}_{2,t} - u^{\mathrm M}_{1,t}} \bigl( e^{v_{2,t} - v_{1,t}} - 1 \bigr) \\
           t \bar{z}^d e^{2u^{\mathrm M}_{1,t}} \bigl( e^{2v_{1,t}} - 1 \bigr) & 0 & 0 & 0
         \end{pmatrix} 
\end{align*}
with \(v_{1,t} = u_{1,t} - u^{\mathrm M}_{1,t},
v_{2,t} = u_{2,t} - u^{\mathrm M}_{2,t}\).
Using Theorem~\ref{thm:general-planar-theory}, one can check that the quantities
\[
|t z^d|^2 e^{2u^{\mathrm M}_{1,t}},\qquad
e^{u^{\mathrm M}_{2,t}-u^{\mathrm M}_{1,t}} ,\qquad
e^{-2u^{\mathrm M}_{2,t}}
\]
all grow polynomially as \(t \to +\infty\). For details, see the proof of Proposition~\ref{prop:interior-comparison-v-derivatives-Laplacians}. Fix \(0 <\varepsilon_1<\varepsilon_0\).
By Proposition~\ref{prop:interior-comparison-v} and Proposition~\ref{prop:interior-comparison-v-derivatives-Laplacians}, we have 
\[
\|U_t(z)\|+\|V_t(z)\| \leq C_3(t)e^{-\sqrt{2}t^{1/4} \frac{4}{d+4} \varepsilon_0^{\frac{d+4}{4}}},
\] 
for all \(z \in W_{\varepsilon_1}\), where \(C_3(t)\) grows polynomially as \(t \to +\infty\). Here we fix \(\delta=\frac12\) in Proposition~\ref{prop:interior-comparison-v} and Proposition~\ref{prop:interior-comparison-v-derivatives-Laplacians}.
Choose some small positive number \(\varepsilon\) such that \(\varepsilon <\varepsilon_1\) and 
\[
4\frac{4}{d+4}|\varepsilon|^{\frac{d+4}{4}} < \sqrt{2}\frac{4}{d+4} \varepsilon_0^{\frac{d+4}{4}} .
\]
Then on \(W_{\varepsilon}\), we have
\[
\|\Theta_t\| \leq C_4(t) e^{-\alpha t^{1/4}}
\]
for some \(\alpha >0\), where \(C_4(t)\) grows polynomially as \(t \to +\infty\). Since \(C_4(t)\) grows at most polynomially, it can be absorbed into
the exponential term. Thus, after decreasing \(\alpha>0\) if
necessary, we have
\[
\|\Theta_t\|
\leq
e^{-\alpha t^{1/4}}
\]
for all sufficiently large \(t\).This completes the proof.
\end{proof}
From now on, we work on the disk \(B_{\varepsilon}\) with the radius \(\varepsilon\) given by Proposition~\ref{prop:conjugated-error-estimate}.
\begin{proposition} \label{prop:comparison-parallel-transport-matrices}
Let \(\|\cdot\|\) be any sub-multiplicative matrix norm. For any \(z \in B_{\varepsilon}\) that does not lie on a Stokes direction, there exists \(t_0\) such that for all \(t>t_0\)
 \[
  \|G_t(z)-I\|\leq e^{-\alpha t^{1/4}},\qquad  \|G_t(z)^{-1}-I\|\leq e^{-\alpha t^{1/4}} 
  \] 
for some \(\alpha>0\). 
\end{proposition} 
\begin{proof} 
Choose a closed sector \(W_{\varepsilon}\subset B_{\varepsilon}\)
with vertex at the origin such that \(z\in W_{\varepsilon}\) and
\(W_{\varepsilon}\setminus\{0\}\) is compactly contained in the open
sector between two adjacent Stokes directions. Then the segment \(c(s)=sz\) for \(0 \leq s \leq 1\) lies entirely in \(W_{\varepsilon}\).
Recall that 
\begin{align*}
    \odv{G_t(c(s))}{s}&=G_t(c(s))\Theta_t(\dot{c}(s))   \\
    \odv{G_t^{-1}(c(s))}{s}&=-\Theta_t(\dot{c}(s))G_t^{-1}(c(s))
\end{align*}
with \(G_t(c(0))=I\) and \(G_t^{-1}(c(0))=I\). Let \(H_t=(G_t^{-1})^{\mathsf T}\); then we have 
\[
\odv{H_t(c(s))}{s}=-H_t(c(s))\left(\Theta_t(\dot{c}(s))\right)^{\mathsf T}.
\]
Applying \cite[Lemma~B.1]{dumas2015polynomial} and Proposition~\ref{prop:conjugated-error-estimate} to the above equations, we obtain this proposition.
\end{proof}

For the holonomy computation, we need a version of the preceding comparison in the natural-coordinate frames. For every \(j\in\mathbb Z/N\mathbb Z\), let \(w_j\) be the sectorial natural coordinate introduced in Section~\ref{subsec:explicit-computation-for-z^d}. Recall that, under the rescaling
\(
\tau=t^{1/N}z
\)
for \(N=d+4\),
the corresponding natural coordinate for \(q_t=tq\) is
\(
w_j(\tau)=t^{1/4}w_j(z).
\)
Hence the canonical holomorphic frame induced by the \(q_t\)-natural coordinate \(w_j(\tau)\) is
\[
\mathcal F_{t,w_j}
=
\left(
t^{\frac38}(\odif w_j)^{\frac32},
\;
t^{\frac18}(\odif w_j)^{\frac12},
\;
t^{-\frac18}(\odif w_j)^{-\frac12},
\;
t^{-\frac38}(\odif w_j)^{-\frac32}
\right).
\]
On the overlap \(U_j\cap U_{j+1}\), the transition law
\(
w_{j+1}=-\sqrt{-1}w_j
\)
gives
\[
\mathcal F_{t,w_{j+1}}
=
\mathcal F_{t,w_j}R,
\]
where \(R\) is the constant matrix defined in \eqref{eq:definition-for-frame-change-R}. Thus \(\{\mathcal F_{t,w_j}\}\) is a system of sectorial frames with constant transition matrices.

Since \(\mathcal F_{t,w_j}\) is not defined at the zero, we fix an arbitrary nonzero reference point
\(
p_j^*\in U_j\cap B_\varepsilon^*
\)
which is not in a Stokes direction. We normalize the surface and model fundamental matrices in the frame \(\mathcal F_{t,w_j}\) by
\[
\Psi_t^{w_j}(p_j^*)=I,
\qquad
\Psi_t^{\mathrm M,w_j}(p_j^*)=I.
\]
The choice of \(p_j^*\) will be suppressed from the notation.
\begin{proposition}
\label{prop:sectorial-comparison-natural-frames}
Fix \(j\in\mathbb Z/N\mathbb Z\), and let
\(
z\in U_j\cap B_\varepsilon^*
\)
be a point that does not lie on a Stokes direction. Denote by
\(
\Psi_t^{w_j}(z)
\)
and
\(\Psi_t^{\mathrm M,w_j}(z)
\)
the fundamental matrices of the surface and polynomial model, respectively, in the sectorial frame
\(\mathcal F_{t,w_j}\). Define
\[
G_t^{w_j}(z)
:=
\Psi_t^{w_j}(z)
\bigl(\Psi_t^{\mathrm M,w_j}(z)\bigr)^{-1}.
\]
Then there exists a constant \(\alpha>0\) such that, for all sufficiently large \(t\),
\[
\bigl\|G_t^{w_j}(z)-I\bigr\|
\leq
e^{-\alpha t^{1/4}}.
\]
In particular,
\[
G_t^{w_j}(z)\longrightarrow I
\]
as \(t\to+\infty\). Equivalently,
\[
\Psi_t^{w_j}(z)
=
\bigl(I+o(I)\bigr)
\Psi_t^{\mathrm M,w_j}(z).
\]
\end{proposition}

\begin{proof}
Let \(C_{t,j}(z)\) be the change-of-frame matrix determined by
\[
\mathcal F_{t,w_j}
=
\mathcal F_\tau C_{t,j}.
\]
Writing
\(
\tau=t^{1/N}z,
\)
we have
\[
a_{t,j}(z)
:=
\frac{\odif w_j}{\odif\tau}
=
e^{-\frac{\pi j\sqrt{-1}}{2}}\tau^{d/4},
\]
and hence
\[
C_{t,j}(z)
=
\operatorname{diag}\left(
a_{t,j}(z)^{3/2},
a_{t,j}(z)^{1/2},
a_{t,j}(z)^{-1/2},
a_{t,j}(z)^{-3/2}
\right).
\]
Applying the same change of frame to the surface and model fundamental solutions, we obtain
\[
G_t^{w_j}(z)
=
C_{t,j}(z)^{-1}G_t(p_j^*)^{-1}G_t(z)C_{t,j}(z).
\]
Therefore,
\[
G_t^{w_j}(z)-I
=
C_{t,j}(z)^{-1}
\bigl(G_t(p_j^*)^{-1} G_t(z)-I\bigr)
C_{t,j}(z).
\]

For the fixed point \(z\neq0\), the change-of-frame matrix and its inverse grow at most polynomially in \(t\). More precisely,
\[
\bigl\|C_{t,j}(z)\bigr\|\cdot
\bigl\|C_{t,j}(z)^{-1}\bigr\|
\leq
C_1 t^M
\]
for some constants \(C_1>0\) and \(M>0\).

By Proposition~\ref{prop:comparison-parallel-transport-matrices}, there exists \(\alpha>0\) such that
\[
\bigl\|G_t(z)-I\bigr\|
\leq
e^{-\alpha t^{1/4}}, \qquad \bigl\|G_t(p_j^*)^{-1}-I\bigr\| \leq e^{-\alpha t^{1/4}}
\]
for all sufficiently large \(t\). 
Note that 
\[
G_t(p_j^*)^{-1} G_t(z)-I=(G_t(p_j^*)^{-1}-I)G_t(z)+ (G_t(z)-I).
\]
Hence there exist constants \(C_2>0\) and \(\alpha>0\) such that 
\[
\bigl\|G_t(p_j^*)^{-1} G_t(z)-I\bigr\|\leq C_2 e^{-\alpha t^{1/4}}
\]
for all sufficiently large \(t\). 
It follows that
\[
\bigl\|G_t^{w_j}(z)-I\bigr\|
\leq
\bigl\|C_{t,j}(z)^{-1}\bigr\|\cdot
\bigl\|G_t(p_j^*)^{-1} G_t(z)-I\bigr\|\cdot
\bigl\|C_{t,j}(z)\bigr\| 
\leq
C_3 t^M e^{-\alpha t^{1/4}}
\]
for some constant \(C_3>0\) and all sufficiently large \(t\). 
Since the polynomial factor can be absorbed into the exponential decay, after replacing \(\alpha\) by a smaller positive constant \(\alpha'\), we obtain
\[
\bigl\|G_t^{w_j}(z)-I\bigr\|
\leq
e^{-\alpha' t^{1/4}}
\]
for all sufficiently large \(t\).
\end{proof}
\begin{remark}
The preceding argument does not require the initial point \(p_j^*\) and the terminal point \(z\) to lie in the same stable sector.
\end{remark}
\section{Asymptotic holonomy}
\label{sec:asymptotic-holonomy}

In this section we study the asymptotic behavior of the holonomy of the flat
connection \(D_t\) along the ray \(q_t=tq\) as \(t\to+\infty\).

Let \(X\) be a compact Riemann surface and let
\(q \in H^0(X, K^4)\) be a holomorphic quartic differential.
We consider the ray family
\[
q_t = t q , \qquad t > 0 .
\]
Denote by
\((\mathbb{K}_{\mathrm{Sp}(4,\mathbb{R})}, \phi_t, h_t, D_t)\)
the corresponding cyclic \(\mathrm{Sp}(4,\mathbb{R})\)-Higgs bundle in the
Hitchin section, its harmonic metric, and the associated flat connection,
where
\[
\phi_t = \phi(q_t) , \qquad
D_t = \nabla_{h_t} + \phi_t + \phi_t^{*_{h_t}} .
\]
(See Section~\ref{subsec:the-cyclic-SP(4,R)-Higgs-bundles-in-the-Hitchin-ection}.)

Let \([\gamma]\) be a non-trivial free homotopy class of closed curves
on \(X\). As discussed in
Section~\ref{subsec:quartic-differentials-and-the-induced-flat-metrics},
we choose a closed geodesic representative \(c_\gamma\) for the
singular flat metric \(|q|^{1/2}\). We assume that \(c_\gamma\) is a
finite concatenation of saddle connections and write
\[
c_\gamma=c_{\ell}*c_{\ell-1}*\cdots*c_1.
\]
Let \(p_1,\ldots,p_\ell\) denote the successive zeros encountered by
\(c_\gamma\). We orient each saddle connection \(c_i\) from \(p_i\) to
\(p_{i-1}\), where the indices are understood modulo \(\ell\).

Following the construction of Loftin, Tamburelli, and Wolf
\cite{loftin2026limits}, we modify \(c_\gamma\) by replacing its passage through each zero of \(q\) with an arc along the boundary of a small disk centered at that zero.

For each zero \(p_i\), choose a local holomorphic coordinate \(z_i\),
centered at \(p_i\), such that
\[
q=z_i^{d_i}\odif{z_i}^{4},
\]
where \(d_i\) is the order of the zero. Choose
\(\varepsilon>0\) sufficiently small and set
\[
B_i(\varepsilon)
=
\bigl\{
z_i:|z_i|\leq\varepsilon
\bigr\}.
\]
We require that the closed disks \(B_i(\varepsilon)\) are pairwise
disjoint, contain no zeros of \(q\) other than their centers, and
intersect the saddle connections only in radial segments near their
endpoints. We also choose \(\varepsilon\) sufficiently small so that
the conclusion of
Proposition~\ref{prop:comparison-parallel-transport-matrices}
holds on each \(B_i(\varepsilon)\).

Let \(m_i\) be any point in the interior of \(c_i\). We also assume that every \(m_i\) lies outside the disks
\(B_j(\varepsilon)\). Decompose
\[
c_i=\delta_i*\alpha_{i-1},
\]
where \(\delta_i\) is the segment from \(p_i\) to \(m_i\), and
\(\alpha_{i-1}\) is the segment from \(m_i\) to \(p_{i-1}\).

After deleting the interiors of the disks, define the truncated
segments
\[
\widetilde\delta_i
=
\delta_i\cap
\left(
X\setminus\bigcup_{j=1}^{\ell}B_j(\varepsilon)^\circ
\right),
\]
and
\[
\widetilde\alpha_{i-1}
=
\alpha_{i-1}\cap
\left(
X\setminus\bigcup_{j=1}^{\ell}B_j(\varepsilon)^\circ
\right).
\]
Thus the truncated portion of the saddle connection \(c_i\) is
\[
\widetilde c_i
=
\widetilde\delta_i*\widetilde\alpha_{i-1}.
\]

Near the zero \(p_i\), the geodesic \(c_\gamma\) enters the disk along
\(\widetilde\alpha_i\) and exits along
\(\widetilde\delta_i\). Let
\[
\widetilde\beta_i\subset\partial B_i(\varepsilon)
\]
be the counterclockwise boundary arc from the terminal point of
\(\widetilde\alpha_i\) to the initial point of
\(\widetilde\delta_i\).

More precisely, choose real lifts \(\theta_i\) and \(\theta_i'\) of
the arguments of these two endpoints such that
\[
\theta_i<\theta_i'<\theta_i+2\pi,
\]
and parametrize
\[
\widetilde\beta_i(\theta)
=
\varepsilon e^{\sqrt{-1}\theta},
\qquad
\theta_i\leq\theta\leq\theta_i'.
\]
Thus every arc \(\widetilde\beta_i\) is oriented counterclockwise.

We define the modified closed curve \(\widetilde c_\gamma\) by
replacing the passage of \(c_\gamma\) through each zero \(p_i\) with
the corresponding circular arc \(\widetilde\beta_i\). Starting at the
point \(m_{1}\), it is given by
\[
\begin{aligned}
\widetilde c_\gamma
={}&
\widetilde\alpha_\ell*
\widetilde\beta_\ell*
\widetilde\delta_\ell*
\widetilde\alpha_{\ell-1}*
\widetilde\beta_{\ell-1}*
\widetilde\delta_{\ell-1}
*\cdots*
\widetilde\alpha_1*
\widetilde\beta_1*
\widetilde\delta_1.
\end{aligned}
\]
The curve \(\widetilde c_\gamma\) is contained in
\(X\setminus Z(q)\). Moreover, it is homotopic to \(c_\gamma\)
relative to the base point \(m_1\). Hence
\[
\operatorname{Hol}_t(\widetilde c_\gamma)
=
\operatorname{Hol}_t(c_\gamma).
\]

See Figure~\ref{fig:global-modification} for the case of three zeros.

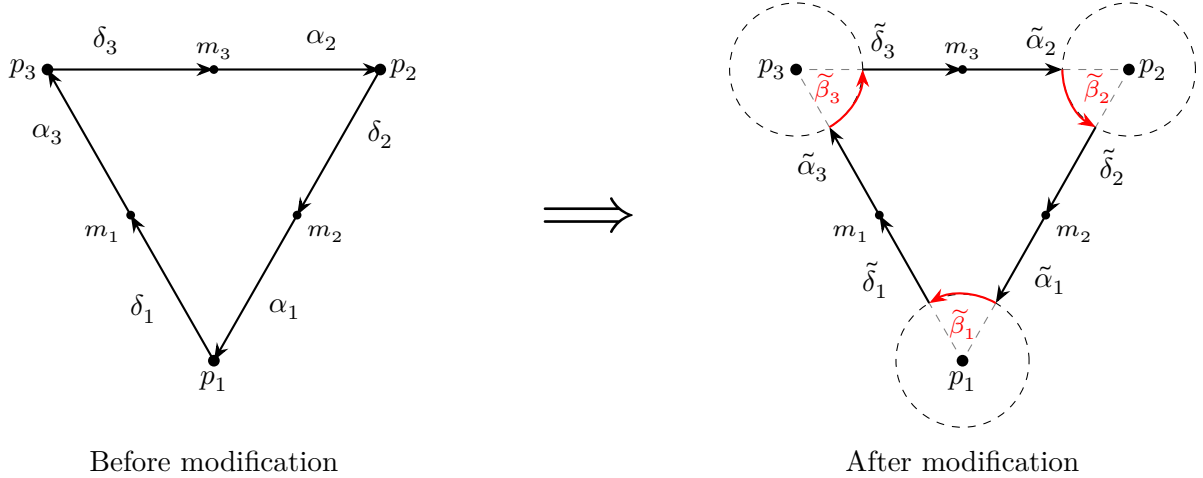
\begin{figure}[htbp]
\centering
\begin{tikzpicture}[scale=1.1,>=Stealth]

%------------------------------------------------
% Before modification
%------------------------------------------------
\begin{scope}
  \coordinate (p1) at (0,-2.0);
  \coordinate (p2) at (2.0,1.5);
  \coordinate (p3) at (-2.0,1.5);

  % c_i is oriented from p_i to p_{i-1}
  \coordinate (m1) at ($(p1)!0.5!(p3)$);
  \coordinate (m2) at ($(p2)!0.5!(p1)$);
  \coordinate (m3) at ($(p3)!0.5!(p2)$);

  % Oriented saddle connections
  \draw[->,thick] (m2) -- (p1);
  \draw[->,thick] (p1) -- (m1);

  \draw[->,thick] (m3) -- (p2);
  \draw[->,thick] (p2) -- (m2);

  \draw[->,thick] (m1) -- (p3);
  \draw[->,thick] (p3) -- (m3);

  % Labels
  \node at ($(m2)!0.5!(p1)+(0.35,-0.25)$)
        {$\alpha_1$};
  \node at ($(p1)!0.5!(m1)+(-0.35,-0.25)$)
        {$\delta_1$};

  \node at ($(m3)!0.5!(p2)+(0.30,0.35)$)
        {$\alpha_2$};
  \node at ($(p2)!0.5!(m2)+(0.50,0.10)$)
        {$\delta_2$};

  \node at ($(m1)!0.5!(p3)+(-0.50,0.10)$)
        {$\alpha_3$};
  \node at ($(p3)!0.5!(m3)+(-0.30,0.35)$)
        {$\delta_3$};

  % Zeros
  \fill (p1) circle (2pt) node[below] {$p_1$};
  \fill (p2) circle (2pt) node[right] {$p_2$};
  \fill (p3) circle (2pt) node[left]  {$p_3$};

  % Midpoints
  \fill (m1) circle (1.5pt)
    node[below left] {\footnotesize \(m_1\)};
  \fill (m2) circle (1.5pt)
    node[below right] {\footnotesize \(m_2\)};
  \fill (m3) circle (1.5pt)
    node[above] {\footnotesize \(m_3\)};

  \node at (0,-3.2) {Before modification};
\end{scope}

\node at (4.5,-0.25) {\Huge \(\Longrightarrow\)};

%------------------------------------------------
% After modification
%------------------------------------------------
\begin{scope}[xshift=9cm]
  \def\R{0.8}

  \coordinate (p1) at (0,-2.0);
  \coordinate (p2) at (2.0,1.5);
  \coordinate (p3) at (-2.0,1.5);

  % Disks around the zeros
  \foreach \i in {1,2,3} {
    \draw[dashed,thin] (p\i) circle (\R);
  }

  % Incoming and outgoing endpoints
  \coordinate (a1in)  at ($(p1)+(60:\R)$);
  \coordinate (a1out) at ($(p1)+(120:\R)$);

  \coordinate (a2in)  at ($(p2)+(180:\R)$);
  \coordinate (a2out) at ($(p2)+(240:\R)$);

  \coordinate (a3in)  at ($(p3)+(300:\R)$);
  \coordinate (a3out) at ($(p3)+(0:\R)$);

  % Midpoints of the truncated saddle connections
  \coordinate (m1) at ($(a1out)!0.5!(a3in)$);
  \coordinate (m2) at ($(a2out)!0.5!(a1in)$);
  \coordinate (m3) at ($(a3out)!0.5!(a2in)$);

  % Deleted radial pieces
  \draw[gray,dashed,thin]
    (a1in) -- (p1) -- (a1out);
  \draw[gray,dashed,thin]
    (a2in) -- (p2) -- (a2out);
  \draw[gray,dashed,thin]
    (a3in) -- (p3) -- (a3out);

  % Truncated straight segments
  \draw[->,thick] (m2) -- (a1in);
  \draw[->,thick] (a1out) -- (m1);

  \draw[->,thick] (m3) -- (a2in);
  \draw[->,thick] (a2out) -- (m2);

  \draw[->,thick] (m1) -- (a3in);
  \draw[->,thick] (a3out) -- (m3);

  % Labels of truncated segments
  \node at ($(m2)!0.5!(a1in)+(0.35,-0.25)$)
        {$\widetilde\alpha_1$};
  \node at ($(a1out)!0.5!(m1)+(-0.35,-0.25)$)
        {$\widetilde\delta_1$};

  \node at ($(m3)!0.5!(a2in)+(0.35,0.35)$)
        {$\widetilde\alpha_2$};
  \node at ($(a2out)!0.5!(m2)+(0.50,0.05)$)
        {$\widetilde\delta_2$};

  \node at ($(m1)!0.5!(a3in)+(-0.50,0.05)$)
        {$\widetilde\alpha_3$};
  \node at ($(a3out)!0.5!(m3)+(-0.35,0.35)$)
        {$\widetilde\delta_3$};

  % Counterclockwise circular arcs
  \draw[->,thick,red]
    (a1in) arc[start angle=60,end angle=120,radius=\R];

  \draw[->,thick,red]
    (a2in) arc[start angle=180,end angle=240,radius=\R];

  \draw[->,thick,red]
    (a3in) arc[start angle=300,end angle=360,radius=\R];

  % Labels inside the circular arcs
  \node[red,font=\footnotesize]
    at ($(p1)+(90:0.55*\R)$)
    {$\widetilde\beta_1$};

  \node[red,font=\footnotesize]
    at ($(p2)+(210:0.55*\R)$)
    {$\widetilde\beta_2$};

  \node[red,font=\footnotesize]
    at ($(p3)+(330:0.55*\R)$)
    {$\widetilde\beta_3$};

  % Zeros
  \fill (p1) circle (2pt) node[below] {$p_1$};
  \fill (p2) circle (2pt) node[right] {$p_2$};
  \fill (p3) circle (2pt) node[left]  {$p_3$};

  % Midpoints
  \fill (m1) circle (1.5pt)
    node[below left] {\footnotesize \(m_1\)};
  \fill (m2) circle (1.5pt)
    node[below right] {\footnotesize \(m_2\)};
  \fill (m3) circle (1.5pt)
    node[above] {\footnotesize \(m_3\)};

  \node at (0,-3.2) {After modification};
\end{scope}

\end{tikzpicture}
\caption{The geodesic \(c_\gamma\) and its modification near the
zeros, illustrated in the case of three zeros.}
\label{fig:global-modification}
\end{figure}
By multiplicativity of parallel transport,
\begin{equation}
\label{eq:holonomy-decomposition-modified-curve}
\operatorname{Hol}_t(\widetilde c_\gamma)
=
\prod_{i=1}^{\ell}
\left[
\operatorname{Hol}_t(\widetilde\delta_i)
\operatorname{Hol}_t(\widetilde\beta_i)
\operatorname{Hol}_t(\widetilde\alpha_i)
\right].
\end{equation}

To study the holonomy, we work in a natural coordinate \(w\) for
\(q\) and use the renormalized frame
\[
\mathcal F_{t,w}
=
\left(
t^{\frac{3}{8}}(\odif w)^{\frac{3}{2}},
\;
t^{\frac{1}{8}}(\odif w)^{\frac{1}{2}},
\;
t^{-\frac{1}{8}}(\odif w)^{-\frac{1}{2}},
\;
t^{-\frac{3}{8}}(\odif w)^{-\frac{3}{2}}
\right).
\]
Equivalently, this is the holomorphic frame induced by the
\(q_t\)-natural coordinate \(t^{1/4}w\). 
Note that changing between \(\mathcal F_w\) and
\(\mathcal F_{t,w}\) only grows polynomially in \(t\), does not affect the following estimate. All matrix asymptotics in this subsection
will nevertheless be stated in the renormalized frames
\(\mathcal F_{t,w}\).

\subsection{Asymptotic behavior of the holonomy factors}

We first analyze the holonomy along the straight pieces
\(\widetilde\alpha_i\) and \(\widetilde\delta_i\). These segments stay
a positive distance away from the zero set of \(q\), so a
single-valued natural coordinate \(w\) for \(q\) may be chosen on a
neighborhood of each of them. In this coordinate,
\[
q=\odif w^4,
\qquad
q_t=t\odif w^4.
\]
The constant model is diagonalized by the fixed matrix \(S\)
introduced in Section~\ref{sec:local-model}. The following is the
quartic analogue of
\cite[Proposition~5.6]{loftin2026limits}.

\begin{proposition}
\label{prop:holonomy-straight-pieces}
Let \(\widetilde\alpha_i\) and \(\widetilde\delta_i\) be the straight
pieces of the modified curve \(\widetilde c_\gamma\). Choose a natural
coordinate \(w\) for \(q\) on a neighborhood of each segment, and let
\[
\phi_k=\sqrt{-1}^{1-k}\odif w,
\qquad
k=1,2,3,4,
\]
be the four local fourth roots of \(q\). Define
\[
\tilde \mu_{i,k}
=
-2\operatorname{Re}\int_{\widetilde\alpha_i}\phi_k,
\qquad
\tilde \lambda_{i,k}
=
-2\operatorname{Re}\int_{\widetilde\delta_i}\phi_k,
\]
and set
\[
\tilde \mu_i=\max_{1\leq k\leq4}\tilde\mu_{i,k},
\qquad
\tilde \lambda_i=\max_{1\leq k\leq4}\tilde\lambda_{i,k}.
\]
Then, with respect to the frames \(\mathcal F_{t,w}\) induced by the
chosen natural coordinate at the endpoints,
\[
\operatorname{Hol}_t(\widetilde\alpha_i)
=
S
\operatorname{diag}
\left(
e^{t^{1/4}\tilde\mu_{i,1}},
e^{t^{1/4}\tilde\mu_{i,2}},
e^{t^{1/4}\tilde\mu_{i,3}},
e^{t^{1/4}\tilde\mu_{i,4}}
\right)
S^{-1}
+
o\left(e^{t^{1/4}\tilde\mu_i}\right),
\]
and
\[
\operatorname{Hol}_t(\widetilde\delta_i)
=
S
\operatorname{diag}
\left(
e^{t^{1/4}\tilde\lambda_{i,1}},
e^{t^{1/4}\tilde\lambda_{i,2}},
e^{t^{1/4}\tilde\lambda_{i,3}},
e^{t^{1/4}\tilde\lambda_{i,4}}
\right)
S^{-1}
+
o\left(e^{t^{1/4}\tilde\lambda_i}\right)
\]
as \(t\to+\infty\). Here \(S\) is the unitary matrix defined in
\eqref{eq:definition-of-S}, and the error estimates are taken in any
fixed matrix norm.
\end{proposition}

\begin{proof}
The statement follows from
\cite[Theorem~4.4]{collier2017asymptotics}. We briefly explain how it translates into our notation. Two minor differences between the conventions of \cite{collier2017asymptotics} and those used here must be taken into account.

First, the diagonalizing matrix used in
\cite{collier2017asymptotics} is not exactly the matrix \(S\) employed
here. This does not affect the result, since both matrices diagonalize
the same constant-model connection. After translating to our
convention, the leading diagonal exponential matrices are precisely
those appearing in this proposition.

Second, \cite[Theorem~4.4]{collier2017asymptotics} is stated for paths
satisfying certain local length restrictions. For a general straight
segment \(\gamma\) disjoint from the zero set, we subdivide it into
finitely many shorter collinear subsegments
\[
\gamma=\gamma_1*\cdots*\gamma_m,
\]
each of which satisfies those restrictions. Applying
\cite[Theorem~4.4]{collier2017asymptotics} to each subsegment gives
\begin{align*}
\operatorname{Hol}_t(\gamma_j)
&=
S
\begin{pmatrix}
e^{t^{1/4}\upsilon_{j,1}} & & & \\
& e^{t^{1/4}\upsilon_{j,2}} & & \\
& & e^{t^{1/4}\upsilon_{j,3}} & \\
& & & e^{t^{1/4}\upsilon_{j,4}}
\end{pmatrix}
\bigl(I+O(t^{-1/8})\bigr)
S^{-1}
\\
&=
S
\begin{pmatrix}
e^{t^{1/4}\upsilon_{j,1}} & & & \\
& e^{t^{1/4}\upsilon_{j,2}} & & \\
& & e^{t^{1/4}\upsilon_{j,3}} & \\
& & & e^{t^{1/4}\upsilon_{j,4}}
\end{pmatrix}
S^{-1}
+
o\bigl(e^{t^{1/4}\upsilon_j}\bigr),
\end{align*}
where
\[
\upsilon_{j,k}
=
-2\operatorname{Re}\int_{\gamma_j}\phi_k,
\qquad
\upsilon_j
=
\max_{1\leq k\leq4}\upsilon_{j,k}.
\]
Here the first equality is the estimate supplied by
\cite[Theorem~4.4]{collier2017asymptotics}, after translating to our
choice of diagonalizing matrix, while the second is the weaker form
needed below.

Because the subsegments are collinear and have the same orientation,
the ordering of the four quantities
\(\upsilon_{j,1},\ldots,\upsilon_{j,4}\) is independent of \(j\). Multiplying the parallel transports of the subsegments therefore gives the estimates we need.
\end{proof}
\begin{remark}
We emphasize that the estimate obtained after multiplying the local
parallel transports is weaker than the estimate for each individual
subsegment: only the fact that the accumulated error has lower
exponential order than the leading term is retained. This weaker
estimate is sufficient for us.
\end{remark}
We next analyze the asymptotic holonomy along the circular arc \(\widetilde\beta_i\). Let \(p_i\) be the corresponding zero of \(q\), of order \(d_i\), and choose a local coordinate \(z_i\), centered at \(p_i\), such that
\(
q=z_i^{d_i}\odif{z_i}^{4}.
\)
Write
\(
\widetilde\beta_i(\theta)
=
\varepsilon e^{\sqrt{-1}\theta}
\) for 
\(
\theta_i^-\leq\theta\leq\theta_i^+,
\)
and set
\(
N_i=d_i+4.
\)

We use the sectorial natural coordinates, stable sectors, and indexing
convention introduced in
Section~\ref{subsec:explicit-computation-for-z^d}. Suppose that the
initial point satisfies
\(
\widetilde\beta_i(\theta_i^-)\in J_{a_i}^{(j_i)}
\)
for some \(a_i\in\{0,1\}\) and
\(j_i\in\mathbb Z/N_i\mathbb Z\).

Since \(\widetilde\beta_i\) is oriented counterclockwise,starting from the
chart \(U_{j_i}\), the index increases by one whenever the arc enters to the next sectorial chart. Suppose that
\(
\widetilde\beta_i(\theta_i^+)
\in
J_{b_i}^{(j_i+k_i)}
\)
for some \(b_i\in\{0,1\}\) and \(k_i\geq0\) , where the superscript is reduced modulo \(N_i\) when referring to the actual sector.

Throughout this discussion,
\(
T_{a_ib_i}^{(k_i)}
\)
denotes precisely the transition matrix introduced in Proposition~\ref{prop:large-circular-holonomy}. Recall that \(\Psi_0\) is the fundamental matrix for the constant model \(q=\odif z^4\).

\begin{proposition}
\label{prop:holonomy-circular-arcs}
Assume that neither \(\theta_i\) nor \(\theta_i'\) is a Stokes direction. Then, as \(t\to+\infty\),
\[
\operatorname{Hol}_t(\widetilde\beta_i)
=
\Bigl(
\Psi_0
\bigl(
t^{1/4}w_{j_i+k_i}(\widetilde\beta_i(\theta_i^+))
\bigr)
\Bigr)^{-1}
\left(
S T_{a_ib_i}^{(k_i)}S^{-1}+o(I)
\right)
\Psi_0
\bigl(
t^{1/4}w_{j_i}(\widetilde\beta_i(\theta_i^-))
\bigr),
\]
with respect to the initial frame \(\mathcal F_{t,w_{j_i}}\) and the terminal frame \(\mathcal F_{t,w_{j_i+k_i}}\). Here \(T_{ab}^{(k)}\) is precisely the transition matrix introduced in Proposition~\ref{prop:large-circular-holonomy}.
\end{proposition}

\begin{proof}
Under the rescaling
\(
\tau=t^{1/N_i}z_i,
\)
we have
\[
t z_i^{d_i}\odif{z_i}^{4}
=
\tau^{d_i}\odif{\tau}^{4},
\]
and the arc \(\widetilde\beta_i\) is identified with the large circular arc \(t^{1/N_i}\beta_{i}\).

Moreover, the rescaling identifies the renormalized frame \(\mathcal F_{t,w_j}\) on the original \(z_i\)-disk with the canonical model frame \(\mathcal F_{w_j(\tau)}\) on the rescaled \(\tau\)-plane. Applying Proposition~\ref{prop:sectorial-comparison-natural-frames} and Proposition~\ref{prop:large-circular-holonomy} to the \(\tau\)-plane yields the proposition. We only point out that in Proposition~\ref{prop:large-circular-holonomy}, we use \(\Psi_0^{w_k}\). Note that \(\Psi_0^{w_k}(t^{1/N_i}z)=\Psi_0(w_k(t^{1/N_i}z))=\Psi_0(t^{1/4}w_k(z))\), which is the form used here.
\end{proof}

\subsection{Asymptotics of a single holonomy block}
\label{subsec:asymptotics-single-holonomy-block}
For a fixed index \(i\), we study the asymptotic behavior as \(t\to+\infty\) of the single holonomy block
\[
\operatorname{Hol}_t(\widetilde\delta_i)
\operatorname{Hol}_t(\widetilde\beta_i)
\operatorname{Hol}_t(\widetilde\alpha_i).
\]
We first explain how the diagonal terms arising from the straight pieces combine with the constant-model matrices that appear in the circular-arc asymptotics.

Let \(\gamma\) be an oriented straight segment contained in a natural
coordinate chart \(w\). Suppose that \(\gamma\) is parametrized by
\[
w(\gamma(s))
=
w_0+s e^{\sqrt{-1}\theta},
\qquad
0\leq s\leq L,
\]
where \(L\) is the \(|q|^{1/2}\)-length of \(\gamma\). For the four local fourth roots \(\phi_k = \sqrt{-1}^{1-k}\odif w\) (\(k=1,2,3,4\)), define
\[
\upsilon_k(\gamma) := -2\operatorname{Re}\int_\gamma\phi_k.
\]
We associate to \(\gamma\) the diagonal matrix
\begin{equation}
\label{eq:definition-D-gamma}
D_w(\gamma)
:=
\operatorname{diag}
\bigl(
\upsilon_1(\gamma),
\upsilon_2(\gamma),
\upsilon_3(\gamma),
\upsilon_4(\gamma)
\bigr)
\end{equation}
For the above parametrization, we have
\[
\int_\gamma\phi_k
=
\sqrt{-1}^{1-k}L e^{\sqrt{-1}\theta}.
\]
Hence
\[
\upsilon_k(\gamma)
=
-2L\cos\Bigl(\theta-\frac{(k-1)\pi}{2}\Bigr),
\]
and therefore
\begin{equation}
\label{eq:D-straight-segment-explicit}
D_w(\gamma)
=
\operatorname{diag}
\left(
-2L\cos\theta,
-2L\sin\theta,
2L\cos\theta,
2L\sin\theta
\right).
\end{equation}
In particular, the position of the largest diagonal entry of
\(D_w(\gamma)\) depends only on the direction angle \(\theta\) of
\(\gamma\) in the natural coordinate \(w\).

Recall that the fundamental matrix \(
\Psi_0\) of the constant model is given by 
\begin{equation}
    \Psi_0(z)
       = S\exp \left(D(z)\right)S^{-1} 
       = S\operatorname{diag}
       \left(
        e^{2|z|\cos\theta},e^{2|z|\sin\theta},e^{-2|z|\cos\theta},e^{-2|z|\sin\theta}
       \right)
        S^{-1}.
 \end{equation}
for \(z=|z|e^{\sqrt{-1}\theta}\). 

If \(x\) and \(y\)
are respectively the initial and terminal points of \(\gamma\), then
\[
w(y)-w(x)=Le^{\sqrt{-1}\theta}.
\]
Hence 
\begin{equation}
\label{eq:D-gamma-Psi-zero}
S\exp\bigl(t^{1/4}D_w(\gamma)\bigr)S^{-1}
=
\Psi_0\bigl(t^{1/4}w(y)\bigr)^{-1}
\Psi_0\bigl(t^{1/4}w(x)\bigr).
\end{equation}
Although the natural-coordinate frame is not defined at a zero of
\(q\), the sectorial natural coordinate \(w\) extends continuously to
the zero, with value \(w=0\). Moreover, the constant-model matrix
\(\Psi_0\) is defined on the whole complex plane and satisfies
\(
\Psi_0(0)=I.
\)
Consequently, \eqref{eq:D-gamma-Psi-zero} continues to hold when one of the endpoints of \(\gamma\) is a zero.

We will also use the additivity of \(D_w\). Suppose that
\[
\gamma=\gamma_1*\gamma_2
\]
is the concatenation of two collinear straight segments with the same orientation, expressed in the same natural coordinate. Then
\[
D_w(\gamma)
=
D_w(\gamma_1)+D_w(\gamma_2).
\]
Since these matrices are diagonal, it follows that
\begin{equation}
\label{eq:additivity-D-straight-segments}
\exp\bigl(t^{1/4}D_w(\gamma)\bigr)
=
\exp\bigl(t^{1/4}D_w(\gamma_1)\bigr)
\exp\bigl(t^{1/4}D_w(\gamma_2)\bigr).
\end{equation}

For each zero \(p_i\), let
\[
x_i^-
=
\widetilde\alpha_i\cap\partial B_i(\varepsilon),
\qquad
x_i^+
=
\widetilde\delta_i\cap\partial B_i(\varepsilon).
\]
The circular arc \(\widetilde\beta_i\) is oriented from \(x_i^-\) to
\(x_i^+\). As in Proposition~\ref{prop:holonomy-circular-arcs}, let
\(J_{a_i}^{(j_i)}\) be the stable sector containing the initial point \(x_i^-\). Thus the natural coordinate used at \(x_i^-\) is
\(w_{j_i}\). Since \(\widetilde\beta_i\) is oriented
counterclockwise, the sector index increases along the arc. Let
\(k_i\) denote the relative increase of the sector index. Then the terminal point \(x_i^+\) lies in
\(
J_{b_i}^{(j_i+k_i)}
\)
for some \(b_i\in\{0,1\}\), and the natural coordinate used at
\(x_i^+\) is \(w_{j_i+k_i}\), with the sector index understood modulo
\(N_i\).

Accordingly, we set
\[
w_i^-:=w_{j_i},
\qquad
w_i^+:=w_{j_i+k_i}.
\]
It follows from \eqref{eq:D-gamma-Psi-zero} that, the leading terms in Proposition~\ref{prop:holonomy-straight-pieces} are 
\begin{align}
S\exp\bigl(t^{1/4}D_{w_i^-}(\widetilde\alpha_i)\bigr)S^{-1}
&=
\Psi_0\bigl(t^{1/4}w_i^-(x_i^-)\bigr)^{-1}
\Psi_0\bigl(t^{1/4}w_i^-(m_{i+1})\bigr),
\label{eq:alpha-Psi-zero}
\\
S\exp\bigl(t^{1/4}D_{w_i^+}(\widetilde\delta_i)\bigr)S^{-1}
&=
\Psi_0\bigl(t^{1/4}w_i^+(m_i)\bigr)^{-1}
\Psi_0\bigl(t^{1/4}w_i^+(x_i^+)\bigr).
\label{eq:delta-Psi-zero}
\end{align}
Here \(m_i\) is the chosen point in the interior of \(c_i\). Note that the indices are understood modulo \(\ell\).

We first consider the case in which neither endpoint of
\(\widetilde\beta_i\) lies on a Stokes direction; that is, both
\(x_i^-\) and \(x_i^+\) are assumed to avoid the Stokes directions.
Other cases will
be treated separately later.

\begin{proposition}
\label{prop:asymptotics-single-holonomy-block-no-Stokes-directions}
Under the above assumption that neither endpoint of
\(\widetilde\beta_i\) lies on a Stokes direction and the notation
of Proposition~\ref{prop:holonomy-straight-pieces}. With respect to the frame
\(\mathcal F_{t,w_i^-}\) at \(m_{i+1}\) and the frame
\(\mathcal F_{t,w_i^+}\) at \(m_i\), we have
    \[
\operatorname{Hol}_t(\widetilde\delta_i)
 \operatorname{Hol}_t(\widetilde\beta_i)
 \operatorname{Hol}_t(\widetilde\alpha_i)=
S\exp\bigl(t^{1/4}D_{w_i^+}(\delta_i)\bigr)
T_i
\exp\bigl(t^{1/4}D_{w_i^-}(\alpha_i)\bigr)
S^{-1}
+
o\left(
e^{t^{1/4}(\lambda_i+\mu_i)}
\right)
\]
as \(t\to+\infty\). Here 
\[
\lambda_i:=
\max_{1\leq k\leq4}
\left\{
-2\operatorname{Re}\int_{\delta_i}\phi_k
\right\}, \qquad 
\mu_i:=
\max_{1\leq k\leq4}
\left\{
-2\operatorname{Re}\int_{\alpha_i}\phi_k
\right\}.
\]
And \(
T_i=T_{a_i b_i}^{(k_i)}
\)
is precisely the transition matrix defined in
Proposition~\ref{prop:large-circular-holonomy}.
\end{proposition}
\begin{proof}
Let \(\alpha_i^0\) and \(\delta_i^0\) denote the radial segments inside
\(B_i(\varepsilon)\), oriented so that
\[
\alpha_i
=
\widetilde\alpha_i * \alpha_i^0,
\qquad
\delta_i
=
\delta_i^0 * \widetilde\delta_i.
\]
Thus \(\alpha_i^0\) runs from \(x_i^-\) to \(p_i\), whereas
\(\delta_i^0\) runs from \(p_i\) to \(x_i^+\).

Under the assumption on the endpoints,
Proposition~\ref{prop:holonomy-circular-arcs} gives
\[
\operatorname{Hol}_t(\widetilde\beta_i)
=
\Psi_0\bigl(t^{1/4}w_i^+(x_i^+)\bigr)^{-1}
\left(
S T_iS^{-1}+o(I)
\right)
\Psi_0\bigl(t^{1/4}w_i^-(x_i^-)\bigr).
\]
Since
\[
w_i^+(p_i)=w_i^-(p_i)=0
\qquad\text{and}\qquad
\Psi_0(0)=I,
\]
formula~\eqref{eq:D-gamma-Psi-zero}, applied to the two radial
segments, yields
\[
\Psi_0\bigl(t^{1/4}w_i^+(x_i^+)\bigr)^{-1}
=
S\exp\bigl(t^{1/4}D_{w_i^+}(\delta_i^0)\bigr)S^{-1}, \qquad
\Psi_0\bigl(t^{1/4}w_i^-(x_i^-)\bigr)
=
S\exp\bigl(t^{1/4}D_{w_i^-}(\alpha_i^0)\bigr)S^{-1}.
\]
Consequently,
\begin{equation}
\label{eq:circular-holonomy-radial-D-form}
\operatorname{Hol}_t(\widetilde\beta_i)
=
S\exp\bigl(t^{1/4}D_{w_i^+}(\delta_i^0)\bigr)
\left(
T_i+o(I)
\right)
\exp\bigl(t^{1/4}D_{w_i^-}(\alpha_i^0)\bigr)
S^{-1}.
\end{equation}

By Proposition~\ref{prop:holonomy-straight-pieces}and \eqref{eq:D-gamma-Psi-zero}, we obtain
\begin{align*}
\operatorname{Hol}_t(\widetilde\alpha_i)
&=
S\exp\bigl(t^{1/4}D_{w_i^-}(\widetilde\alpha_i)\bigr)S^{-1}
+
o\left(e^{t^{1/4}\widetilde \mu_i}\right),
\\
\operatorname{Hol}_t(\widetilde\delta_i)
&=
S\exp\bigl(t^{1/4}D_{w_i^+}(\widetilde\delta_i)\bigr)S^{-1}
+
o\left(e^{t^{1/4}\widetilde\lambda_i}\right),
\end{align*}
The segments \(\delta_i^0\) and \(\widetilde\delta_i\) are collinear
and have the same orientation. The same is true of
\(\widetilde\alpha_i\) and \(\alpha_i^0\). Hence the ordering of the
four diagonal entries is unchanged along each pair, and
\[
D_{w_i^+}(\delta_i)
=
D_{w_i^+}(\delta_i^0)+D_{w_i^+}(\widetilde\delta_i),
\qquad
D_{w_i^-}(\alpha_i)
=
D_{w_i^-}(\widetilde\alpha_i)+D_{w_i^-}(\alpha_i^0).
\]
Since all these matrices are diagonal, their exponentials commute.
Multiplying the leading terms in the three holonomy factors therefore
gives
\begin{align*}
&S\exp\bigl(t^{1/4}D_{w_i^+}(\widetilde\delta_i)\bigr)
 \exp\bigl(t^{1/4}D_{w_i^+}(\delta_i^0)\bigr)
 T_i
 \exp\bigl(t^{1/4}D_{w_i^-}(\alpha_i^0)\bigr)
 \exp\bigl(t^{1/4}D_{w_i^-}(\widetilde\alpha_i)\bigr)
 S^{-1}
\\
=&
S\exp\bigl(t^{1/4}D_{w_i^+}(\delta_i)\bigr)
T_i
\exp\bigl(t^{1/4}D_{w_i^-}(\alpha_i)\bigr)
S^{-1}.
\end{align*}

It remains to estimate the error terms.
 Note that the maximal
exponential weights of the radial portions are \(\lambda_i-\widetilde\lambda_i\) and \(\mu_i-\widetilde\mu_i\), respectively. Thus we can rewrite \(\operatorname{Hol}_t(\widetilde\beta_i)\) as
\[
\operatorname{Hol}_t(\widetilde\beta_i)=O\left(
e^{t^{1/4}
\left[
(\lambda_i-\widetilde\lambda_i)
+
(\mu_i-\widetilde\mu_i)
\right]}
\right) + o\left(
e^{t^{1/4}
\left[
(\lambda_i-\widetilde\lambda_i)
+
(\mu_i-\widetilde\mu_i)
\right]}
\right).
\]
Similarly,
\begin{align*}
\operatorname{Hol}_t(\widetilde\alpha_i)
&=
O\left(e^{t^{1/4}\widetilde \mu_i}\right)
+
o\left(e^{t^{1/4}\widetilde \mu_i}\right),
\\
\operatorname{Hol}_t(\widetilde\delta_i)
&=
O\left(e^{t^{1/4}\widetilde\lambda_i}\right)
+
o\left(e^{t^{1/4}\widetilde\lambda_i}\right),
\end{align*}
Note that multiplying the three
asymptotic expansions, every term containing at least one
lower-order factor is
\(
o\left(e^{t^{1/4}(\lambda_i+\mu_i)}\right).
\)
Hence we obtain this proposition.
\end{proof}
We now identify the dominant contribution to the leading term of a
single holonomy block. Under the assumption that neither
endpoint of \(\widetilde\beta_i\) lies on a Stokes direction, the
directions of the full straight segments \(\alpha_i\) and
\(\delta_i\) also avoid the Stokes directions in the corresponding
natural coordinates. By the explicit formula
\eqref{eq:D-straight-segment-explicit} and Definition~\ref{def:stokes-directions} for Stokes directions, each of the diagonal matrices
\(D(\alpha_i)\) and \(D(\delta_i)\) therefore has a unique largest
entry.

\begin{proposition}
\label{prop:dominant-term-single-holonomy-block-no-Stokes-directions}
Assume that neither endpoint of \(\widetilde\beta_i\) lies on a Stokes
direction, namely,
\[
x_i^-\in J_{a_i}^{(j_i)},
\qquad
x_i^+\in J_{b_i}^{(j_i+k_i)},
\]
for some \(a_i,b_i\in\{0,1\}\).

Then the incoming and outgoing maximal index sets are given by
\[
\begin{array}{c|cc}
a_i & 0 & 1\\
\hline
\mathcal S_i & \{4\} & \{1\}
\end{array}
\qquad\text{and}\qquad
\begin{array}{c|cc}
b_i & 0 & 1\\
\hline
\mathcal R_i & \{2\} & \{3\}.
\end{array}
\]
and the geodesic condition gives
\[
4\leq 2k_i+b_i-a_i\leq 2N_i-4,
\]
or equivalently,
\[
2+a_i(1-b_i)
\leq k_i\leq
N_i-2-(1-a_i)b_i.
\]
Moreover, with respect to the frame
\(\mathcal F_{t,w_i^-}\) at \(m_{i+1}\) and the frame
\(\mathcal F_{t,w_i^+}\) at \(m_i\), as \(t\to+\infty\),
\[
\operatorname{Hol}_t(\widetilde\delta_i)
\operatorname{Hol}_t(\widetilde\beta_i)
\operatorname{Hol}_t(\widetilde\alpha_i)
=
e^{t^{1/4}(\lambda_i+\mu_i)}
S\mathcal C_iS^{-1}
+
o\left(
e^{t^{1/4}(\lambda_i+\mu_i)}
\right),
\]
where
\[
\mathcal C_i
=
\sum_{\substack{r\in\mathcal R_i\\ s\in\mathcal S_i}}
\bigl(T^{k_i}\bigr)_{r,s}E_{r,s}.
\]
Here \(T^{k_i}\) is the transition matrix defined in
Proposition~\ref{prop:large-circular-holonomy}.
\end{proposition}
\begin{proof} 
Note that the initial and terminal points of
\(\widetilde\beta_i\) satisfy
\[
x_i^-\in J_{a_i}^{(j_i)},
\qquad
x_i^+\in J_{b_i}^{(j_i+k_i)},
\]
for some \(a_i,b_i\in\{0,1\}\). And 
\(
w_i^-=w_{j_i}, w_i^+=w_{j_i+k_i}.
\)

Since the path is locally geodesic at \(p_i\), both angles determined
by the incoming and outgoing rays are at least \(\pi\). In particular,
the angle swept out by \(\widetilde\beta_i\) is at least \(\pi\), and
so is its complementary angle.

Since two consecutive Stokes directions are separated by an angle
\(\pi/4\), both \(\widetilde\beta_i\) and its complementary arc cross
at least four Stokes directions. Let
\(
N_i=d_i+4.
\)
The total cone angle at \(p_i\) is
\(
\frac{N_i\pi}{2},
\)
and hence there are \(2N_i\) Stokes directions around \(p_i\).

The number of Stokes directions crossed from
\(J_{a_i}^{(j_i)}\) to \(J_{b_i}^{(j_i+k_i)}\) is
\[
2k_i+b_i-a_i.
\]
The complementary arc therefore crosses
\[
2N_i-\bigl(2k_i+b_i-a_i\bigr)
\]
Stokes directions. Consequently,
\[
4
\leq
2k_i+b_i-a_i
\leq
2N_i-4.
\]
Since \(a_i,b_i\in\{0,1\}\) and \(k_i\) is an integer, this is
equivalent to
\[
k_i\in
\left\{
2+a_i(1-b_i),\ldots,
N_i-2-(1-a_i)b_i
\right\}.
\]

Now, we identify the unique indices \(r_i,s_i\in\{1,2,3,4\}\) such that
\[
\lambda_{i,r_i}=\lambda_i,
\qquad
\mu_{i,s_i}=\mu_i.
\]
Let
\[
\vartheta_i^-:=\arg w_{j_i}(x_i^-),
\qquad
\vartheta_i^+:=\arg w_{j_i+k_i}(x_i^+).
\]
The segment \(\alpha_i\) is oriented toward \(p_i\), so its direction
in the coordinate \(w_{j_i}\) is \(\vartheta_i^-+\pi\). If \(a_i=0\),
then
\[
-\frac{\pi}{2}
<
\vartheta_i^-
<
-\frac{\pi}{4},
\]
and hence
\[
\frac{\pi}{2}
<
\vartheta_i^-+\pi
<
\frac{3\pi}{4}.
\]
By \eqref{eq:D-straight-segment-explicit}, the unique largest entry of
\(D_{w_i^-}(\alpha_i)=D_{w_{j_i}}(\alpha_i)\) occurs in position \(4\). If \(a_i=1\), then
\[
\frac{3\pi}{4}
<
\vartheta_i^-+\pi
<
\pi,
\]
and the unique largest entry occurs in position \(1\). Therefore,
\[
s_i=
\begin{cases}
4, & a_i=0,\\
1, & a_i=1.
\end{cases}
\]

Similarly, \(\delta_i\) is oriented away from \(p_i\), so its direction
in the coordinate \(w_{j_i+k_i}\) is \(\vartheta_i^+\). Applying
\eqref{eq:D-straight-segment-explicit} again, we obtain
\[
r_i=
\begin{cases}
2, & b_i=0,\\
3, & b_i=1.
\end{cases}
\]

Recalling the explicit form of \(U_1\) from
Proposition~\ref{prop:explicit-U}, Proposition~\ref{prop:large-circular-holonomy}
gives
\[
(T_i)_{r_i,s_i}
=
\begin{cases}
\bigl(T^{k_i}\bigr)_{2,4},
&
(a_i,b_i)=(0,0),
\qquad 2\leq k_i\leq N_i-2,
\\
\bigl(U_1^{-1}T^{k_i}\bigr)_{3,4}
=
\bigl(T^{k_i}\bigr)_{3,4},
&
(a_i,b_i)=(0,1),
\qquad 2\leq k_i\leq N_i-3,
\\
\bigl(T^{k_i}U_1\bigr)_{2,1}
=
\bigl(T^{k_i}\bigr)_{2,1},
&
(a_i,b_i)=(1,0),
\qquad 3\leq k_i\leq N_i-2,
\\
\bigl(U_1^{-1}T^{k_i}U_1\bigr)_{3,1}
=
\bigl(T^{k_i}\bigr)_{3,1},
&
(a_i,b_i)=(1,1),
\qquad 2\leq k_i\leq N_i-2.
\end{cases}
\]
Combining with Proposition~\ref{prop:asymptotics-single-holonomy-block-no-Stokes-directions}, we obtain this proposition.
\end{proof}

\begin{remark}
This is slightly simpler than the cubic setting considered in \cite{loftin2026limits}, where the corresponding
assumptions do not by themselves guarantee a unique dominant diagonal
entry.
\end{remark}

We now remove the temporary assumption that the endpoints of the
circular arcs avoid the Stokes directions, following the perturbation
argument of
\cite[Proposition~5.13]{loftin2026limits}.

For every endpoint of \(\widetilde\beta_i\) lying on a Stokes
direction, choose \(\eta>0\) sufficiently small and perturb its angle
by \(\pm\eta\). Endpoints that do not lie on Stokes directions are
left unchanged. Denote the resulting initial and terminal points by
\[
x_i^{-,\eta},
\qquad
x_i^{+,\eta},
\]
respectively, and let \(\widetilde\beta_i^\eta\) be the
counterclockwise circular arc from \(x_i^{-,\eta}\) to
\(x_i^{+,\eta}\).

We modify the adjacent straight pieces at the same time. 
Let
\(
\widetilde\alpha_i^\eta
\)
be the straight segment joining \(m_{i+1}\) to
\(x_i^{-,\eta}\), and let
\(
\widetilde\delta_i^\eta
\)
be the straight segment joining \(x_i^{+,\eta}\) to \(m_i\).
Thus the original truncated pieces are replaced directly by the
straight segments connecting the same midpoints to the perturbed
endpoints of the circular arcs. Let
\(
\alpha_i^{0,\eta}
\)
be the straight radial segment in the local \(z_i\)-coordinate from
\(x_i^{-,\eta}\) to \(p_i\), and let
\(
\delta_i^{0,\eta}
\)
be the straight radial segment from \(p_i\) to
\(x_i^{+,\eta}\). Define the perturbed full segments by
\[
\alpha_i^\eta
=
\widetilde\alpha_i^\eta * \alpha_i^{0,\eta},
\qquad
\delta_i^\eta
=
\delta_i^{0,\eta} * \widetilde\delta_i^\eta.
\]
Thus \(\alpha_i^\eta\) is oriented from \(m_{i+1}\) to \(p_i\),
whereas \(\delta_i^\eta\) is oriented from \(p_i\) to \(m_i\). See Figure~\ref{fig:perturbation-single-holonomy-block}.
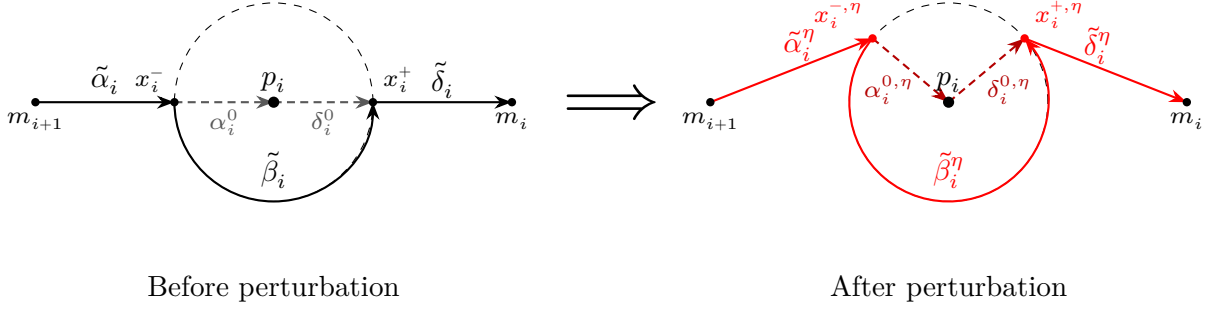
\begin{figure}[htbp]
\centering
\begin{tikzpicture}[scale=1.05,>=Stealth]

% ================================================================
% Before perturbation
% ================================================================
\begin{scope}
  \def\R{1.25}

  \coordinate (p)    at (0,0);
  \coordinate (mL)   at (-3,0);
  \coordinate (mR)   at (3,0);
  \coordinate (xin)  at ($(p)+(180:\R)$);
  \coordinate (xout) at ($(p)+(0:\R)$);

  % Boundary circle
  \draw[dashed,thin] (p) circle (\R);

  % Radial pieces alpha_i^0 and delta_i^0
  \draw[->,thick,densely dashed,gray!70!black]
    (xin) -- (p);
  \draw[->,thick,densely dashed,gray!70!black]
    (p) -- (xout);

  % Holonomy block
  \draw[->,thick] (mL) -- (xin);
  \draw[->,thick]
    (xin) arc[start angle=180,end angle=360,radius=\R];
  \draw[->,thick] (xout) -- (mR);

  % Points
  \fill (p) circle (2pt)
    node[above] {$p_i$};
  \fill (mL) circle (1.5pt)
    node[below] {\footnotesize \(m_{i+1}\)};
  \fill (mR) circle (1.5pt)
    node[below] {\footnotesize \(m_i\)};
  \fill (xin) circle (1.5pt)
    node[above left] {\footnotesize \(x_i^-\)};
  \fill (xout) circle (1.5pt)
    node[above right] {\footnotesize \(x_i^+\)};

  % Labels for the truncated pieces
  \node at ($(mL)!0.5!(xin)+(0,0.28)$)
    {$\widetilde\alpha_i$};
  \node at ($(xout)!0.5!(mR)+(0,0.28)$)
    {$\widetilde\delta_i$};
  \node at ($(p)+(270:0.68*\R)$)
    {$\widetilde\beta_i$};

  % Labels for the radial pieces
  \node[gray!70!black]
    at ($(xin)!0.5!(p)+(0,-0.25)$)
    {\footnotesize \(\alpha_i^0\)};
  \node[gray!70!black]
    at ($(p)!0.5!(xout)+(0,-0.25)$)
    {\footnotesize \(\delta_i^0\)};

  \node at (0,-2.35) {Before perturbation};
\end{scope}

\node at (4.25,0) {\Huge \(\Longrightarrow\)};

% ================================================================
% After perturbation
% ================================================================
\begin{scope}[xshift=8.5cm]
  \def\R{1.25}

  \coordinate (p)       at (0,0);
  \coordinate (mL)      at (-3,0);
  \coordinate (mR)      at (3,0);
  \coordinate (xineta)  at ($(p)+(140:\R)$);
  \coordinate (xouteta) at ($(p)+(40:\R)$);

  % Boundary circle
  \draw[dashed,thin] (p) circle (\R);

  % Radial pieces alpha_i^{0,eta} and delta_i^{0,eta}
  \draw[->,thick,densely dashed,red!70!black]
    (xineta) -- (p);
  \draw[->,thick,densely dashed,red!70!black]
    (p) -- (xouteta);

  % Perturbed holonomy block
  \draw[->,thick,red] (mL) -- (xineta);
  \draw[->,thick,red]
    (xineta) arc[start angle=140,end angle=400,radius=\R];
  \draw[->,thick,red] (xouteta) -- (mR);

  % Points
  \fill (p) circle (2pt)
    node[above] {$p_i$};
  \fill (mL) circle (1.5pt)
    node[below] {\footnotesize \(m_{i+1}\)};
  \fill (mR) circle (1.5pt)
    node[below] {\footnotesize \(m_i\)};
  \fill[red] (xineta) circle (1.5pt)
    node[above left] {\footnotesize \(x_i^{-,\eta}\)};
  \fill[red] (xouteta) circle (1.5pt)
    node[above right] {\footnotesize \(x_i^{+,\eta}\)};

  % Labels for the perturbed truncated pieces
  \node[red] at ($(mL)!0.55!(xineta)+(0,0.28)$)
    {$\widetilde\alpha_i^\eta$};
  \node[red] at ($(xouteta)!0.45!(mR)+(0,0.28)$)
    {$\widetilde\delta_i^\eta$};
  \node[red] at ($(p)+(270:0.68*\R)$)
    {$\widetilde\beta_i^\eta$};

  % Labels for the perturbed radial pieces
  \node[red!70!black]
    at ($(xineta)!0.5!(p)+(-0.28,-0.24)$)
    {\footnotesize \(\alpha_i^{0,\eta}\)};
  \node[red!70!black]
    at ($(p)!0.5!(xouteta)+(0.28,-0.24)$)
    {\footnotesize \(\delta_i^{0,\eta}\)};

  \node at (0,-2.35) {After perturbation};
\end{scope}

\end{tikzpicture}
\caption{Perturbation near a Stokes direction.}
\label{fig:perturbation-single-holonomy-block}
\end{figure}

For \(\eta>0\) sufficiently small, the perturbed path
\(
\widetilde\alpha_i^\eta
*
\widetilde\beta_i^\eta
*
\widetilde\delta_i^\eta
\)
is homotopic to
\(
\widetilde\alpha_i
*
\widetilde\beta_i
*
\widetilde\delta_i
\)
relative to the endpoints \(m_{i+1}\) and \(m_i\). Since \(D_t\) is
flat, their parallel transports coincide. Therefore,
\[
\operatorname{Hol}_t(\widetilde\delta_i)
 \operatorname{Hol}_t(\widetilde\beta_i)
 \operatorname{Hol}_t(\widetilde\alpha_i)
 =
\operatorname{Hol}_t(\widetilde\delta_i^\eta)
\operatorname{Hol}_t(\widetilde\beta_i^\eta)
\operatorname{Hol}_t(\widetilde\alpha_i^\eta).
\]
Fix a natural coordinate \(w\), the four roots are 
\[
\phi_k=(\sqrt{-1})^{1-k}\odif w,
\qquad k=1,2,3,4.
\]
For the perturbed full segments, define
\[
\mu_{i,k}^{\eta}
:=
-2\operatorname{Re}\int_{\alpha_i^\eta}\phi_k,
\qquad
\lambda_{i,k}^{\eta}
:=
-2\operatorname{Re}\int_{\delta_i^\eta}\phi_k,
\]
and set
\[
\mu_i^\eta
:=
\max_{1\leq k\leq4}\mu_{i,k}^\eta,
\qquad
\lambda_i^\eta
:=
\max_{1\leq k\leq4}\lambda_{i,k}^\eta.
\]
For \(\eta>0\) sufficiently small, the paths
\(\alpha_i^\eta\) and \(\alpha_i\) are also homotopic relative to their endpoints. The same
holds for \(\delta_i^\eta\) and \(\delta_i\). Since \(\phi_k\)
is exact, we have
\[
\int_{\alpha_i^\eta}\phi_k
=
\int_{\alpha_i}\phi_k,
\qquad
\int_{\delta_i^\eta}\phi_k
=
\int_{\delta_i}\phi_k.
\]
Consequently,
\[
D_w(\alpha_i^\eta)=D_w(\alpha_i),
\qquad
D_w(\delta_i^\eta)=D_w(\delta_i).
\]
In particular,
\(
\mu_{i,k}^\eta=\mu_{i,k},
\lambda_{i,k}^\eta=\lambda_{i,k},
\)
and hence
\(
\mu_i^\eta=\mu_i,
\lambda_i^\eta=\lambda_i.
\)
\begin{proposition}
\label{prop:asymptotics-single-holonomy-block-Stokes-directions}
Let \(\eta>0\) be sufficiently small, and let
\(
\widetilde\alpha_i^\eta,
\widetilde\beta_i^\eta,
\widetilde\delta_i^\eta
\)
be an admissible perturbation constructed above. Suppose that
\[
x_i^{-,\eta}\in J_{a_i^\eta}^{(j_i^\eta)},
\qquad
x_i^{+,\eta}\in
J_{b_i^\eta}^{(j_i^\eta+k_i^\eta)},
\]
where \(a_i^\eta,b_i^\eta\in\{0,1\}\). Choose the corresponding
sectorial natural coordinates
\[
w_i^{-,\eta}:=w_{j_i^\eta},
\qquad
w_i^{+,\eta}:=w_{j_i^\eta+k_i^\eta},
\]
and set
\[
T_i^\eta
:=
T_{a_i^\eta b_i^\eta}^{(k_i^\eta)}.
\]
Then,  with respect to the frame
\(\mathcal F_{t,w_i^{-,\eta}}\) at \(m_{i+1}\) and the frame
\(\mathcal F_{t,w_i^{+,\eta}}\) at \(m_i\),
 \[
\operatorname{Hol}_t(\widetilde\delta_i)
 \operatorname{Hol}_t(\widetilde\beta_i
 \operatorname{Hol}_t(\widetilde\alpha_i)
 =
S
\exp\left(
t^{1/4}D_{w_i^{+,\eta}}(\delta_i)
\right)
T_i^\eta
\exp\left(
t^{1/4}D_{w_i^{-,\eta}}(\alpha_i)
\right)
S^{-1}
+
o\left(
e^{t^{1/4}(\lambda_i+\mu_i)}
\right)
\]
as \(t\to+\infty\). Here 
\[
\lambda_i:=
\max_{1\leq k\leq4}
\left\{
-2\operatorname{Re}\int_{\delta_i}\phi_k
\right\}, \qquad 
\mu_i:=
\max_{1\leq k\leq4}
\left\{
-2\operatorname{Re}\int_{\alpha_i}\phi_k
\right\}.
\]
\end{proposition}
\begin{proof}
The result follows immediately from
Proposition~\ref{prop:asymptotics-single-holonomy-block-no-Stokes-directions} and the discussion above.
\end{proof}
Before stating the dominant-term asymptotics, we point out an
essential difference from the non-Stokes case: If an original endpoint lies on a
Stokes direction, the maximal diagonal entry of the corresponding
exponential matrix need not be attained at a unique index.

For the computation of the dominant term, we now fix the signs of the
endpoint perturbations as follows. At each endpoint on the Stokes directions, choose the unique sectorial
natural coordinate \((U_j,w_j)\) in which the corresponding radial ray from \(p_i\)
has argument either
\(
-\pi/4
\)
or
\(
-\pi/2.
\)
If the ray has argument \(-\pi/4\), we perturb it to
\(
-\frac{\pi}{4}-\eta,
\)
so that the perturbed endpoint lies in the first stable sector
\(J_0^{(j)}\). If the ray has argument \(-\pi/2\), we perturb it to
\(
-\pi/2+\eta.
\)
Thus, both perturbed endpoints lie in the stable sector
 \(J_0^{(j)}\).
\begin{proposition}
\label{prop:dominant-term-single-holonomy-block-Stokes-directions}
Under the assumptions of
Proposition~\ref{prop:asymptotics-single-holonomy-block-Stokes-directions}, and the perturbation convention above. Then the incoming maximal index set is given by
\[
\begin{array}{c|cccc}
\text{incoming direction}
&
-\dfrac{\pi}{4}
&
-\dfrac{\pi}{2}
&
J_0
&
J_1
\\[1ex]
\hline
\mathcal S_i^\eta
&
\{1,4\}
&
\{4\}
&
\{4\}
&
\{1\}
\end{array}
\]
and the outgoing maximal index set is given by
\[
\begin{array}{c|cccc}
\text{outgoing direction}
&
-\dfrac{\pi}{4}
&
-\dfrac{\pi}{2}
&
J_0
&
J_1
\\[1ex]
\hline
\mathcal R_i^\eta
&
\{2,3\}
&
\{2\}
&
\{2\}
&
\{3\}.
\end{array}
\]
Here \(J_0\) and \(J_1\) indicate that the corresponding endpoint is
non-Stokes and lies in a stable sector of type \(J_0\) or \(J_1\),
respectively.
Let
\(
k_i:=k_i^\eta.
\)
Then, with respect to the frame
\(\mathcal F_{t,w_i^{-,\eta}}\) at \(m_{i+1}\) and the frame
\(\mathcal F_{t,w_i^{+,\eta}}\) at \(m_i\), as \(t\to+\infty\),
\[
\operatorname{Hol}_t(\widetilde\delta_i)
\operatorname{Hol}_t(\widetilde\beta_i)
\operatorname{Hol}_t(\widetilde\alpha_i)
=
e^{t^{1/4}(\lambda_i+\mu_i)}
S\mathcal C_i^\eta S^{-1}
+
o\left(
e^{t^{1/4}(\lambda_i+\mu_i)}
\right),
\]
where
\[
\mathcal C_i^\eta
=
\sum_{\substack{
r\in\mathcal R_i^\eta\\
s\in\mathcal S_i^\eta
}}
\bigl(T^{k_i}\bigr)_{r,s}E_{r,s}.
\]
Let \(N_i=d_i+4\). With the perturbation convention fixed above, the corresponding admissible ranges of \(k\) are
\[
\renewcommand{\arraystretch}{1.5}
\begin{array}{c|c|c|c|c}
x_i^+\backslash x_i^-
&
-\dfrac{\pi}{4}
&
-\dfrac{\pi}{2}
&
J_0
&
J_1
\\[1.2ex]
\hline
\rule{0pt}{4ex}
-\dfrac{\pi}{4}
&
2\leq k\leq N_i-2
&
2\leq k\leq N_i-3
&
2\leq k\leq N_i-3
&
3\leq k\leq N_i-2
\\[1.2ex]
\hline
\rule{0pt}{4ex}
-\dfrac{\pi}{2}
&
3\leq k\leq N_i-2
&
2\leq k\leq N_i-2
&
3\leq k\leq N_i-2
&
3\leq k\leq N_i-2
\\[1.2ex]
\hline
\rule{0pt}{3.5ex}
J_0
&
3\leq k\leq N_i-2
&
2\leq k\leq N_i-3
&
2\leq k\leq N_i-2
&
2\leq k\leq N_i-3
\\[1ex]
\hline
\rule{0pt}{3.5ex}
J_1
&
2\leq k\leq N_i-3
&
2\leq k\leq N_i-3
&
3\leq k\leq N_i-2
&
2\leq k\leq N_i-2
\end{array}
\]
\end{proposition}
\begin{proof}
We first determine the positions of the maximal diagonal entries and
identify the corresponding entries of \(T_i^\eta\) that contribute to
the dominant term. Recall that we use the perturbation convention
described above. Consequently, with respect to the sectorial natural coordinate
associated with the perturbed endpoint, the original radial ray has
direction either
\[
-\frac{\pi}{4}
\qquad\text{or}\qquad
-\frac{\pi}{2}.
\]

We now list all the possible entries occurring in the dominant term.
Set
\[
\vartheta_i^-
:=
\arg_{w_i^{-,\eta}}( x_i^-),
\qquad
\vartheta_i^+
:=
\arg_{w_i^{+,\eta}}( x_i^+).
\]
Whenever the corresponding endpoint lies on a Stokes direction, our
perturbation convention gives
\[
\vartheta_i^\pm
\in
\left\{
-\frac{\pi}{4},-\frac{\pi}{2}
\right\}.
\]
For the incoming segment,
\[
\mathcal S_i^\eta
=
\begin{cases}
\{1,4\},
&
\vartheta_i^-=-\tfrac{\pi}{4},
\\
\{4\},
&
\vartheta_i^-=-\tfrac{\pi}{2},
\end{cases}
\]
while for the outgoing segment,
\[
\mathcal R_i^\eta
=
\begin{cases}
\{2,3\},
&
\vartheta_i^+=-\tfrac{\pi}{4},
\\
\{2\},
&
\vartheta_i^+=-\tfrac{\pi}{2}.
\end{cases}
\]

Before considering the different cases, recall that the original path
is locally geodesic at \(p_i\). Hence both angles determined by the
original incoming and outgoing rays are at least \(\pi\). Recall that
 the total cone angle at \(p_i\) is
\(
\frac{N_i\pi}{2}.
\)

Let \(\kappa_i\) denote the sectorial index difference determined by
the original incoming and outgoing rays, before applying the
perturbation. If
\[
\vartheta_i^-
=
\arg_{w_i^-}(p_ix_i^-),
\qquad
\vartheta_i^+
=
\arg_{w_i^+}(p_ix_i^+),
\]
then the counterclockwise angle from the original incoming ray to the
original outgoing ray in the developed natural coordinate is
\[
\Theta_i
=
\vartheta_i^+
+
\frac{\kappa_i\pi}{2}
-
\vartheta_i^-.
\]
Its complementary angle is
\[
\frac{N_i\pi}{2}-\Theta_i.
\]
The local geodesic condition therefore gives
\[
\Theta_i\geq\pi,
\qquad
\frac{N_i\pi}{2}-\Theta_i\geq\pi,
\]
or equivalently,
\[
\pi
\leq
\vartheta_i^+
+
\frac{\kappa_i\pi}{2}
-
\vartheta_i^-
\leq
\frac{N_i\pi}{2}-\pi.
\]

The integer appearing in the perturbed holonomy block is
\(k_i=k_i^\eta\), rather than \(\kappa_i\). Depending on the Stokes
directions containing the two endpoints, the perturbation may change
the sectorial index difference. Thus the relation between
\(\kappa_i\) and \(k_i^\eta\) is case-dependent. 

\smallskip

\noindent
\emph{Case 1: both \(x_i^-\) and \(x_i^+\) lie on Stokes
directions.}
In this case,
\[
a_i^\eta=b_i^\eta=0,
\qquad
T_i^\eta=T^{k_i^\eta}.
\]
The entries occurring in the dominant term are
\[
\renewcommand{\arraystretch}{1.5}
\begin{array}{c|c|c}
\vartheta_i^+\backslash\vartheta_i^-
&
-\dfrac{\pi}{4}
&
-\dfrac{\pi}{2}
\\[1ex]
\hline
\rule{0pt}{4ex}
-\dfrac{\pi}{4}
&
\begin{matrix}
\bigl(T^{k_i^\eta}\bigr)_{2,1},
&
\bigl(T^{k_i^\eta}\bigr)_{2,4},
\\
\bigl(T^{k_i^\eta}\bigr)_{3,1},
&
\bigl(T^{k_i^\eta}\bigr)_{3,4}
\end{matrix}
&
\begin{matrix}
\bigl(T^{k_i^\eta}\bigr)_{2,4},
\\
\bigl(T^{k_i^\eta}\bigr)_{3,4}
\end{matrix}
\\[2ex]
\hline
\rule{0pt}{4ex}
-\dfrac{\pi}{2}
&
\begin{matrix}
\bigl(T^{k_i^\eta}\bigr)_{2,1},
&
\bigl(T^{k_i^\eta}\bigr)_{2,4}
\end{matrix}
&
\bigl(T^{k_i^\eta}\bigr)_{2,4}.
\end{array}
\]
The corresponding admissible ranges are
\[
\begin{array}{c|cc}
\vartheta_i^+\backslash\vartheta_i^-
&
-\dfrac{\pi}{4}
&
-\dfrac{\pi}{2}
\\[1ex]
\hline
\rule{0pt}{4ex}
-\dfrac{\pi}{4}
&
2\leq k_i^\eta\leq N_i-2
&
2\leq k_i^\eta\leq N_i-3
\\[1ex]
\hline
\rule{0pt}{4ex}
-\dfrac{\pi}{2}
&
3\leq k_i^\eta\leq N_i-2
&
2\leq k_i^\eta\leq N_i-2.
\end{array}
\]

\smallskip

\noindent
\emph{Case 2: only \(x_i^-\) lies on a Stokes direction.}
Here \(a_i^\eta=0\). If the terminal endpoint lies in a stable sector
of type \(J_0\), then
\[
b_i^\eta=0,
\qquad
T_i^\eta=T^{k_i^\eta},
\]
and the relevant entries are
\[
\begin{cases}
(T^{k_i^\eta})_{2,1},
\ (T^{k_i^\eta})_{2,4},
&
\vartheta_i^-=-\tfrac{\pi}{4},
\qquad 3\leq k\leq N_i-2,
\\
(T^{k_i^\eta})_{2,4},
&
\vartheta_i^-=-\tfrac{\pi}{2},
\qquad 2\leq k\leq N_i-3.
\end{cases}
\]
If the terminal endpoint lies in a stable sector of type \(J_1\),
then
\[
b_i^\eta=1,
\qquad
T_i^\eta=U_1^{-1}T^{k_i^\eta},
\]
and the relevant entries are
\[
\begin{cases}
(U_1^{-1}T^{k_i^\eta})_{3,1},
\ (U_1^{-1}T^{k_i^\eta})_{3,4},
&
\vartheta_i^-=-\tfrac{\pi}{4},
\qquad 2\leq k\leq N_i-3,
\\
(U_1^{-1}T^{k_i^\eta})_{3,4},
&
\vartheta_i^-=-\tfrac{\pi}{2},
\qquad 2\leq k\leq N_i-3.
\end{cases}
\]

\smallskip

\noindent
\emph{Case 3: only \(x_i^+\) lies on a Stokes direction.}
Here \(b_i^\eta=0\). If the initial endpoint lies in a stable sector
of type \(J_0\), then
\[
a_i^\eta=0,
\qquad
T_i^\eta=T^{k_i^\eta},
\]
and the relevant entries are
\[
\begin{cases}
(T^{k_i^\eta})_{2,4},
\ (T^{k_i^\eta})_{3,4},
&
\vartheta_i^+=-\tfrac{\pi}{4},
\qquad 2\leq k\leq N_i-3,
\\
(T^{k_i^\eta})_{2,4},
&
\vartheta_i^+=-\tfrac{\pi}{2},
\qquad 3\leq k\leq N_i-2.
\end{cases}
\]
If the initial endpoint lies in a stable sector of type \(J_1\), then
\[
a_i^\eta=1,
\qquad
T_i^\eta=T^{k_i^\eta}U_1,
\]
and the relevant entries are
\[
\begin{cases}
(T^{k_i^\eta}U_1)_{2,1},
\ (T^{k_i^\eta}U_1)_{3,1},
&
\vartheta_i^+=-\tfrac{\pi}{4},
\qquad 3\leq k\leq N_i-2,
\\
(T^{k_i^\eta}U_1)_{2,1},
&
\vartheta_i^+=-\tfrac{\pi}{2},
\qquad 3\leq k\leq N_i-2.
\end{cases}
\]
These exhaust all possibilities in the Stokes case. Recalling the
explicit form of \(U_1\) from
Proposition~\ref{prop:explicit-U}, we have
\begin{align*}
  (U_1^{-1}T^k)_{3,s}&=(T^k)_{3,s},
  \qquad s\in\{1,4\}, \\
  (T^kU_1)_{r,1}&=(T^k)_{r,1},
  \qquad r\in\{2,3\}.
\end{align*}
Combining this with
Proposition~\ref{prop:asymptotics-single-holonomy-block-Stokes-directions}
proves the result.
\end{proof}

\subsection{Powers of the transition matrix and non-vanishing}
We consider the transition matrix \(T\) in Proposition~\ref{prop:large-circular-holonomy} for \(q=z^d\odif z^4\). Let \(N=d+4\) and \(\theta=\frac{\pi}{d+4}.\) By Proposition~\ref{prop:explicit-U},
Proposition~\ref{prop:large-circular-holonomy}, and
\eqref{eq:definition-for-frame-change-Q}, we have
\begin{equation}
  \begin{aligned}
  T
  &=
  Q^{-1}U_2^{-1}U_1^{-1}
  \\ \notag
  &=
  \begin{pmatrix}
  0 & 1 & 0 & 0\\
  0 & 0 & 1 & 0\\
  0 & 0 & 0 & 1\\
  -1 & 0 & 0 & 0
  \end{pmatrix}
  \begin{pmatrix}
  1 & 0 & s_2 & 0\\
  0 & 1 & 0 & 0\\
  0 & 0 & 1 & 0\\
  0 & 0 & 0 & 1
  \end{pmatrix}
  \begin{pmatrix}
  1 & -s_1 & 0 & 0\\
  0 & 1 & 0 & 0\\
  0 & 0 & 1 & 0\\
  0 & 0 & s_1 & 1
  \end{pmatrix}
  \\ \notag
  &=
  \begin{pmatrix}
  0 & 1 & 0 & 0\\
  0 & 0 & 1 & 0\\
  0 & 0 & s_1 & 1\\
  -1 & s_1 & -s_2 & 0
  \end{pmatrix},
  \label{eq:explicit-value-for-T}
  \end{aligned}
\end{equation}
where
\[
s_1
=
\frac{\sin(4\theta)}{\sin\theta},
\qquad
s_2
=
\frac{\sin(4\theta)\sin(3\theta)}
{\sin\theta\sin(2\theta)}.
\]

By the results of the preceding subsection, the only entries of
\(T^k\) that occur in the dominant coefficient matrices are
\[
(T^k)_{2,1},
\qquad
(T^k)_{2,4},
\qquad
(T^k)_{3,1},
\qquad
(T^k)_{3,4}.
\] 
We show that all four entries can be expressed in terms of the single sequence.
\begin{lemma}\label{lem:relevant-entries-in-terms-of-bk}
  Let \(b_k:=(T^k)_{2,4}.\) Then.
  \[
(T^k)_{2,1}=-b_{k-1},
\qquad
(T^k)_{2,4}=b_k,
\qquad
(T^k)_{3,1}=-b_k,
\qquad
(T^k)_{3,4}=b_{k+1}.
  \]
\end{lemma}
\begin{proof}
  Let \(e_1,\ldots,e_4\) denote the standard basis of
\(\mathbb R^4\). Since
\[
Te_1=-e_4,
\]
for every \(k\geq1\) we have
\[
(T^k)_{r,1}
=
e_r^{\mathsf T}T^{k-1}(Te_1)
=
-(T^{k-1})_{r,4}.
\]
Moreover, since the second row of \(T\) is \(e_3^{\mathsf T}\),
\[
e_2^{\mathsf T}T=e_3^{\mathsf T},
\]
and hence
\[
(T^k)_{3,4}
=
e_3^{\mathsf T}T^ke_4
=
e_2^{\mathsf T}T^{k+1}e_4
=
b_{k+1}.
\]
It follows that
\[
(T^k)_{2,1}=-b_{k-1},
\qquad
(T^k)_{2,4}=b_k,
\qquad
(T^k)_{3,1}=-b_k,
\qquad
(T^k)_{3,4}=b_{k+1}.
\]
\end{proof}
Thus it remains only to determine the sequence \((b_k)\).
\begin{proposition}
\label{prop:nonvanishing-relevant-entries}
For every \(k\geq0\),
\[
b_k
=
\frac{
\sin((k-1)\theta)\sin(k\theta)\sin((k+1)\theta)
}{
\sin\theta\sin(2\theta)\sin(3\theta)
}.
\]
In particular,
\[
b_k>0
\qquad
\text{for }\qquad
2\leq k\leq d+2,
\]
while
\[
b_1=b_{d+3}=0.
\]
Consequently, the signs and possible vanishing of the four relevant
entries are determined by
Lemma~\ref{lem:relevant-entries-in-terms-of-bk}.
\end{proposition}
\begin{proof}
  To prove this proposition, we first calculate the characteristic polynomial of \(T\):
\[
f(\lambda)
=
\det(\lambda I-T)
=
\lambda^4
-s_1\lambda^3
+s_2\lambda^2
-s_1\lambda
+1.
\]
Note that 
\[
\sin(4\theta)=2\sin(2\theta)\cos(2\theta)=4\sin(\theta)\cos(\theta)\cos(2\theta).
\]
Hence
\[
s_1=\frac{\sin(4\theta)}{\sin\theta}=4\cos(\theta)\cos(2\theta).
\]
Then using
\[
2\cos(\theta)\cos(2\theta)=\cos(\theta)+\cos(3\theta),
\]
we obtain
\[
s_1=2(\cos(\theta)+\cos(3\theta)).
\]
For \(s_2\), we have
\[
s_2=\frac{\sin(4\theta)\sin(3\theta)}
{\sin\theta\sin(2\theta)}=\frac{2\cos(2\theta)\sin(3\theta)}
{\sin\theta}.
\]
Using
\[
\frac{\sin(3\theta)}{\sin(\theta)}=\cos(2\theta)+2\cos^2(\theta)=2\cos(2\theta)+1,
\]
we obtain
\begin{align*}
    s_2&=2\cos(2\theta)(2\cos(2\theta)+1)=4\cos^2(2\theta) +2\cos(2\theta) \\
    &=2(\cos(4\theta)+2\cos(2\theta))+2=4\cos(\theta)\cos(3\theta)+2.
\end{align*}
Hence we can factor the characteristic polynomial \(f(\lambda)\) of \(T\) as 
\begin{align*}
  f(\lambda)&=(\lambda^2-2\cos(\theta)+1)(\lambda^2-2\cos(3\theta)+1) \\
  &= (\lambda-e^{\sqrt{-1}\theta})(\lambda-e^{-\sqrt{-1}\theta})(\lambda-e^{3\sqrt{-1}\theta})(\lambda-e^{-3\sqrt{-1}\theta}).
\end{align*}

Therefore, the roots of \(f(\lambda)\) are
\(
e^{\sqrt{-1}\theta},
e^{-\sqrt{-1}\theta},
e^{3\sqrt{-1}\theta},
e^{-3\sqrt{-1}\theta}.
\) 
Hence \(T\) is similar to the diagonal matrix \(\operatorname{diag}\left(
e^{\sqrt{-1}\theta},
e^{-\sqrt{-1}\theta},
e^{3\sqrt{-1}\theta},
e^{-3\sqrt{-1}\theta}
\right)\). Moreover, \(T^k\) is similar to the diagonal matrix 
\[
\operatorname{diag}
\left(
e^{\sqrt{-1}k\theta},
e^{-\sqrt{-1}k\theta},
e^{3\sqrt{-1}k\theta},
e^{-3\sqrt{-1}k\theta}
\right).
\]
Hence the explicit formula for \(b_k\) has the following form:
\[
b_k
=
A\cos(k\theta)
+B\sin(k\theta)
+C\cos(3k\theta)
+D\sin(3k\theta).
\]
This formula also holds for \(T^{-1}\). Therefore we can consider \(b_k\) for negative \(k\). By \eqref{eq:explicit-value-for-T}, \(b_1=b_0=b_{-1}=0\). Hence
\begin{align*}
  b_1&=A\cos(\theta)
  +B\sin(\theta)
  +C\cos(3\theta)
  +D\sin(3\theta)=0, \\
  b_0&= A + C=0, \\
  b_{-1}&=A\cos(\theta)
  -B\sin(\theta)
  +C\cos(3\theta)
  -D\sin(3\theta)=0.
\end{align*}
Therefore, 
\[
0=A\cos(\theta)+C\cos(3\theta)=A(\cos(\theta)-\cos(3\theta)).
\]
Recall that
\(
0<\theta=\frac{\pi}{d+4}<\frac{\pi}{3}.
\)
It means that \(\cos(\theta)-\cos(3\theta)\neq0\), hence 
\[
A=C=0.
\]
Moreover, 
\[
B\sin(\theta)
  +D\sin(3\theta) =0.
\] 
Hence 
\[
B=-D\frac{\sin(3\theta)}{\sin(\theta)}.
\]
Combining with the above formula for \(b_k\), we have
\begin{equation}\label{eq:explicit-computaion-for-bk}
  b_k=-D\left(\frac{\sin(3\theta)}{\sin(\theta)}\sin(k\theta)-\sin(3k\theta)\right).
\end{equation}
Note that
\begin{align*}
  \frac{\sin(3\theta)}{\sin(\theta)} 
  &=\frac{\cos(2\theta)\sin(\theta)+\sin(2\theta)\cos(\theta)}{\sin(\theta)} \\
  &=\cos(2\theta)+\frac{2\sin(\theta)\cos(\theta)\cos(\theta)}{\sin(\theta)} \\
  &=\cos(2\theta)+2\cos^2(\theta) \\
  &=1+2\cos(2\theta).
\end{align*}
Therefore, for the equation \eqref{eq:explicit-computaion-for-bk} inside the parentheses on the right-hand side, we have
\begin{align*}
 \frac{\sin(3\theta)}{\sin(\theta)}\sin(k\theta)-\sin(3k\theta)
 &= \left(1+2\cos(2\theta)\right)\sin(k\theta)-\sin(3k\theta) \\
 &= \left(1+2\cos(2\theta)\right)\sin(k\theta)-(\sin(k\theta)\cos(2k\theta)+\sin(2k\theta)\cos(k\theta)) \\
 &=\sin(k\theta)\left(1+2\cos(2\theta)-\cos(2k\theta)\right)-\sin(2k\theta)\cos(k\theta) \\
 &=\sin(k\theta)\left(1+\cos(2\theta)+2\sin((k+1)\theta)\sin((k-1)\theta)\right)-2\sin(k\theta)\cos^2(k\theta) \\
 &=2\sin((k-1)\theta)\sin(k\theta)\sin((k+1)\theta)+\sin(k\theta)\left(1+\cos(2\theta)-2\cos^2(k\theta)\right) \\
 &=2\sin((k-1)\theta)\sin(k\theta)\sin((k+1)\theta)+\sin(k\theta)\left(\cos(2\theta)-\cos(2k\theta)\right)  \\
 &=4\sin((k-1)\theta)\sin(k\theta)\sin((k+1)\theta).
\end{align*}
Hence
\[
b_k=G\sin((k-1)\theta)\sin(k\theta)\sin((k+1)\theta)
\]
for some constant \(G=-4D\).
By \eqref{eq:explicit-value-for-T}, we have \(b_2=1\). Therefore 
\begin{align*}
G=\frac{1}{\sin(\theta)\sin(2\theta)\sin(3\theta)},
\end{align*}
namely,
\[
b_k=\frac{\sin((k-1)\theta)\sin(k\theta)\sin((k+1)\theta)}{\sin(\theta)\sin(2\theta)\sin(3\theta)}.
\]
This completes the proof.
\end{proof}
\subsection{Composition of the holonomy blocks}
We now compose the individual holonomy blocks obtained in Section~\ref{subsec:asymptotics-single-holonomy-block}. There are two points to address before composing the individual
holonomy blocks. First, at each zero we must choose sectorial natural
coordinates in which the incoming and outgoing radial directions
satisfy the angular assumptions of
Propositions~\ref{prop:dominant-term-single-holonomy-block-no-Stokes-directions}
and\ref{prop:dominant-term-single-holonomy-block-Stokes-directions}. Second, even after such coordinates have been chosen compatibly, the
terminal frame of one block and the initial frame of the next block, although based at the same point, are not identical. They differ by a fixed change-of-frame matrix, which must be inserted when the corresponding holonomy matrices are multiplied.

\paragraph{Compatible choice of sectorial natural coordinates.}
We first construct the sectorial natural coordinates so that the
hypotheses of the preceding propositions are satisfied at every zero. More precisely, at each zero \(p_i\) of order \(d_i\), we seek a
normalized local coordinate \(z_i\), centered at \(p_i\), such that
\[
q=z_i^{d_i}\odif z_i^4,
\]
together with incoming and outgoing sectorial natural coordinates
\(w_i^-\) and \(w_i^+\) arising from the construction of
Section~\ref{subsec:explicit-computation-for-z^d}, with the following
properties:
\[
-\frac{\pi}{2}
\leq
\arg_{w_i^-}(p_i x_i^-)
<
0,
\qquad
-\frac{\pi}{2}
\leq
\arg_{w_i^+}(p_i x_i^+)
<
0.
\]
If the corresponding radial ray is not in a Stokes direction, it lies in one of the stable sectors \(J_0\) or \(J_1\). If it is in a Stokes direction, its argument in the chosen sectorial coordinate is either
\[
-\frac{\pi}{2}
\qquad\text{or}\qquad
-\frac{\pi}{4}.
\]
Once such coordinates have been chosen, the Stokes perturbation does
not require a further change of sectorial natural coordinate: the
directions \(-\pi/4\) and \(-\pi/2\) are perturbed into the interior
of \(J_0\) while keeping the same coordinate. We now construct these
coordinates inductively along the saddle connections.

At the first zero \(p_\ell\), choose a normalized local coordinate
\(z_\ell\) such that
\(
q=z_\ell^{d_\ell}\odif z_\ell^4,
\)
and construct the corresponding sectorial natural coordinates in Section~\ref{subsec:explicit-computation-for-z^d}. We then choose the incoming and outgoing sectorial natural
coordinates \(w_\ell^-\) and \(w_\ell^+\), respectively, so that
\[
-\frac{\pi}{2}
\leq
\arg_{w_\ell^-}(p_\ell x_\ell^-)
<
0,
\qquad
-\frac{\pi}{2}
\leq
\arg_{w_\ell^+}(p_\ell x_\ell^+)
<
0.
\]
Suppose inductively that the normalized coordinates and the
sectorial natural coordinates have been chosen at
\[
p_\ell,p_{\ell-1},\ldots,p_i.
\]
We now construct the incoming coordinate at \(p_{i-1}\). Recall that
\[
c_i=\delta_i*\alpha_{i-1}
\]
is the saddle connection oriented from \(p_i\) to \(p_{i-1}\) (see Figure~\ref{fig:global-modification}).
Extend the chosen sectorial natural coordinate \(w_i^+\) along a
zero-free neighborhood of the interior of \(c_i\) to
\(p_{i-1}\).
 Set
\[
\Pi_i:=\widehat w_i(p_{i-1}).
\]
We define the incoming natural coordinate at \(p_{i-1}\) by
\[
w_{i-1}^-:=\Pi_i-w_i^+.
\]
Then
\[
w_{i-1}^-(p_{i-1})=0,
\qquad
\odif w_{i-1}^-=-\odif w_i^+.
\]
Moreover, 
\[
\arg_{w_{i-1}^-}(p_{i-1}x_{i-1}^-)
=
\arg \Pi_i
=
\arg_{w_i^+}(p_i x_i^+).
\]
Therefore,
\[
-\frac{\pi}{2}
\leq
\arg_{w_{i-1}^-}(p_{i-1}x_{i-1}^-)
<0.
\]
After possibly replacing \(z_{i-1}\) by \(\widehat z_{i-1}:=\xi_i z_{i-1}\) for some \(\xi_i\) such that \(
\xi_i^{N_{i-1}}=1,\) with \(N_{i-1}=d_{i-1}+4\), we may assume that \(w_{i-1}^-\) is
precisely the incoming sectorial natural coordinate constructed in Section~\ref{subsec:explicit-computation-for-z^d}.
We then choose the outgoing coordinate \(w_{i-1}^+\) by the same
sectorial construction, so that
\[
-\frac{\pi}{2}
\leq
\arg_{w_{i-1}^+}(p_{i-1}x_{i-1}^+)
<0.
\]
Repeating this procedure gives compatible coordinates at all
successive zeros. 
\paragraph{Change of frame at the common points.}
We now determine the change of frame between two consecutive holonomy
blocks at their common point.

First consider \(i=2,\ldots,\ell\). By construction,
\[
w_{i-1}^-=\Pi_i-w_i^+,
\]
and hence
\[
\odif {w_{i-1}^-}=-\odif{w_i^+}.
\]
It remains to compare the coordinates \(w_1^+\) and \(w_\ell^-\)
along the closing saddle connection \(c_1\). By our choice of
sectorial natural coordinates, the outgoing ray at \(p_1\) and the
incoming ray at \(p_\ell\) have the same argument in their respective
coordinates:
\[
\arg_{w_1^+}(p_1x_1^+)
=
\arg_{w_\ell^-}(p_\ell x_\ell^-).
\]
These two rays lie on the same straight saddle connection \(c_1\), but
have opposite orientations. Hence the direction of the oriented
segment \(c_1\) in the two natural coordinates differs by \(\pi\).
Therefore, at the common midpoint \(m_1\),
\[
\odif w_\ell^-=-\odif w_1^+.
\]

Thus, with the indices understood cyclically, at every common
midpoint \(m_i\) we have
\[
\odif w_{i-1}^-=-\odif w_i^+
=
(-\sqrt{-1})^2\odif w_i^+.
\]
Hence the passage from \(w_i^+\) to \(w_{i-1}^-\) corresponds to two
successive changes of sectorial natural coordinate. With the lift to the half-canonical bundle fixed in
\eqref{eq:half-differential-transition}, the induced natural frames satisfy
\begin{equation}
\label{eq:frame-change-at-midpoint}
\mathcal F_{t,w_{i-1}^-}
=
\mathcal F_{t,w_i^+}R^2,
\qquad
i=1,\ldots,\ell,
\end{equation}
where \(i-1\) is understood modulo \(\ell\), and \(R\) is defined in
\eqref{eq:definition-for-frame-change-R}.

\paragraph{Composition of the holonomy blocks.}
For each \(i=1,\ldots,\ell\), set
\[
\mathcal B_{i,t}
:=
\operatorname{Hol}_t(\widetilde\delta_i)
\operatorname{Hol}_t(\widetilde\beta_i)
\operatorname{Hol}_t(\widetilde\alpha_i).
\]
Recall that \(\mathcal B_{i,t}\) is represented with respect to the
frame \(\mathcal F_{t,w_i^-}\) at \(m_{i+1}\) and the frame
\(\mathcal F_{t,w_i^+}\) at \(m_i\). At the common midpoint \(m_i\),
the frames associated with two consecutive blocks satisfy
\[
\mathcal F_{t,w_{i-1}^-}
=
\mathcal F_{t,w_i^+}R^2.
\]
Hence the corresponding change of coordinates is \(R^{-2}\), and the
\(\ell\) holonomy blocks are composed as
\begin{equation}
\label{eq:composition-all-holonomy-blocks}
\operatorname{Hol}_t(\widetilde c_\gamma)
=
\mathcal B_{1,t}R^{-2}\mathcal B_{2,t}R^{-2}
\cdots R^{-2}\mathcal B_{\ell,t}R^{-2}.
\end{equation}
Here the holonomy is represented with respect to the frame
\(\mathcal F_{t,w_{1}^+}\) at \(m_1\), both at the initial and at the
terminal point.

By Proposition~\ref{prop:dominant-term-single-holonomy-block-no-Stokes-directions} and Proposition~\ref{prop:dominant-term-single-holonomy-block-Stokes-directions},
\[
\mathcal B_{i,t}
=
e^{t^{1/4}(\lambda_i+\mu_i)}
S\mathcal C_iS^{-1}
+
o\left(
e^{t^{1/4}(\lambda_i+\mu_i)}
\right).
\]
By \eqref{eq:definition-for-frame-change-Q},
\begin{equation}\label{eq:value-for-Q(-2)}
  R^{-2}=SQ^{-2}S^{-1}=S
  \begin{pmatrix}
  0 & 0 & 1 & 0\\
  0 & 0 & 0 & 1\\
  -1 & 0 & 0 & 0\\
  0 & -1 & 0 & 0
  \end{pmatrix}
  S^{-1}.
\end{equation}
Combining with Proposition~\ref{prop:asymptotics-single-holonomy-block-no-Stokes-directions} and Proposition~\ref{prop:asymptotics-single-holonomy-block-Stokes-directions}, we have
\begin{equation}\label{eq:asymptotics-composition-all-holonomy-blocks}
  \operatorname{Hol}_t(\widetilde c_\gamma)
  =
  e^{t^{1/4}\sum_{i=1}^{\ell}(\lambda_i+\mu_i)}
  S
  \left(
  \mathcal C_1Q^{-2}\mathcal C_2Q^{-2}
  \cdots Q^{-2}\mathcal C_\ell Q^{-2}
  \right)
  S^{-1}
  +o\left(
  e^{t^{1/4}\sum_{i=1}^{\ell}(\lambda_i+\mu_i)}
  \right).
\end{equation}

The key algebraic input for deriving the asymptotic formulas is the
following non-nilpotence statement.
\begin{proposition}\label{prop:non-nilpotence-global-dominant-coefficient}
  The matrix \(\mathcal P_\ell=\mathcal C_1Q^{-2}\mathcal C_2Q^{-2}
\cdots Q^{-2}\mathcal C_{\ell}Q^{-2}\) is not nilpotent. 
\end{proposition}
\begin{proof}
 By Proposition~\ref{prop:dominant-term-single-holonomy-block-Stokes-directions}, Lemma~\ref{lem:relevant-entries-in-terms-of-bk} and Proposition~\ref{prop:nonvanishing-relevant-entries}, for each \(i=1,\ldots,\ell\), the matrix \(\mathcal C_i\) has the
form
\begin{equation}
\label{eq:general-form-Ci}
\mathcal C_i
=
-a_iE_{2,1}
+b_iE_{2,4}
-c_iE_{3,1}
+d_iE_{3,4}
=
\begin{pmatrix}
0&0&0&0\\
-a_i&0&0&b_i\\
-c_i&0&0&d_i\\
0&0&0&0
\end{pmatrix},
\end{equation}
where
\(
a_i,b_i,c_i,d_i\geq 0,
\)
and not all four coefficients vanish. 
Then 
\[
\mathcal C_iQ^{-2}=-\begin{pmatrix}
0&0&0&0\\
0&b_i&a_i&0\\
0&d_i&c_i&0\\
0&0&0&0
\end{pmatrix}.
\]
Setting
\[
A_i
:=
\begin{pmatrix}
b_i&a_i\\
d_i&c_i
\end{pmatrix},
\]
we obtain
\[
\mathcal P_\ell
=(-1)^{\ell}
\begin{pmatrix}
0&0&0&0\\
0&p_{11}&p_{12}&0\\
0&p_{21}&p_{22}&0\\
0&0&0&0
\end{pmatrix},
\qquad
\begin{pmatrix}
p_{11}&p_{12}\\
p_{21}&p_{22}
\end{pmatrix}
=
A_1A_2\cdots A_\ell.
\]
It remains to show that
\[
A:=A_1A_2\cdots A_\ell
\]
is not nilpotent. 

The possible zero patterns of \(A_i\) are determined by the incoming
and outgoing directions at \(p_i\). More precisely, combining
Proposition~\ref{prop:dominant-term-single-holonomy-block-Stokes-directions},
Lemma~\ref{lem:relevant-entries-in-terms-of-bk},
and
Proposition~\ref{prop:nonvanishing-relevant-entries},
we obtain the following description:
\begingroup
\renewcommand{\arraystretch}{1.25}
\setlength{\extrarowheight}{1.5pt}
\[
\makebox[\textwidth][c]{%
\(
\begin{array}{c@{\hspace{1em}}c}
\begin{array}{c|c}
\text{incoming direction at }p_i
&
\text{nonzero columns of }A_i
\\
\hline
J_0\text{ or }-\dfrac{\pi}{2}
&
\{1\}
\\
J_1
&
\{2\}
\\
-\dfrac{\pi}{4}
&
\{1,2\}
\end{array}
&
\begin{array}{c|c}
\text{outgoing direction at }p_i
&
\text{nonzero rows of }A_i
\\
\hline
J_0\text{ or }-\dfrac{\pi}{2}
&
\{1\}
\\
J_1
&
\{2\}
\\
-\dfrac{\pi}{4}
&
\{1,2\}
\end{array}
\end{array}
\)
}
\]
\endgroup

Recall that the compatible choice of sectorial natural coordinates
ensures that the incoming direction at \(p_i\) and the outgoing
direction at \(p_{i+1}\), which correspond to the same saddle
connection, have the same argument in their respective natural
coordinates:
\[
\arg_{w_i^-}(p_ix_i^-)
=
\arg_{w_{i+1}^+}(p_{i+1}x_{i+1}^+).
\]
In the Stokes case, the same equality holds for the perturbed
directions. Consequently, with the indices understood cyclically, for
every \(q\in\{1,2\}\),
\begin{equation}
\label{eq:matching-rows-columns-Ai}
\text{the \(q\)-th column of \(A_i\) is nonzero}
\qquad\Longleftrightarrow\qquad
\text{the \(q\)-th row of \(A_{i+1}\) is nonzero}.
\end{equation}

Choose \(r_0,r_1\in\{1,2\}\) such that
\(
(A_1)_{r_0,r_1}>0.
\)
By \eqref{eq:matching-rows-columns-Ai}, the \(r_1\)-th row of
\(A_2\) is nonzero. Hence there exists
\(r_2\in\{1,2\}\) such that
\(
(A_2)_{r_1,r_2}>0.
\)
Proceeding inductively and then repeating the argument cyclically, we
obtain an infinite sequence
\(
r_0,r_1,r_2,\ldots\in\{1,2\}
\)
such that, for every \(m\geq0\) and \(i=1,\ldots,\ell\),
\[
(A_i)_{r_{m\ell+i-1},r_{m\ell+i}}>0.
\]
The sequence\(
r_0,r_\ell,r_{2\ell},\ldots,
\)
takes values in the two-element set \(\{1,2\}\). Therefore, there
exist integers \(0\leq m<n\) such that
\(
r_{m\ell}=r_{n\ell}.
\)
The product
\(
\prod_{q=m}^{n-1}
\prod_{i=1}^{\ell}
(A_i)_{r_{q\ell+i-1},r_{q\ell+i}}
\)
is strictly positive and occurs as a summand of
\[
\left(A^{n-m}\right)_{r_{m\ell},r_{n\ell}}
=
\left(A^{n-m}\right)_{r_{m\ell},r_{m\ell}}.
\]
Since all entries of \(A^{n-m}\) are nonnegative, it follows that
\[
\left(A^{n-m}\right)_{r_{m\ell},r_{m\ell}}>0.
\]
Thus \(A^{n-m}\) has a nonzero diagonal entry, and hence \(A\) is not
nilpotent.
\end{proof}

\subsection{Asymptotics of the second exterior power}
\label{subsec:second-exterior-power}

In this subsection, we study the asymptotic behavior of the holonomy
in the second exterior power. Since the overall argument is parallel
to the one developed above, we retain all the notation and
conventions introduced previously without further explanation.

We first fix our notation and conventions for the second exterior
power. Let \(V=\mathbb R^4\). For \(A\in\operatorname{End}(V)\), we write
\[
A^{[2]}:=\bigwedge\nolimits^2A.
\]
Thus
\[
A^{[2]}(u\wedge v)=Au\wedge Av,
\qquad
(AB)^{[2]}=A^{[2]}B^{[2]}.
\]

We use the ordered basis
\[
e_{12},\ e_{13},\ e_{14},\ e_{23},\ e_{24},\ e_{34},
\qquad
e_{rs}:=e_r\wedge e_s,
\]
of \(\bigwedge^2V\). The natural frames used above induce the
corresponding exterior-power frames.

If
\[
Ae_r=\sum_{i=1}^4 A_{i,r}e_i,
\]
then, for \(1\leq r<s\leq4\),
\[
A^{[2]}e_{rs}
=
\sum_{1\leq i<j\leq4}
\left(
A_{i,r}A_{j,s}-A_{i,s}A_{j,r}
\right)e_{ij}.
\]
Equivalently, if \(I=\{i,j\}\) and \(J=\{r,s\}\), with
\(i<j\) and \(r<s\), then
\begin{equation}
\label{eq:entries-second-exterior-power}
\bigl(A^{[2]}\bigr)_{I,J}
=
\det
\begin{pmatrix}
A_{i,r}&A_{i,s}\\
A_{j,r}&A_{j,s}
\end{pmatrix}.
\end{equation}
In particular, for a straight segment \(\gamma\) of \(q\)-length \(L\)
and direction \(\theta\) in the natural coordinate \(w\),
\[
D_w(\gamma)
=
2L\operatorname{diag}
\bigl(
-\cos\theta,
-\sin\theta,
\cos\theta,
\sin\theta
\bigr),
\]
with respect to the induced exterior-power frame,
\begin{align}
\left(
\exp\bigl(t^{1/4}D_w(\gamma)\bigr)
\right)^{[2]}
&=
\operatorname{diag}\Bigl(
e^{-2Lt^{1/4}(\cos\theta+\sin\theta)},
1,
e^{2Lt^{1/4}(-\cos\theta+\sin\theta)},
\notag\\
&\hspace{35mm}
e^{2Lt^{1/4}(\cos\theta-\sin\theta)},
1,
e^{2Lt^{1/4}(\cos\theta+\sin\theta)}
\Bigr).
\label{eq:exponential-diagonal-model-wedge2}
\end{align}

Retain the notation of
Proposition~\ref{prop:holonomy-straight-pieces}, and let
\(
\mathcal I_2
:=
\left\{
I\subset\{1,2,3,4\}
:
|I|=2
\right\}.
\)
For \(I=\{r,s\}\in\mathcal I_2\), set
\[
\tilde\mu_{i,I}^{[2]}
:=
\tilde\mu_{i,r}+\tilde\mu_{i,s},
\qquad
\tilde\lambda_{i,I}^{[2]}
:=
\tilde\lambda_{i,r}+\tilde\lambda_{i,s}.
\]
Define
\[
\tilde\mu_i^{[2]}
:=
\max_{I\in\mathcal I_2}
\tilde\mu_{i,I}^{[2]},
\qquad
\tilde\lambda_i^{[2]}
:=
\max_{I\in\mathcal I_2}
\tilde\lambda_{i,I}^{[2]},
\]
and the corresponding maximal index sets
\[
\widetilde{\mathcal S}_i^{[2]}
:=
\left\{
I\in\mathcal I_2
:
\tilde\mu_{i,I}^{[2]}=\tilde\mu_i^{[2]}
\right\},
\qquad
\widetilde{\mathcal R}_i^{[2]}
:=
\left\{
I\in\mathcal I_2
:
\tilde\lambda_{i,I}^{[2]}=\tilde\lambda_i^{[2]}
\right\}.
\]
The corresponding quantities without tildes,
\[
\mu_{i,I}^{[2]},\qquad
\lambda_{i,I}^{[2]},\qquad
\mu_i^{[2]},\qquad
\lambda_i^{[2]},\qquad
\mathcal S_i^{[2]},\qquad
\mathcal R_i^{[2]},
\]
are defined analogously. Then we have the following propositions.
\begin{proposition}
\label{prop:holonomy-straight-pieces-wedge2}
With respect to the induced exterior-power frames by frames in Proposition~\ref{prop:holonomy-straight-pieces}, as
\(t\to+\infty\),
\[
\bigwedge\nolimits^2
\operatorname{Hol}_t(\widetilde\alpha_i)
=
S^{[2]}
\operatorname{diag}_{I\in\mathcal I_2}
\left(
e^{t^{1/4}\tilde\mu_{i,I}^{[2]}}
\right)
(S^{[2]})^{-1}
+
o\left(
e^{t^{1/4}\tilde\mu_i^{[2]}}
\right),
\]
and
\[
\bigwedge\nolimits^2
\operatorname{Hol}_t(\widetilde\delta_i)
=
S^{[2]}
\operatorname{diag}_{I\in\mathcal I_2}
\left(
e^{t^{1/4}\tilde\lambda_{i,I}^{[2]}}
\right)
(S^{[2]})^{-1}
+
o\left(
e^{t^{1/4}\tilde\lambda_i^{[2]}}
\right).
\]
\end{proposition}
\begin{proof}[Sketch of proof]
The argument is the same as in the proof of
Proposition~\ref{prop:holonomy-straight-pieces}. On each sufficiently
short straight subsegment, the local estimate of Li and Collier
\cite[Theorem~4.4]{collier2017asymptotics} gives the corresponding
diagonal approximation for the parallel transport. Applying
\(\bigwedge^2\) to this short-segment estimate yields the short-segment version
of the formulas above.

We then subdivide \(\widetilde\alpha_i\) and
\(\widetilde\delta_i\) into finitely many such short subsegments and
compose their parallel transports. This gives the stated formulas.
\end{proof}
Using the perturbation convention we have used, the same argument gives
the following formula, which applies simultaneously to the Stokes and
non-Stokes cases.

\begin{proposition}
\label{prop:asymptotics-single-holonomy-block-wedge2}
With respect to the exterior-power frames induced by
\(\mathcal F_{t,w_i^{-,\eta}}\) at \(m_{i+1}\) and
\(\mathcal F_{t,w_i^{+,\eta}}\) at \(m_i\), we have
\begin{align*}
&\bigwedge\nolimits^2
\left(
\operatorname{Hol}_t(\widetilde\delta_i)
\operatorname{Hol}_t(\widetilde\beta_i)
\operatorname{Hol}_t(\widetilde\alpha_i)
\right)
\\
&=
S^{[2]}
\left(
\exp\bigl(
t^{1/4}D_{w_i^{+,\eta}}(\delta_i)
\bigr)
\right)^{[2]}
(T_i^\eta)^{[2]}
\left(
\exp\bigl(
t^{1/4}D_{w_i^{-,\eta}}(\alpha_i)
\bigr)
\right)^{[2]}
(S^{[2]})^{-1}
+
o\left(
e^{t^{1/4}
\left(
\lambda_i^{[2]}+\mu_i^{[2]}
\right)}
\right)
\end{align*}
as \(t\to+\infty\).
\end{proposition}
The six exterior-power exponents in
\eqref{eq:exponential-diagonal-model-wedge2}
determine the corresponding maximal index sets. We write \(rs\) for
the element \(\{r,s\}\in\mathcal I_2\). The dominant term of a single
holonomy block is described as follows.

\begin{proposition}
\label{prop:dominant-term-single-holonomy-block-wedge2}
The incoming maximal index set is given by
\[
\begin{array}{c|cccc}
\text{incoming direction}
&
-\dfrac{\pi}{2}
&
-\dfrac{\pi}{4}
&
J_0
&
J_1
\\[1ex]
\hline
\mathcal S_i^{[2]}
&
\{14,34\}
&
\{14\}
&
\{14\}
&
\{14\}
\end{array}
\]
and the outgoing maximal index set is given by
\[
\begin{array}{c|cccc}
\text{outgoing direction}
&
-\dfrac{\pi}{2}
&
-\dfrac{\pi}{4}
&
J_0
&
J_1
\\[1ex]
\hline
\mathcal R_i^{[2]}
&
\{12,23\}
&
\{23\}
&
\{23\}
&
\{23\}.
\end{array}
\]

With respect to the exterior-power frames induced by
\(\mathcal F_{t,w_i^{-,\eta}}\) at \(m_{i+1}\) and
\(\mathcal F_{t,w_i^{+,\eta}}\) at \(m_i\), as \(t\to+\infty\),
\[
\bigwedge\nolimits^2
\left(
\operatorname{Hol}_t(\widetilde\delta_i)
\operatorname{Hol}_t(\widetilde\beta_i)
\operatorname{Hol}_t(\widetilde\alpha_i)
\right)
=
e^{t^{1/4}(\lambda_i^{[2]}+\mu_i^{[2]})}
S^{[2]}\mathcal C_i^{(2)}(S^{[2]})^{-1}
+
o\left(
e^{t^{1/4}(\lambda_i^{[2]}+\mu_i^{[2]})}
\right),
\]
where
\[
\mathcal C_i^{(2)}
=
\sum_{\substack{
I\in\mathcal R_i^{[2]}\\
J\in\mathcal S_i^{[2]}
}}
\bigl((T^{k_i})^{[2]}\bigr)_{I,J}
E_{I,J}^{[2]}.
\]
Here \(E_{I,J}^{[2]}\) denotes the elementary endomorphism of
\(\bigwedge^2V\) sending \(e_J\) to \(e_I\) and vanishing on the
remaining basis vectors. And the admissible range of \(k_i\) is exactly the one established in
Propositions~\ref{prop:dominant-term-single-holonomy-block-no-Stokes-directions}
and
\ref{prop:dominant-term-single-holonomy-block-Stokes-directions},
according to whether the endpoints are non-Stokes or Stokes.
\end{proposition}
\begin{proof}[Sketch of proof]
The maximal index sets follow directly by comparing the six
exponential rates in
\eqref{eq:exponential-diagonal-model-wedge2}. 
It remains to identify the relevant entries of the transition
matrix. Just as in Propositions~\ref{prop:dominant-term-single-holonomy-block-no-Stokes-directions}
and
\ref{prop:dominant-term-single-holonomy-block-Stokes-directions}, depending on the incoming and outgoing stable
sectors, \(T_i^\eta\) is obtained from \(T^{k_i}\) by multiplying on
the left or on the right by \(U_1^{-1}\) or \(U_1\). From the explicit
form of \(U_1\) in Proposition~\ref{prop:explicit-U}, its second exterior power leaves unchanged the
columns indexed by \(14\) and \(34\), while
\((U_1^{-1})^{[2]}\) leaves unchanged the rows indexed by \(12\) and
\(23\). Since
\[
\mathcal R_i^{[2]}\subset\{12,23\},
\qquad
\mathcal S_i^{[2]}\subset\{14,34\},
\]
we obtain
\[
\bigl((T_i^\eta)^{[2]}\bigr)_{I,J}
=
\bigl((T^{k_i})^{[2]}\bigr)_{I,J}
\]
for every
\(I\in\mathcal R_i^{[2]}\) and
\(J\in\mathcal S_i^{[2]}\).
\end{proof}
We now compute the entries of \((T^k)^{[2]}\) that occur in
Proposition~\ref{prop:dominant-term-single-holonomy-block-wedge2}. We restrict to the \(q=z^d \odif{z}^4\) case. Let \(N=d+4\) and \(\theta=\pi/N\). In Lemma~\ref{lem:relevant-entries-in-terms-of-bk} and Proposition~\ref{prop:nonvanishing-relevant-entries}, we have obtained that 
\[
(T^k)_{2,1}=-b_{k-1},
\qquad
(T^k)_{2,4}=b_k,
\qquad
(T^k)_{3,1}=-b_k,
\qquad
(T^k)_{3,4}=b_{k+1}
\]
for
\[
b_k=\frac{\sin((k-1)\theta)\sin(k\theta)\sin((k+1)\theta)}{\sin(\theta)\sin(2\theta)\sin(3\theta)}.
\]
Set
\[
\Delta_k:=b_k^2-b_{k-1}b_{k+1}.
\]

\begin{lemma}
\label{lem:relevant-minors-of-Tk}
For every \(k\geq2\), the relevant entries of
\((T^k)^{[2]}\) are
\[
\begin{aligned}
  \bigl((T^k)^{[2]}\bigr)_{12,14}
  &=\Delta_{k-1},&
  \bigl((T^k)^{[2]}\bigr)_{12,34}
  &=\Delta_k,\\
  \bigl((T^k)^{[2]}\bigr)_{23,14}
  &=\Delta_k,&
  \bigl((T^k)^{[2]}\bigr)_{23,34}
  &=\Delta_{k+1}.
\end{aligned}
\]
Moreover,
\[
\Delta_k
=
\frac{
\sin((k-1)\theta)\sin^2(k\theta)\sin((k+1)\theta)
}{
\sin\theta\sin^2(2\theta)\sin(3\theta)
}.
\]
\end{lemma}

\begin{proof}
Recalling the explicit form of \(T\) in
\eqref{eq:explicit-value-for-T}, we have
\[
e_1^{\mathsf T}T=e_2^{\mathsf T},
\qquad
e_2^{\mathsf T}T=e_3^{\mathsf T},
\qquad
Te_4=e_3.
\]
Together with
Lemma~\ref{lem:relevant-entries-in-terms-of-bk}, these identities
determine all the entries of \(T^k\) needed below. For example,
\begin{align*}
(T^k)_{1,1}
&=
e_1^{\mathsf T}T^ke_1
=
e_2^{\mathsf T}T^{k-1}e_1
=
(T^{k-1})_{2,1}
=
-b_{k-2},\\
(T^k)_{1,4}
&=
e_1^{\mathsf T}T^ke_4
=
e_2^{\mathsf T}T^{k-1}e_4
=
(T^{k-1})_{2,4}
=
b_{k-1},\\
(T^k)_{1,3}
&=
e_1^{\mathsf T}T^ke_3
=
e_1^{\mathsf T}T^{k+1}e_4
=
(T^{k+1})_{1,4}
=
b_k.
\end{align*}
Similarly, we obtain
\[
\begin{array}{lll}
(T^k)_{1,1}=-b_{k-2},
&
(T^k)_{1,3}=b_k,
&
(T^k)_{1,4}=b_{k-1},
\\
(T^k)_{2,1}=-b_{k-1},
&
(T^k)_{2,3}=b_{k+1},
&
(T^k)_{2,4}=b_k,
\\[1mm]
(T^k)_{3,1}=-b_k,
&
(T^k)_{3,3}=b_{k+2},
&
(T^k)_{3,4}=b_{k+1}.
\end{array}
\]
By formula~\eqref{eq:entries-second-exterior-power}, we obtain
\[
\begin{alignedat}{2}
\bigl((T^k)^{[2]}\bigr)_{12,14}
&= \det\begin{pmatrix} -b_{k-2}&b_{k-1}\\ -b_{k-1}&b_k \end{pmatrix} =\Delta_{k-1},
&&\bigl((T^k)^{[2]}\bigr)_{12,34}
= \det\begin{pmatrix} b_k&b_{k-1}\\ b_{k+1}&b_k \end{pmatrix} =\Delta_k,\\
\bigl((T^k)^{[2]}\bigr)_{23,14}
&= \det\begin{pmatrix} -b_{k-1}&b_k\\ -b_k&b_{k+1} \end{pmatrix} =\Delta_k,
&&\bigl((T^k)^{[2]}\bigr)_{23,34}
= \det\begin{pmatrix} b_{k+1}&b_k\\ b_{k+2}&b_{k+1} \end{pmatrix} =\Delta_{k+1}.
\end{alignedat}
\]
Finally, substituting
\[
b_k
=
\frac{
\sin((k-1)\theta)\sin(k\theta)\sin((k+1)\theta)
}{
\sin\theta\sin(2\theta)\sin(3\theta)
}
\]
and using
\[
\sin((k-1)\theta)\sin((k+1)\theta)
-
\sin((k-2)\theta)\sin((k+2)\theta)
=
\sin\theta\sin(3\theta)
\]
gives the stated formula for \(\Delta_k\).
\end{proof}
Combining with\eqref{eq:frame-change-at-midpoint}, we obtain, with respect to the exterior-power
frame induced by \(\mathcal F_{t,w_1^+}\) at \(m_1\),
\begin{equation}  \label{eq:global-asymptotics-wedge2}
  \bigwedge\nolimits^2
  \operatorname{Hol}_t(\widetilde c_\gamma)
  =
  e^{t^{1/4}
  \sum_{i=1}^{\ell}
  \left(
  \lambda_i^{[2]}+\mu_i^{[2]}
  \right)}
  S^{[2]}\mathcal P_\ell^{(2)}(S^{[2]})^{-1}
  +
  o\left(
  e^{t^{1/4}
  \sum_{i=1}^{\ell}
  \left(
  \lambda_i^{[2]}+\mu_i^{[2]}
  \right)}
  \right),
\end{equation}
where
\begin{equation}
\label{eq:global-dominant-coefficient-wedge2}
\mathcal P_\ell^{(2)}
=
\mathcal C_1^{(2)}(Q^{-2})^{[2]}
\mathcal C_2^{(2)}(Q^{-2})^{[2]}
\cdots
\mathcal C_\ell^{(2)}(Q^{-2})^{[2]}.
\end{equation}
\begin{proposition}
\label{prop:non-nilpotence-global-dominant-coefficient-wedge2}
The matrix \(\mathcal P_\ell^{(2)}\) is not nilpotent.
\end{proposition}

\begin{proof}
Proposition~\ref{prop:dominant-term-single-holonomy-block-wedge2} and Lemma~\ref{lem:relevant-minors-of-Tk} imply that
\(
\bigl(\mathcal C_i^{[2]}\bigr)_{23,14}
=
\Delta_{k_i}>0.
\)
Recall that
\[
Q^{-2}
=
\begin{pmatrix}
0&0&1&0\\
0&0&0&1\\
-1&0&0&0\\
0&-1&0&0
\end{pmatrix}.
\]
Namely,
\[
Q^{-2}e_1=-e_3,
\qquad
Q^{-2}e_2=-e_4,
\qquad
Q^{-2}e_3=e_1,
\qquad
Q^{-2}e_4=e_2.
\]
It follows that
\[
\begin{aligned}
(Q^{-2})^{[2]}e_{12}&=e_{34},&
(Q^{-2})^{[2]}e_{13}&=e_{13},\\
(Q^{-2})^{[2]}e_{14}&=e_{23},&
(Q^{-2})^{[2]}e_{23}&=e_{14},\\
(Q^{-2})^{[2]}e_{24}&=e_{24},&
(Q^{-2})^{[2]}e_{34}&=e_{12}.
\end{aligned}
\]
Then 
\[
\left(
\mathcal C_i^{(2)}(Q^{-2})^{[2]}
\right)_{23,23}
=
\Delta_{k_i}>0.
\]
and all the relevant entries are nonnegative, and therefore
\[
\bigl(\mathcal P_\ell^{(2)}\bigr)_{23,23}
\geq
\prod_{i=1}^{\ell}\Delta_{k_i}>0.
\]
Hence \(\mathcal P_\ell^{(2)}\) is not nilpotent.
\end{proof}

\subsection{Main results}
\label{subsec:main-results}

We now collect the preceding estimates and state the main asymptotic
formulas. For each saddle connection \(c_i\), let
\[
v_i
=
\bigl(
v_{i,1},v_{i,2},v_{i,3},v_{i,4}
\bigr),
\qquad
v_{i,1}\geq v_{i,2}\geq v_{i,3}\geq v_{i,4},
\]
be the non-increasing rearrangement of
\[
-2\operatorname{Re}\int_{c_i}\phi_1,\qquad
-2\operatorname{Re}\int_{c_i}\phi_2,\qquad
-2\operatorname{Re}\int_{c_i}\phi_3,\qquad
-2\operatorname{Re}\int_{c_i}\phi_4.
\]
Note that 
\[
v_{i,4}=-v_{i,1},
\qquad
v_{i,3}=-v_{i,2}.
\]
Hence
\[
\sum_{j=1}^3v_{i,j}=v_{i,1},
\qquad
\sum_{j=1}^4v_{i,j}=0.
\]
\begin{theorem}
\label{thm:exterior-power-holonomy-asymptotics-norm}
Let \(c_\gamma\) be the geodesic representative of \(\gamma\) with respect to the singular flat metric induced by \(q\). Then, for every \(k=1,2,3,4\),
\[
\lim_{t\to+\infty}
\frac{1}{t^{1/4}}
\log
\left\|
\bigwedge\nolimits^k
\operatorname{Hol}_t(c_\gamma)
\right\|
=
\sum_{i=1}^{\ell}
\sum_{j=1}^{k}v_{i,j}.
\]
where \(\lVert\cdot\rVert\) is any fixed matrix norm.
\end{theorem}

\begin{proof}
By equivalence of norms, it suffices to consider the operator norm. 

For \(k=1,2\), set
\[
S^{[1]}:=S,
\qquad
\mathcal P_\ell^{(1)}:=\mathcal P_\ell,
\]
and retain the notation \(S^{[2]}\) and
\(\mathcal P_\ell^{(2)}\) introduced in Section~\ref{subsec:second-exterior-power}. Define
\[
L_k
:=
\sum_{i=1}^{\ell}
\left(
\lambda_i^{[k]}+\mu_i^{[k]}
\right).
\]
By the asymptotic formulas~\eqref{eq:asymptotics-composition-all-holonomy-blocks} and \eqref{eq:global-asymptotics-wedge2}, we have 
\[
\bigwedge\nolimits^k
\operatorname{Hol}_t(\widetilde c_\gamma)
=
e^{t^{1/4}L_k}
\left(S^{[k]}
\mathcal P_\ell^{(k)}(S^{[k]})^{-1}+o(I)
\right),
\qquad
k=1,2.
\]
It follows that
\[
\log
\left\|
\bigwedge\nolimits^k
\operatorname{Hol}_t(\widetilde c_\gamma)
\right\|_{\mathrm{op}}
=t^{1/4}L_k+
\log \left\|
S^{[k]}\mathcal P_\ell^{(k)}(S^{[k]})^{-1}+o(I)
\right\|_{\mathrm{op}}.
\]
By
Propositions~\ref{prop:non-nilpotence-global-dominant-coefficient}
and
\ref{prop:non-nilpotence-global-dominant-coefficient-wedge2},
the matrices
\(\mathcal P_\ell^{(1)}\) and
\(\mathcal P_\ell^{(2)}\) are not nilpotent. In particular, they are
nonzero. Hence
\[
S^{[k]}\mathcal P_\ell^{(k)}(S^{[k]})^{-1}\neq0,
\qquad
k=1,2.
\]
Namely, as \(t \to +\infty\),
\[
\left\|
S^{[k]}\mathcal P_\ell^{(k)}(S^{[k]})^{-1}+o(I)
\right\|_{\mathrm{op}} \to \left\|
S^{[k]}\mathcal P_\ell^{(k)}(S^{[k]})^{-1}
\right\|_{\mathrm{op}} >0
\]
Therefore,
\[
\lim_{t\to+\infty}
\frac{1}{t^{1/4}}
\log
\left\|
\bigwedge\nolimits^k
\operatorname{Hol}_t(\widetilde c_\gamma)
\right\|_{\mathrm{op}}
=
L_k,
\qquad
k=1,2.
\]
Recall that
\(
c_i=\delta_i*\alpha_{i-1},
\)
with indices understood cyclically. We have
\[
\sum_{j=1}^k v_{i,j}
=
\lambda_i^{[k]}+\mu_{i-1}^{[k]},
\qquad
k=1,2.
\]
Summing over \(i\), we obtain
\begin{align*}
L_k
=
\sum_{i=1}^{\ell}
\left(
\lambda_i^{[k]}+\mu_i^{[k]}
\right)
=
\sum_{i=1}^{\ell}
\left(
\lambda_i^{[k]}+\mu_{i-1}^{[k]}
\right)
=
\sum_{i=1}^{\ell}\sum_{j=1}^k v_{i,j}.
\end{align*}
This proves the formula for \(k=1,2\).

The cases \(k=3\) and \(k=4\) now follow from the duality.
Note that for a symplectic matrix \(M\), its singular values occur in reciprocal
pairs:
\[
\sigma_4(M)=\sigma_1(M)^{-1},
\qquad
\sigma_3(M)=\sigma_2(M)^{-1}.
\]
Hence
\[
\left\|\bigwedge\nolimits^3M\right\|_{\mathrm{op}}
=
\sigma_1(M)\sigma_2(M)\sigma_3(M)
=
\sigma_1(M)
=
\|M\|_{\mathrm{op}},
\]
and
\[
\left\|\bigwedge\nolimits^4M\right\|_{\mathrm{op}}
=
\sigma_1(M)\sigma_2(M)\sigma_3(M)\sigma_4(M)
=
1.
\]
Together with
\[
\sum_{j=1}^3v_{i,j}=v_{i,1},
\qquad
\sum_{j=1}^4v_{i,j}=0,
\]
this proves the cases \(k=3\) and \(k=4\). This completes the proof.
\end{proof}
The preceding theorem immediately determines the asymptotic behavior
of each singular value. Indeed, the product of the first
\(k\) singular values is the operator norm of the \(k\)-th exterior
power. Taking the quotient of two consecutive such products gives the
following corollary.
\begin{corollary}
\label{cor:singular-value-asymptotics}
For every
\(k=1,2,3,4\), we have
\[
\lim_{t\to+\infty}
\frac{1}{t^{1/4}}
\log
\sigma_k\left(
\operatorname{Hol}_t(c_\gamma)
\right)
=
\sum_{i=1}^{\ell}v_{i,k}.
\]
\end{corollary}
We next consider the spectral radius \(\Lambda\). As in the singular-value case,
the symplectic structure reduces the proof to the first two exterior
powers. Indeed, the eigenvalues of a symplectic matrix occur in
reciprocal pairs. Consequently, for
\(M\in\operatorname{Sp}(4,\mathbb R)\),
\[
\Lambda\left(\bigwedge\nolimits^3M\right)
=
\Lambda(M),
\qquad
\Lambda\left(\bigwedge\nolimits^4M\right)
=
1.
\]
Thus only the cases \(k=1,2\) require an independent argument.
\begin{theorem}
\label{thm:exterior-power-spectral-radius-asymptotics}
For every \(k=1,2,3,4\),
\[
\lim_{t\to+\infty}
\frac{1}{t^{1/4}}
\log
\Lambda\left(
\bigwedge\nolimits^k
\operatorname{Hol}_t(c_\gamma)
\right)
=
\sum_{i=1}^{\ell}\sum_{j=1}^k v_{i,j}.
\]
\end{theorem}

\begin{proof}
  Retain the notation introduced in proof of Theorem~\ref{thm:exterior-power-holonomy-asymptotics-norm}.
  The asymptotic formulas~\eqref{eq:asymptotics-composition-all-holonomy-blocks} and \eqref{eq:global-asymptotics-wedge2} imply that 
\[
\bigwedge\nolimits^k
\operatorname{Hol}_t(\widetilde c_\gamma)
=
e^{t^{1/4}L_k}
\left(S^{[k]}
\mathcal P_\ell^{(k)}(S^{[k]})^{-1}+o(I)
\right),
\qquad
k=1,2.
\]
By
Propositions~\ref{prop:non-nilpotence-global-dominant-coefficient}
and
\ref{prop:non-nilpotence-global-dominant-coefficient-wedge2},
the matrix \(\mathcal P_\ell^{(k)}\) is not nilpotent. Hence
\[
\Lambda\left(S^{[k]}
\mathcal P_\ell^{(k)}(S^{[k]})^{-1}\right)=\Lambda\left(\mathcal P_\ell^{(k)}\right)>0.
\]
By continuity of the spectral radius, as \(t \to +\infty\),
\[
\Lambda\left(S^{[k]}\mathcal P_\ell^{(k)}(S^{[k]})^{-1}+o(I)\right) \to \Lambda\left(S^{[k]}\mathcal P_\ell^{(k)}(S^{[k]})^{-1}\right) >0
\]
It follows that
\[
\lim_{t\to+\infty}
\frac{1}{t^{1/4}}
\log
\Lambda \left(
\bigwedge\nolimits^k
\operatorname{Hol}_t(\widetilde c_\gamma)
\right)
=
L_k
=
\sum_{i=1}^{\ell}\sum_{j=1}^k v_{i,j},
\qquad
k=1,2.
\]
\end{proof}
Using the same argument as in
Corollary~\ref{cor:singular-value-asymptotics}, we obtain the
following consequence.
\begin{corollary}
\label{cor:eigenvalue-modulus-asymptotics}
Let
\[
|\lambda_1(M)|\geq
|\lambda_2(M)|\geq
|\lambda_3(M)|\geq
|\lambda_4(M)|>0
\]
denote the absolute values of eigenvalues of \(M \in \mathrm{SP}(4,\mathbb{R})\). Then, for every
\(k=1,2,3,4\),
\[
\lim_{t\to+\infty}
\frac{1}{t^{1/4}}
\log
|\lambda_k(M)|
=
\sum_{i=1}^{\ell}v_{i,k}.
\]
\end{corollary}
\section*{Declarations} \subsection*{Conflict of Interest} The author declares that there is no conflict of interest. \subsection*{Data Availability} Data sharing is not applicable to this article, as no datasets were generated or analyzed. \subsection*{Declaration of Generative AI and AI-Assisted Technologies} During the preparation of this paper, the author used ChatGPT Pro (GPT-5.6 Thinking, OpenAI) and DeepSeek-V3.1 (DeepSeek) to improve the grammar and readability of the text and to assist in checking calculations. All AI-assisted suggestions and outputs were independently reviewed and verified by the author, who takes full responsibility for the content of the paper.
\printbibliography[heading=bibintoc]

@misc{garciaprada2012hitchinkobayashi,
  title         = {{The Hitchin--Kobayashi correspondence, Higgs pairs and surface group representations}},
  author        = {Garc{\'\i}a-Prada, Oscar and Gothen, Peter B. and Mundet i Riera, Ignasi},
  year          = {2012},
  doi           = {10.48550/arXiv.0909.4487}
}

@article{li2023higgs,
  author  = {Li, Qiongling and Mochizuki, Takuro},
  title   = {Higgs bundles in the Hitchin section over non-compact hyperbolic surfaces},
  journal = {Proceedings of the London Mathematical Society},
  volume  = {129},
  number  = {6},
  pages   = {e70008},
  doi     = {10.1112/plms.70008},
  year    = {2024}
}

@article{hitchin1992lie,
  title     = {{Lie groups and Teichm{\"u}ller space}},
  author    = {Hitchin, Nigel J},
  journal   = {Topology},
  volume    = {31},
  number    = {3},
  pages     = {449--473},
  year      = {1992},
  doi       = {10.1016/0040-9383(92)90044-I},
  publisher = {Elsevier}
}

@article{hitchin1987self,
  author  = {Hitchin, N. J.},
  title   = {The Self-Duality Equations on a Riemann Surface},
  journal = {Proceedings of the London Mathematical Society},
  volume  = {s3-55},
  number  = {1},
  pages   = {59-126},
  doi     = {10.1112/plms/s3-55.1.59},
  year    = {1987}
}

@article{simpson1988constructing,
  title   = {{Constructing variations of Hodge structure using Yang-Mills theory and applications to uniformization}},
  author  = {Simpson, Carlos T},
  journal = {Journal of the American Mathematical Society},
  volume  = {1},
  number  = {4},
  pages   = {867--918},
  doi     = {10.1090/S0894-0347-1988-0944577-9},
  year    = {1988}
}

@article{simpson1990harmonic,
  title   = {{Harmonic bundles on noncompact curves}},
  author  = {Simpson, Carlos T},
  journal = {Journal of the American Mathematical Society},
  volume  = {3},
  number  = {3},
  pages   = {713--770},
  doi     = {10.1090/S0894-0347-1990-1040197-8},
  year    = {1990}
}

@article{biquard2004wild,
  title        = {{Wild non-abelian Hodge theory on curves}},
  author       = {Biquard, Olivier and Boalch, Philip},
  journal      = {Compositio Mathematica},
  volume       = {140},
  number       = {1},
  pages        = {179--204},
  year         = {2004},
  doi          = {10.1112/S0010437X03000010},
  publisher    = {London Mathematical Society}
}

@article{biquard1997higgs,
  author    = {Biquard, Olivier},
  title     = {Fibr{\'e}s de {Higgs} et connexions int{\'e}grables :
               le cas logarithmique (diviseur lisse)},
  journal   = {Annales Scientifiques de l'{\'E}cole Normale Sup{\'e}rieure},
  volume    = {30},
  number    = {1},
  pages     = {41--96},
  year      = {1997},
  publisher = {Elsevier},
  doi       = {10.1016/S0012-9593(97)89915-6}
}

@book{mochizuki2006kobayashi,
  author    = {Mochizuki, Takuro},
  title     = {Kobayashi--Hitchin Correspondence for Tame Harmonic
               Bundles and an Application},
  series    = {Ast{\'e}risque},
  number    = {309},
  pages     = {viii+117},
  year      = {2006},
  publisher = {Soci{\'e}t{\'e} Math{\'e}matique de France},
  address   = {Paris}
}

@article{mochizuki2009kobayashi,
  author    = {Mochizuki, Takuro},
  title     = {Kobayashi--Hitchin Correspondence for Tame Harmonic
               Bundles {II}},
  journal   = {Geometry \& Topology},
  volume    = {13},
  number    = {1},
  pages     = {359--455},
  year      = {2009},
  publisher = {Mathematical Sciences Publishers},
  doi       = {10.2140/gt.2009.13.359}
}

@article{Mochizuki2021good,
  author       = {Takuro Mochizuki},
  title        = {Good wild harmonic bundles and good filtered Higgs bundles},
  journal      = {Symmetry, Integrability and Geometry: Methods and Applications (SIGMA)},
  volume       = {17},
  year         = {2021},
  pages        = {068},
  note         = {66 pages},
  doi          = {10.3842/SIGMA.2021.068}
}

@article{Li2023generically,
  title     = {{Harmonic metrics of generically regular semisimple Higgs bundles on noncompact Riemann surfaces}},
  volume    = {5},
  issn      = {2576-7658},
  doi       = {10.2140/tunis.2023.5.663},
  number    = {4},
  journal   = {Tunisian Journal of Mathematics},
  publisher = {Mathematical Sciences Publishers},
  author    = {Li, Qiongling and Mochizuki, Takuro},
  year      = {2023},
  month     = nov,
  pages     = {663-711}
}

@phdthesis{baraglia2010g2,
  title         = {{$G_2$ geometry and integrable systems}},
  author        = {Baraglia, David},
  year          = {2009},
  school        = {University of Oxford},
  eprint        = {1002.1767},
  archiveprefix = {arXiv}
}

@article{Sirakov2009,
  author  = {Boyan Sirakov},
  title   = {Some estimates and maximum principles for weakly coupled systems of elliptic PDE},
  journal = {Nonlinear Analysis},
  volume  = {70},
  number  = {8},
  pages   = {3039--3046},
  year    = {2009},
  doi     = {10.1016/j.na.2008.12.026}
}

@article{donaldson1987twisted,
  author  = {Donaldson, S. K.},
  title   = {Twisted Harmonic Maps and the Self-Duality Equations},
  journal = {Proceedings of the London Mathematical Society},
  volume  = {s3-55},
  number  = {1},
  pages   = {127-131},
  doi     = {10.1112/plms/s3-55.1.127},
  year    = {1987}
}

@article{corlette1988flat,
  title     = {{Flat $ G $-bundles with canonical metrics}},
  author    = {Corlette, Kevin},
  journal   = {Journal of differential geometry},
  volume    = {28},
  number    = {3},
  pages     = {361--382},
  year      = {1988},
  publisher = {Lehigh University},
  doi       = {10.4310/jdg/1214442469}
}

@article{Mochizuki_2016,
  title   = {{Asymptotic behaviour of certain families of harmonic bundles on Riemann surfaces}},
  author  = {Mochizuki, Takuro},
  journal = {Journal of Topology},
  volume  = {9},
  number  = {4},
  pages   = {1021--1073},
  year    = {2016},
  doi     = {10.1112/jtopol/jtw018}
}

@book{mochizuki2010wild,
  author    = {Mochizuki, Takuro},
  title     = {{Wild harmonic bundles and wild pure twistor $D$-modules}},
  series    = {Ast\'erisque},
  publisher = {Soci\'et\'e math\'ematique de France},
  number    = {340},
  year      = {2011},
  mrnumber  = {2919903},
  zbl       = {1245.32001},
  language  = {en},
  url       = {https://arxiv.org/abs/0803.1344}
}

@article{biquard2020parabolic,
  title         = {{Parabolic Higgs bundles and representations of the fundamental group of a punctured surface into a real group}},
  author        = {Biquard, Olivier and Garc{\'\i}a-Prada, Oscar and Mundet i Riera, Ignasi},
  journal       = {Advances in Mathematics},
  volume        = {372},
  pages         = {107305},
  year          = {2020},
  publisher     = {Elsevier},
  doi           = {10.1016/j.aim.2020.107305}
}

@article{li2025complete,
  title     = {{Complete Solutions of Toda Equations and Cyclic Higgs Bundles Over Non-compact Surfaces}},
  author    = {Li, Qiongling and Mochizuki, Takuro},
  journal   = {International Mathematics Research Notices},
  volume    = {2025},
  number    = {7},
  pages     = {rnaf081},
  year      = {2025},
  doi       = {10.1093/imrn/rnaf081},
  publisher = {Oxford University Press}
}

@article{labourie2004,
  title   = {{Anosov flows, surface groups and curves in projective space}},
  author  = {Labourie, Fran{\c{c}}ois},
  journal = {Inventiones Mathematicae},
  volume  = {165},
  number  = {1},
  pages   = {51--114},
  year    = {2006},
  doi     = {10.1007/s00222-005-0487-3}
}

@article{tamburelli2024planar,
  title     = {{Planar minimal surfaces with polynomial growth in the $\mathrm{Sp}(4,\mathbb{R})$-symmetric space}},
  author    = {Tamburelli, Andrea and Wolf, Michael},
  journal   = {American Journal of Mathematics},
  volume    = {146},
  number    = {4},
  pages     = {871--944},
  year      = {2024},
  doi       = {10.1353/ajm.2024.a932432},
  publisher = {Johns Hopkins University Press}
}

@article{guestlin2026110730,
  title   = {{The tt*-Toda equations of $A_n$ type}},
  author  = {Guest, Martin A. and Its, Alexander R. and Lin, Chang-Shou},
  journal = {Advances in Mathematics},
  volume  = {485},
  pages   = {110730},
  year    = {2026},
  issn    = {0001-8708},
  doi     = {10.1016/j.aim.2025.110730}
}

@article{collier2017asymptotics,
  title     = {Asymptotics of Higgs bundles in the Hitchin component},
  author    = {Collier, Brian and Li, Qiongling},
  journal   = {Advances in Mathematics},
  volume    = {307},
  pages     = {488--558},
  year      = {2017},
  publisher = {Elsevier},
  doi       = {10.1016/j.aim.2016.11.031}
}

@article{dumas2015polynomial,
  title     = {Polynomial cubic differentials and convex polygons in the projective plane},
  author    = {Dumas, David and Wolf, Michael},
  journal   = {Geometric and Functional Analysis},
  volume    = {25},
  number    = {6},
  pages     = {1734--1798},
  year      = {2015},
  publisher = {Springer},
  doi       = {10.1007/s00039-015-0344-5}
}

@article{mochizuki2014harmonic,
  title     = {Harmonic bundles and Toda lattices with opposite sign II},
  author    = {Mochizuki, Takuro},
  journal   = {Communications in Mathematical Physics},
  volume    = {328},
  number    = {3},
  pages     = {1159--1198},
  year      = {2014},
  publisher = {Springer},
  doi       = {10.1007/s00220-014-1994-0}
}

@book{kac2002quantum,
  title     = {{Quantum Calculus}},
  author    = {Kac, Victor G. and Cheung, Pokman},
  series    = {Universitext},
  year      = {2002},
  publisher = {Springer New York},
  doi       = {10.1007/978-1-4613-0071-7}
}

@article{loftin2026limits,
  title     = {{Limits of cubic differentials and buildings}},
  author    = {Loftin, John and Tamburelli, Andrea and Wolf, Michael},
  journal   = {Proceedings of the London Mathematical Society},
  volume    = {132},
  number    = {4},
  pages     = {e70154},
  year      = {2026},
  doi       = {10.1112/plms.70154},
  publisher = {Wiley}
}

@article{labourie2017cyclic,
  title     = {Cyclic surfaces and Hitchin components in rank 2},
  author    = {Labourie, Fran{\c{c}}ois},
  journal   = {Annals of Mathematics},
  volume    = {185},
  number    = {1},
  pages     = {1--58},
  year      = {2017},
  publisher = {Department of Mathematics, Princeton University Princeton, New Jersey, USA},
  doi       = {10.4007/annals.2017.185.1.1}
}

@article{loftin2001affine,
  title     = {{Affine spheres and convex $\mathbb{RP}^n$-manifolds}},
  author    = {Loftin, John C.},
  journal   = {American Journal of Mathematics},
  volume    = {123},
  number    = {2},
  pages     = {255--274},
  year      = {2001},
  doi       = {10.1353/ajm.2001.0011},
  publisher = {Johns Hopkins University Press}
}

@article{labourie2007flat,
  title     = {Flat projective structures on surfaces and cubic holomorphic differentials},
  author    = {Labourie, Fran{\c{c}}ois},
  journal   = {Pure and Applied Mathematics Quarterly},
  volume    = {3},
  number    = {4},
  pages     = {1057--1099},
  year      = {2007},
  publisher = {International Press of Boston},
  doi       = {10.4310/PAMQ.2007.v3.n4.a10}
}

@article{reid2026limits,
  title     = {{Limits of Convex Projective Surfaces and Finsler Metrics}},
  author    = {Reid, Charles},
  journal   = {The Journal of Geometric Analysis},
  volume    = {36},
  number    = {7},
  pages     = {246},
  year      = {2026},
  doi       = {10.1007/s12220-026-02486-x},
  publisher = {Springer}
}

@article{sagman2026local,
  title     = {{Local asymptotics for Hitchin's equations and high energy harmonic maps}},
  author    = {Sagman, Nathaniel and Smillie, Peter},
  journal   = {Mathematische Annalen},
  volume    = {394},
  number    = {1},
  pages     = {17},
  year      = {2026},
  doi       = {10.1007/s00208-026-03375-y},
  publisher = {Springer}
}

@article{parreau2012compactification,
  author    = {Parreau, Anne},
  title     = {Compactification d'espaces de représentations de groupes de type fini},
  journal   = {Mathematische Zeitschrift},
  volume    = {272},
  number    = {1-2},
  pages     = {51--86},
  year      = {2012},
  publisher = {Springer},
  doi       = {10.1007/s00209-011-0921-8}
}

@article{parreau2022invariant,
  author    = {Parreau, Anne},
  title     = {Invariant Weakly Convex Cocompact Subspaces for Surface
               Groups in {$A_2$}-Buildings},
  journal   = {Transactions of the American Mathematical Society},
  volume    = {375},
  number    = {4},
  pages     = {2293--2339},
  year      = {2022},
  publisher = {American Mathematical Society},
  doi       = {10.1090/tran/8344}
}

@article{burger2021currents,
  author    = {Burger, Marc and Iozzi, Alessandra and Parreau, Anne and Pozzetti, Maria Beatrice},
  title     = {Currents, Systoles, and Compactifications of Character Varieties},
  journal   = {Proceedings of the London Mathematical Society},
  volume    = {123},
  number    = {6},
  pages     = {565--596},
  year      = {2021},
  publisher = {Wiley},
  doi       = {10.1112/plms.12419}
}

@article{burger2021real-spectrum,
  author    = {Burger, Marc and Iozzi, Alessandra and Parreau, Anne and Pozzetti, Maria Beatrice},
  title     = {The Real Spectrum Compactification of Character Varieties: Characterizations and Applications},
  journal   = {Comptes Rendus. Math{\'e}matique},
  volume    = {359},
  number    = {4},
  pages     = {439--463},
  year      = {2021},
  publisher = {Acad{\'e}mie des sciences, Paris},
  doi       = {10.5802/crmath.123}
}

@misc{burger2023real-spectrum,
  author        = {Burger, Marc and Iozzi, Alessandra and Parreau, Anne
                   and Pozzetti, Maria Beatrice},
  title         = {The Real Spectrum Compactification of Character Varieties},
  year          = {2023},
  doi           = {10.48550/arXiv.2311.01892},
  note          = {Version 2, revised 14 July 2025}
}

@article{ouyang2021limits,
  author    = {Ouyang, Charles and Tamburelli, Andrea},
  title     = {Limits of {Blaschke} Metrics},
  journal   = {Duke Mathematical Journal},
  volume    = {170},
  number    = {8},
  pages     = {1683--1722},
  year      = {2021},
  publisher = {Duke University Press},
  doi       = {10.1215/00127094-2021-0027}
}

@article{ouyang2023length,
  author    = {Ouyang, Charles and Tamburelli, Andrea},
  title     = {Length Spectrum Compactification of the
               {$\mathrm{SO}_0(2,3)$}-{Hitchin} Component},
  journal   = {Advances in Mathematics},
  volume    = {420},
  pages     = {108997},
  year      = {2023},
  publisher = {Elsevier},
  doi       = {10.1016/j.aim.2023.108997}
}

@article{ouyang2023high-energy,
  author    = {Ouyang, Charles},
  title     = {High-Energy Harmonic Maps and Degeneration of Minimal Surfaces},
  journal   = {Geometry \& Topology},
  volume    = {27},
  number    = {5},
  pages     = {1691--1746},
  year      = {2023},
  publisher = {Mathematical Sciences Publishers},
  doi       = {10.2140/gt.2023.27.1691}
}

@article{katzarkov2015harmonic,
  author    = {Katzarkov, Ludmil and Noll, Alexander and Pandit, Pranav and Simpson, Carlos},
  title     = {Harmonic Maps to Buildings and Singular Perturbation Theory},
  journal   = {Communications in Mathematical Physics},
  volume    = {336},
  number    = {2},
  pages     = {853--903},
  year      = {2015},
  publisher = {Springer},
  doi       = {10.1007/s00220-014-2276-6}
}

@article{martone2024closed,
  author    = {Martone, Giuseppe and Ouyang, Charles and Tamburelli, Andrea},
  title     = {A Closed Ball Compactification of a Maximal Component via Cores of Trees},
  journal   = {Algebraic \& Geometric Topology},
  volume    = {24},
  number    = {7},
  pages     = {3693--3717},
  year      = {2024},
  publisher = {Mathematical Sciences Publishers},
  doi       = {10.2140/agt.2024.24.3693}
}

@article{wolf1989teichmuller,
  author    = {Wolf, Michael},
  title     = {The {Teichm{\"u}ller} Theory of Harmonic Maps},
  journal   = {Journal of Differential Geometry},
  volume    = {29},
  number    = {2},
  pages     = {449--479},
  year      = {1989},
  publisher = {International Press},
  doi       = {10.4310/jdg/1214442885}
}

@article{wolf1995rtrees,
  author    = {Wolf, Michael},
  title     = {Harmonic Maps from Surfaces to {$\mathbb{R}$}-Trees},
  journal   = {Mathematische Zeitschrift},
  volume    = {218},
  number    = {4},
  pages     = {577--593},
  year      = {1995},
  publisher = {Springer},
  doi       = {10.1007/BF02571924}
}

@incollection{katzarkov2017constructing,
  author    = {Katzarkov, Ludmil and Noll, Alexander and Pandit, Pranav and Simpson, Carlos},
  title     = {Constructing Buildings and Harmonic Maps},
  booktitle = {Algebra, Geometry, and Physics in the 21st Century:
               Kontsevich Festschrift},
  editor    = {Auroux, Denis and Katzarkov, Ludmil and Pantev, Tony
               and Soibelman, Yan and Tschinkel, Yuri},
  series    = {Progress in Mathematics},
  volume    = {324},
  pages     = {203--260},
  year      = {2017},
  publisher = {Birkh{\"a}user},
  address   = {Cham},
  doi       = {10.1007/978-3-319-59939-7_6}
}

@misc{evans2024polynomial,
  author        = {Evans, Parker},
  title         = {Polynomial Almost-Complex Curves in
                   {$\widehat{\mathbb S}^{2,4}$}},
  year          = {2024},
  eprint        = {2208.14409},
  archiveprefix = {aXiv}
}
\end{document}